\documentclass[12pt,a4paper,reqno]{amsart}
\usepackage{graphicx} % Required for inserting images
\usepackage{xcolor}
\usepackage{amssymb,amsmath,amsthm,mathtools,xspace,comment}

\makeindex
\usepackage{enumitem}
\usepackage{mathrsfs,calligra}
\usepackage{here,float}
\usepackage[all]{xy}
\usepackage{geometry}
\usepackage{hyperref}

\makeatletter
\def\input@path{{./}{../}{/Users/kyo/home/math/mypaper/}}
\graphicspath{{./figures/}}
\makeatother

\newcommand{\version}{Ver.~0.0}
\newcommand{\setversion}[1]{\renewcommand{\version}{Ver.~{#1}}}
\setversion{0.1 [2024/04/24 18:17:01 JST]}
\setversion{0.2 [2024/12/31 18:34:10 JST]}
\setversion{0.3 [2025/02/28 10:29:08 JST]}
\setversion{0.7 [2025/03/12 12:53:25 JST]}
\setversion{1.0 [2025/06/15 10:38:43 JST]}
\setversion{1.1 [2025/06/29 13:47 JST]}
\setversion{1.11 [2026/04/22 17:45 JST]}
\setversion{1.2 [2026/05/17 23:31:17 JST]}
\setversion{1.21 [2026/05/18 12:08:51 JST]}
\setversion{1.30 [2026/06/23 18:28:59 JST]}
\setversion{1.40 [2026/10/24 18:54:59 JST]}

\numberwithin{equation}{section}

\newtheorem{theorem}{Theorem}[section]
\newtheorem{lemma}[theorem]{Lemma}
\newtheorem{proposition}[theorem]{Proposition}
\newtheorem{corollary}[theorem]{Corollary}
\newtheorem{fact}[theorem]{Fact}

\theoremstyle{definition}
\newtheorem{example}[theorem]{Example}
\newtheorem{notation}[theorem]{Notation}
\newtheorem{definition}[theorem]{Definition}
\newtheorem{remark}[theorem]{Remark}

\newtheorem{convention}[theorem]{Convention}

\newtheorem{Setting}[theorem]{Setting}

\newcounter{penum}
\newenvironment{penumerate}{%
\par\smallskip
\begin{list}{$\;\;(\thepenum)$}{%
\usecounter{penum}
\setlength{\partopsep}{0pt}
\setlength{\parsep}{1ex}
\setlength{\itemindent}{0pt}
\setlength{\labelsep}{.5em}
\setlength{\labelwidth}{0pt}
\setlength{\leftmargin}{0pt}
\setlength{\rightmargin}{0pt}
\setlength{\itemsep}{0pt}
}}
{\end{list}\par}
\newcommand{\skipover}[1]{}

\newcommand{\smallvstrut}{\vphantom{\Bigm|}}

\newcommand{\Z}{\mathbb{Z}}
\newcommand{\Q}{\mathbb{Q}}
\newcommand{\R}{\mathbb{R}}

\newcommand{\C}{\mathbb{C}}

\newcommand{\Xfv}{\mathfrak{X}}

\newcommand{\Lie}{\mathrm{Lie}}

\newcommand{\GL}{\mathrm{GL}}
\newcommand{\U}{\mathrm{U}}

\newcommand{\Sp}{\mathrm{Sp}}

\newcommand{\OO}{\mathrm{O}}

\newcommand{\Mat}{\mathrm{M}}
\newcommand{\Gal}{\mathrm{Gal}}

\newcommand{\diag}{\qopname\relax o{diag}}

\newcommand{\Stab}{\qopname\relax o{Stab}}

\newcommand{\unitmatrix}{\boldsymbol{1}}

\newcommand{\leftaction}{\mathrel{\raisebox{.4em}[0pt][0pt]{$\curvearrowright$}}}

\newcommand{\sgn}{\qopname\relax o{sgn}}

\newcommand{\vectwo}[2]{{\renewcommand{\arraystretch}{.85}\Bigl(\begin{array}{@{\,}c@{\,}}{#1}\\ {#2}\end{array}\Bigr)}}

\newcommand{\dblFV}{\Xfv}

\newcommand{\lie}[1]{\mathfrak{#1}}

\newcommand{\bborbit}{\mathbb{O}}
\newcommand{\calorbit}{\mathcal{O}}

\newcommand{\scrorbitR}{\mathscr{O}^{\R}}
\newcommand{\scrorbitC}{\mathscr{O}^{\C}}

\newcommand{\transpose}[1]{\,{}^t{#1}}
\newcommand{\conjugate}[1]{\overline{\rule{0pt}{1.2ex} #1}}
\newcommand{\LL}{\langle}
\newcommand{\RR}{\rangle}

\newcommand{\nopicture}[1]{}

\newcommand{\fg}{\mathfrak{g}}

\newcommand{\gl}{\mathfrak{gl}}
\newcommand{\fb}{\mathfrak{b}}

\newcommand{\rank}{\qopname\relax o{rank}}

\newcommand{\sign}{\qopname\relax o{sign}}

\newcommand{\SPI}{\mathrm{SPI}}
\newcommand{\SPIst}{\mathrm{SPI}^{+}}
\newcommand{\SPP}{\mathrm{SPP}}
\newcommand{\PP}{\mathrm{PP}}
\newcommand{\PI}{\mathrm{PI}}
\newcommand{\regOmega}{\Omega^{\circ}}
\newcommand{\regT}{\mathfrak{T}^{\circ}}
\newcommand{\regR}{\mathfrak{R}^{\circ}}

\newcommand{\ee}{\varepsilon}

\newcommand{\LGr}{\qopname\relax o{LGr}}

\newcommand{\HLGr}{\qopname\relax o{HLGr}}
\newcommand{\regMat}{\mathrm{M}^{\circ}}

\newcommand{\LregMat}{\mathrm{LM}^{\circ}}

\newcommand{\Sym}{\qopname\relax o{Sym}}

\newcommand{\Her}{\qopname\relax o{Her}}

\newcommand{\Sgroup}{\mathfrak{S}}

\newcommand{\Ogroup}{\mathrm{O}}

\newcommand{\caseA}{{\upshape(A)}\xspace}
\newcommand{\caseB}{{\upshape(B)}\xspace}

\newcommand{\Gr}{\mathrm{Gr}}

\newcommand{\ech}{\operatorname{\acute{e}ch}}

\newcommand{\MIP}{$ (\dagger) $\;\;}
\newcommand{\maru}{\circ}

\title[Real double flag varieties for the Siegel parabolic subgroups]
{Orbit structures on real double flag varieties for the Siegel parabolic subgroups
}
\author[K.\,Nishiyama and T.\,Tauchi]{Kyo Nishiyama$^{1}$ and Taito Tauchi$^{2}$}

\address{$^{1,2}$\; Department of Mathematics, Aoyama Gakuin University, Sagamihara, Japan}

\email{kyo.nishiyama@gmail.com$^1$,
tauchi@math.aoyama.ac.jp$^2$}



\begin{document}

\begin{abstract}
Let $ G $ be a connected reductive algebraic group over $ \R $, and $ H $ its symmetric subgroup.  
For parabolic subgroups $ P_{G} \subset G $ and $ P_{H} \subset H $, 
the product of flag varieties $ \dblFV = H/P_H \times G/P_G $ is called a double flag variety, 
on which $ H $ acts diagonally.

Now let $G$ be either $\U(n,n)$ or $\Sp_{2n}(\R)$.  
We classify the $H$-orbits on $ \dblFV $ in both cases and show that they admit exactly the same parametrization. Concretely, each orbit corresponds to a \emph{signed partial involution}, which can be encoded by simple combinatorial graphs. 
The orbit structure reduces to several families of smaller flag varieties, 
and we find an intimate relation of the orbit decomposition to Matsuki duality and Matsuki-Oshima's notion of clans. 

We also compute the Galois cohomology of each orbit, which exhibits another classification of the orbits by explicit matrix representatives.
\end{abstract}

\maketitle

\section{Introduction}

Let $ G $ be a connected reductive algebraic group and $ P_{G} $ its parabolic subgroup.  
The variety $ G/P_{G} $ is called a partial flag variety.   
For an involutive automorphism $ \sigma $ of $ G $, 
the fixed point subgroup $ G^{\sigma} $ of $ \sigma $ is called a symmetric subgroup, which is often disconnected.  We denote its connected component of the identity by $ H $  (sometimes it is convenient to consider $ G^{\sigma} $ itself, and if so, we don't insist on the connectedness).  
We also call $ H $ a symmetric subgroup by abuse of terminology.  
The group $ H $ is necessarily reductive, and we choose a parabolic subgroup $ P_{H} $ of $ H $.  
The product of two flag varieties 
$ \dblFV = H/P_{H} \times G/P_{G} $ with the diagonal $ H $-action is called a \emph{double flag variety}, which was introduced in \cite{NO.2011} as a generalization of the flag variety.

The theory of double flag variety has been studied so far over the complex number field $ \C $ by several authors.  
The triple flag varieties are considered to be special cases of the double flag varieties and studied by 
Magyar-Weyman-Zelevinsky \cite{MWZ.2000,MWZ.1999}, Matsuki \cite{Matsuki.2013,Matsuki.2015,Matsuki.arXiv2019} and Finkelberg-Ginzburg-Travkin \cite{Finkelbrg.Ginzburg.Travkin.2009}, 
Travkin \cite{Travkin.2009}.  

If there are only finitely many $ H $-orbits in $ \dblFV $, it is called \emph{of finite type}, and those triple flag varieties appeared in the papers mentioned above are of finite type.  
It is interesting to classify the double flag varieties of finite type.  In most of the cases, classifications still remain open.  
However, if $ P_{G} $ is a Borel subgroup of $ G $ or if $ P_{H} $ is a Borel subgroup of $ H $, there are complete classifications by He-N-Ochiai-Oshima \cite{HNOO.2013}; and 
for the double flag varieties of type AIII, 
complete classifications are obtained by Homma \cite{Homma.2021} and Fresse-N \cite{Fresse.Nishiyama.arXiv2309.17085} with no assumptions on $ P_{G} $ and $ P_{H} $.  
Thus, there is a little progress in the theory, 
but, as we mentioned above, these classifications are only given over $ \C $.   
In addition to that, descriptions of the orbits are still open in most of the cases (however, see \cite{Fresse.N.2023,Fresse.Nishiyama.arXiv2309.17085,Fresse.N.2021} for types AIII and CI).  There are interesting geometric problems such as closure relations, dimension formula, relationships to Schubert calculus or Springer-Steinberg theory.  A part of them has a relatively good understandings (e.g., see \cite{Fresse.N.2023,Fresse.N.2022.arXiv,Fresse.N.2020,Fresse.N.2016}) but still a vast majority is not explored yet.

\bigskip

Notice that the definition of double flag varieties is given for arbitrary base fields.  
In this paper, we consider double flag varieties over the \emph{real number field} $ \R $, when $ G $ is the indefinite unitary group $ \U(n,n) $ or the symplectic group $ \Sp_{2n}(\R) $.  
Over $ \R $, there are a few studies on the theory of double flag varieties.    
We mention Kobayashi-Matsuki \cite{Kobayashi.Mastuki.2014}, Kobayashi-Oshima \cite{Kobayashi.Oshima.2013} and Kobayashi-Speh \cite{Kobayashi.Speh.2015,Kobayashi.Speh.2018}, though the notion of double flag varieties is implicit there.  In N-{\O}rsted \cite{Nishiyama.Orsted.2018}, integral intertwiners are constructed by using Bruhat theory on orbits in the double flag varieties of type CI, where $ G = \Sp_{2n}(\R) $.

\medskip

Let us specify the situation more explicitly.  
The main aim of this article is to give a complete classification of the $H$-orbit decomposition of real double flag varieties $ \dblFV = H /B_{H} \times G/P_{G}$ in the following two cases.

\begin{Setting}\label{Setting-DFVR}
Let us refer the following quartet to $(G,H,P_{G},B_{H})$ according as Cases \caseA or \caseB.
\begin{enumerate}
    \item[\caseA]\label{Setting-DFVR:case:A}
$(\U(n,n), \GL_{n}(\C), P_{S}^{J}, B_{n}(\C))$, where $\U(n,n)$ is the indefinite unitary group, $P_{S}^{J}$ is the Siegel parabolic subgroup of $G = \U(n,n)$, and $B_{n}(\C)$ is a Borel subgroup of the symmetric subgroup $H \simeq \GL_{n}(\C)$ of $G$.

    \item[\caseB] $  (\Sp_{2n}(\R), \GL_{n}(\R), P_{S}^{\Sigma}, B_{n}(\R))$, where $\Sp_{2n}(\R)$ is the symplectic group, $P_{S}^{\Sigma}$ is the Siegel parabolic subgroup of $G = \Sp_{2n}(\R)$, and $B_{n}(\R)$ is a minimal parabolic subgroup of the symmetric subgroup $H \simeq \GL_{n}(\R)$ of $G$.
\end{enumerate}
See Section \ref{Section-Preliminary} for more details.
We call these two cases as Case \caseA and Case \caseB, respectively.
\end{Setting}

If we complexify our real double flag variety $ \dblFV $, we obtain a complex double flag variety $ \dblFV_{\C} $.
The following theorem is well known.

\begin{theorem}[{\cite[Thm.~5, Chap.~III \S 4.4]{Serre-Galois-Cohomology}}]\label{theorem-complex-finite-imply-real-finite}
Let $H_{\C}$ be a complex algebraic group defined over $\R$, which acts on an algebraic variety $V_{\C}$ homogeneously. Suppose this action is also defined over $\R$. Then, the number of $H_{\C}(\R)$-orbits on $V_{\C}(\R)$ is finite.
\end{theorem}

\skipover{
\begin{fact}\label{theorem-complex-finite-imply-real-finite}
    Let $G$ be a real reductive algebraic group, $H$ its symmetric subgroup, $P_{G}$ and $P_{H}$ parabolic subgroups of $G$ and $H$, respectively.
    Moreover, let $G_{\C}, H_{\C}, P_{G_{\C}}$ and $P_{H_{\C}}$ be complexifications of $G, H, P_{G}$ and $P_{H}$, respectively,
    and set $\dblFV_{\C} \coloneqq H_{\C} / B_{H_{\C}} \times G_{\C} / P_{G_{\C}}$.
    Then, $\# ( H_{\C} \backslash \dblFV_{\C} ) < \infty$ implies $\# ( H \backslash \dblFV ) < \infty$, 
i.e., $ \dblFV $ is of finite type if and only if $ \dblFV_{\C} $ is so.
\end{fact}

\begin{remark}
Fact \ref{theorem-complex-finite-imply-real-finite} follows from the theory of Galois cohomology. 
See \cite[Thm.~5, Chap.~III \S 4.4]{Serre-Galois-Cohomology} (and \cite[Example a), Chap.~III \S 4.2]{Serre-Galois-Cohomology}).
\end{remark}
}

By the classification from \cite[Thm.~5.2]{HNOO.2013},
we already know that 
the complex double flag varieties obtained from the complexifications of both of Cases \caseA and \caseB 
%%of Setting \ref{Setting-DFVR} 
are of finite type, namely $\# ( H_{\C} \backslash \dblFV_{\C} ) < \infty$.
Thus, Theorem \ref{theorem-complex-finite-imply-real-finite} implies that their real forms are also of finite type (i.e., $\# ( H \backslash \dblFV ) < \infty$) in these two cases.

Since $ \dblFV $ is of finite type, we are interested in the description of orbits.  
Notice that there is a natural bijection
\begin{equation*}
H \backslash \dblFV \simeq B_{H} \backslash G/P_{G}.  
\end{equation*}
So the classification of orbits amount to the classification of 
$ B_{H} $-orbits on the spherical Grassmannian $ G/P_{G} $.  
These Grassmannians are explicitly described as follows (See Section \ref{Section-Preliminary} for more details):
\begin{enumerate}
    \item[\caseA] The partial flag variety $G/P_{G} = \U(n,n) / P_{S}^{J}$ is isomorphic to the closed subvariety $\HLGr(\C^{2n})$ of the Grassmannian $\Gr_{n}(\C ^{2n})$ consisting of maximally isotropic subspaces:
    \begin{equation*}
    % \label{eq-G/P=HLGr}
\begin{aligned}
        & \U(n,n) / P_{S}^{J} \simeq 
%%\\ &
        \HLGr(\C^{2n}) \coloneqq
        \{ W \in \Gr_{n}(\C^{2n}) \mid W \text{ is isotropic in }\C^{(n|n)} \},
\end{aligned}
    \end{equation*}
    where $\C^{(n|n)} $ is a vector space of dimension $ 2 n $ endowed with an indefinite Hermitian inner product of signature $ (n,n) $.  
We call a maximally isotropic subspace $ W $ a \emph{Hermitian Lagrangian subspace}.

    \item[\caseB] The partial flag variety $G/P_{G} = \Sp_{n}(\R) / P_{S}^{\Sigma}$ is isomorphic to the closed subvariety $\LGr(\R^{2n})$ of the Grassmannian $\Gr_{n}(\R^{2n})$ consisting of maximally isotropic subspaces:
    \begin{equation*}
        \Sp_{n}(\R) / P_{S}^{\Sigma} \simeq \LGr(\R^{2n})
        \coloneqq
        \{ W \in \Gr_{n}(\R^{2n}) \mid W \text{ is isotropic in }\R^{2n}_{\Sigma} \},
    % \label{eq-G/P=LGr}
    \end{equation*}
where $\R^{2n}_{\Sigma} $ is a vector space of dimension $ 2 n $ endowed with a symplectic form.  
So the space $ \LGr(\R^{2n}) $ consists of \emph{Lagrangian subspaces}.  
\end{enumerate}

\subsection{Classification of orbits by partial permutations}

For the precise statement of the classification of $ H \backslash \dblFV \simeq  B_{H} \backslash G /P_{G}$, we need to prepare some notations.

\begin{definition}\label{Def-SPP-and-regOmega}  
    An $n\times n$ matrix $ \tau \in \Mat_{n}(\Z) $ is called a \emph{signed partial permutation} if 
    its entries belong to the set $ \{-1,  0, 1 \} $ and there is at most one nonzero entry in each row and also in each column.
    Then, we write $ \SPP_n $ for the set of signed partial permutations:
    \begin{equation*}
        \SPP_{n} \coloneqq \{ \tau \in \Mat_{n}(\Z) \mid \tau\text{ is a signed partial permutation}  \}.
    \end{equation*}
    Moreover, we set
    \begin{equation*}
        \regOmega_{n}
        \coloneqq
        \left\{
            \omega_{\tau_{1},\tau_{2}} \coloneqq   
            \begin{pmatrix}
                \tau_{1} \\ \tau_{2}
            \end{pmatrix}
            \in \Mat_{2n,n}(\Z)
        \Biggm| 
        \begin{array}{c}
             \tau_{1}, \tau_{2} \in \SPP_{n}  \\
             {}^{t}\tau_{1}\tau_{2} \in \Sym_{n}(\Z) \\
             \rank (\omega_{\tau_{1},\tau_{2}}) = n
        \end{array}
        \right\} ,
    \end{equation*}
where $ \Mat_{p,q}(A) $ denotes the space of $ p \times q $ matrices with their entries in a (commutative) ring $ A $ 
and $ \Sym_n(A) $ the space of symmetric matrices of degree $ n $.    
    Note that the third condition $\rank (\omega_{\tau_{1},\tau_{2}}) = n$ above implies that the image of the linear map $\omega_{\tau_{1},\tau_{2}} \colon \Q^{n} \to \Q^{2n}$ is an $n$-dimensional subspace of $\Q^{2n}$.
Thus, $\omega_{\tau_{1},\tau_{2}}$ gives a point $[\omega_{\tau_{1},\tau_{2}}]\in \Gr_{n}(\Q^{2n}) \subset \Gr_{n}(\R^{2n}) \subset \Gr_{n}(\C^{2n})$, a subspace generated by the row vectors of $ \omega_{\tau_{1},\tau_{2}} $.  
Moreover, the second condition ${}^{t}\tau_{1}\tau_{2} \in \Sym_{n}(\Z)$ 
is equivalent to that $[\omega_{\tau_{1},\tau_{2}}]$ is a Lagrangian subspace contained in $G/P_{G}$ in both Cases \caseA and \caseB. See Lemma \ref{Lem-G/P_S=LM/GL}.
\end{definition}

\begin{remark}\label{Remark-reg-Omega-W-Cn-set}
    It is easy to see that $ \regOmega_{n} $ is stable under the right multiplication of $ \Z_{2}^{n} \rtimes \Sgroup_{n}  $, where  $ (\varepsilon_{1},\dots,\varepsilon_{n}) \in  \Z_{2}^{n} \coloneqq \{\pm 1\}^{n} $ acts on $\regOmega_{n}$ by the right multiplication of the diagonal matrix $\diag(\varepsilon_{1},\dots,\varepsilon_{n}) \in \GL_{n}(\Z)$, 
and the symmetric group $\Sgroup_{n}  $ acts by permuting the rows.  
    Thus, the Weyl group of type $ C_{n} $ acts on $ \regOmega_{n} $ from the right.
    Similarly, $ \regOmega_{n} $ is also stable under the left multiplication of $ \diag(\varepsilon_{1},\dots,\varepsilon_{n},\varepsilon_{1},\dots,\varepsilon_{n}) \in \GL_{2n}(\Z) $ for $(\varepsilon_{1},\dots,\varepsilon_{n}) \in \Z_{2}^{n} $.
\end{remark}

The orbit decomposition $H \backslash \dblFV \simeq B_{H} \backslash G /P_{G}$ can be described in terms of $ \regOmega_{n} $.

\begin{theorem}[Theorem \ref{theorem-Omega=H-G-PS-2}]\label{theorem-Omega=H-G-PS}
Let $(G,H,P_{G},B_{H})$ be as in Case \caseA or \caseB of Setting \ref{Setting-DFVR}. Then, we have a bijection
    \begin{equation*}
    \begin{array}{ccc}
         \Z_{2}^{n} \backslash \regOmega_{n} / ( \Z_{2}^{n} \rtimes \Sgroup_{n} ) 
         & \xrightarrow{\sim} & B_{H} \backslash G / P_{G} \simeq H \backslash \dblFV 
         \\
         \rotatebox{90}{$\in$}
         &&
         \rotatebox{90}{$\in$}
         \\
         \omega_{\tau_{1},\tau_{2}}
         & \mapsto &
         B_{H} [\omega_{\tau_{1},\tau_{2}}],
    \end{array}
    \end{equation*}
    where $ G/P_G $ is considered to be a closed subvariety of the Grassmannian $ \Gr_{n}(\C^{2n}) $ or $ \Gr_{n}(\R^{2n}) $ 
(see Remark \ref{remark:orbit.inclusions} below).  
%%$\regOmega_{n}$ is defined in Definition \ref{Def-SPP-and-regOmega}.
In particular, the orbit decompositions $ H \backslash \dblFV \simeq B_{H} \backslash G /P_{G}$ share the same representatives in both Cases \caseA and \caseB.
\end{theorem}

%%The proof of this theorem \ref{theorem-Omega=H-G-PS} is given in Section \ref{Section-BGP-via-SPP} (Theorem \ref{theorem-Omega=H-G-PS-2}).

\begin{remark}\label{remark:orbit.inclusions}
%%By \eqref{eq-G/P=HLGr} and \eqref{eq-G/P=LGr}, 
We have natural inclusions
    \begin{equation*}
    \xymatrix{
        \LGr(\R^{2n}) 
        \ar@{^{(}->}[r]
        \ar[d]^-{\wr}
        &
        \HLGr(\C^{2n}) 
        \ar@{^{(}->}[r]
        \ar[d]^-{\wr}
        &
        \Gr_{n}(\C^{2n})
        \ar[d]^-{\wr}
        \\
        \Sp_{2n}(\R) / P_{S}^{\Sigma}
        \ar@{^{(}->}[r]
        &
        \U(n,n) / P_{S}^{J}
        \ar@{^{(}->}[r]
        &
        \GL_{2n}(\C) / P_{(n,n)},
    }
    \end{equation*}
    where $P_{(n,n)}$ is the maximal parabolic subgroup of $\GL_{2n}(\C)$ corresponding to the partition $2n=n+n$.
    This induces a map 
    \begin{equation*}
    \begin{array}{ccc}
         B_{n}(\R) \backslash \Sp_{2n}(\R) / P_{S}^{\Sigma}
         & \xrightarrow{\sim} &
         B_{n}(\C) \backslash \U(n,n) / P_{S}^{J} \\
         \rotatebox{90}{$\in$} & & \rotatebox{90}{$\in$} \\
         \mathcal{O} & \mapsto & B_{n}(\C) \mathcal{O} 
    \end{array} ,
    \end{equation*}
which, in fact, is a bijection by Theorem \ref{theorem-Omega=H-G-PS}.  
Therefore, each $B_{n}(\R)$-orbit on $\Sp_{2n}(\R) / P_{S}^{\Sigma}$ 
generates a unique $B_{n}(\C)$-orbit on $\U(n,n) / P_{S}^{J}$, and 
conversely, each $B_{n}(\C)$-orbit on $\U(n,n) / P_{S}^{J}$ cuts out a single $B_{n}(\R)$-orbit on $\Sp_{2n}(\R) / P_{S}^{\Sigma}$.  
This phenomenon is observed even in the case of complex double flag varieties, though the orbit correspondence is merely an injection (\cite{Fresse.N.2021}).  
Thus our result claims such orbit correspondence carries over to real points (at least in our present case).
\end{remark}

\subsection{A family of smaller symmetric spaces}
%% and Matsuki correspondence

In the course of the proof of 
Theorem \ref{theorem-Omega=H-G-PS}, 
we notice 
%%of this Theorem \ref{theorem-Omega=H-G-PS} 
that the orbit space $B_{H} \backslash G / P_{G}$ naturally contains a family of the orbit spaces of smaller full flag varieties with respect to the actions of \emph{real} symmetric subgroups.  
Let $ [n] = \{ 1, 2, \dots, n \} $ denote the set of positive integers up to $ n $, 
and define $ \binom{[n]}{m} $ as the collection of subsets $ I \subset [n] $ with the cardinality $ \# I = m $.

\begin{theorem}[Theorem \ref{theorem-DFV-contaion-KGB-caseA}]\label{theorem-DFV-contain-KGB}
\begin{penumerate}
        \item%%[\caseA]  
        Let $(G,H,P_{G},B_{H}) = (\U(n,n), \GL_{n}(\C), P_{S}^{J}, B_{n}(\C)) $ be in Case \caseA. 
%%of Setting \ref{Setting-DFVR}. 
Then, we have a bijection 
        \begin{equation*}
            H \backslash \dblFV \simeq B_{H} \backslash G /P_{G}
            \simeq
\coprod_{m = 0}^n 
\coprod_{\substack{p, q \geq 0 \\ p + q \leq m}} 
%%\coprod_{\substack{I \in \binom{[n]}{k} \\ J \in \binom{[k]}{p + q}}}
\;
\coprod_{I \in \binom{[n]}{m}, \; J \in \binom{[m]}{p + q}}
	       B_{p+q}(\C)  \backslash  \GL_{p+q}(\C) / \U(p,q),
        \end{equation*}  
        where $B_{p+q}(\C)$ is a Borel subgroup of $\GL_{p+q}(\C)$
        and $\U(p,q)$ is the indefinite unitary group with signature $(p,q)$.  

        \item%%[\caseB]  
Let $(G,H,P_{G},B_{H}) = (\Sp_{2n}(\R), \GL_{n}(\R), P_{S}^{\Sigma}, B_{n}(\R)) $ be in Case \caseB.
%%of Setting \ref{Setting-DFVR}. 
Then, we have a bijection 
        \begin{equation*}
            H \backslash \dblFV \simeq B_{H} \backslash G /P_{G}
            \simeq
\coprod_{m = 0}^n 
\coprod_{\substack{p, q \geq 0 \\ p + q \leq m}} 
%%\coprod_{\substack{I \in \binom{[n]}{k} \\ J \in \binom{[k]}{p + q}}}
\;
\coprod_{I \in \binom{[n]}{m}, \; J \in \binom{[m]}{p + q}}
	       B_{p+q}(\R)  \backslash  \GL_{p+q}(\R) / \Ogroup(p,q),
        \end{equation*}  
        where $B_{p+q}(\R)$ is a minimal parabolic subgroup of $\GL_{p+q}(\R)$
        and $\Ogroup(p,q)$ is the indefinite orthogonal group with signature $(p,q)$.
    \end{penumerate}
\end{theorem}

%%The proof of this Theorem \ref{theorem-DFV-contain-KGB} is given in Section \ref{Section-Matsuki-duality} (Theorem \ref{theorem-DFV-contaion-KGB-caseA}).

\begin{remark}
    The indices $ I $ and $ J $ in the unions in Theorem \ref{theorem-DFV-contain-KGB} do not affect the orbit spaces 
$ B_{p+q}(\C)  \backslash  \GL_{p+q}(\C) / \U(p,q) $ or 
$ B_{p+q}(\R)  \backslash  \GL_{p+q}(\R) / \Ogroup(p,q) $,  
so that they only give the multiplicities $ \mu_{p,q} $ of the orbit spaces (they must be counted as many as multiplicities).  
The multiplicity can be calculated as
\begin{equation*}
\mu_{p,q} = \sum_{m = p + q}^n \dbinom{n}{m} \dbinom{m}{p+q} = 2^{n - (p + q)} \dbinom{n}{p + q} .  
\end{equation*}
Actually those indices $ I $ and $ J $  have good meanings, which can be explained in terms of the Weyl group. 
    Namely, the index $I$ can be interpreted as the representatives of $(\Z_{2}^{n} \rtimes \Sgroup_{n})/\Sgroup_{n}$,
    which is related to 
%%the Weyl group corresponding to 
the partial flag variety $G/P_{G}$, and
    the index $J$ can be interpreted as the representatives of 
$ \Sgroup_{m}/ (\Sgroup_{p + q}\times  \Sgroup_{m - (p + q)})$, which is related to 
%%the Weyl group corresponding to 
the partial flag variety $\GL_{m}(\C) / P_{(p + q, m-(p + q))}$.
    See Remarks \ref{Remark-sigma-I-correspond-coprod} and \ref{Remark-sigma-J-correspond-coprod} for more details.
\end{remark}

\subsection{Graphs, Matsuki correspondence and Matsuki-Oshima's clans}

Note that the orbit spaces in
Theorem \ref{theorem-DFV-contain-KGB} are related to yet another orbit spaces in the full flag varieties by the celebrated Matsuki correspondence (also called Matsuki duality).

\begin{theorem}[{\cite{Matsuki.1979}, see Section \ref{subsec:KGB.Matsuki.duakity}}]\label{theorem-Matsuki-Duality-Intro}
There exist bijections which reverse the closure ordering:
    \begin{equation*}
    \begin{array}{ccc}
        B_{p+q}(\C)  \backslash  \GL_{p+q}(\C) / \U(p,q) 
        & \simeq & 
        B_{p+q}(\C)  \backslash  \GL_{p+q}(\C) / \GL_{p}(\C) \times \GL_{q}(\C),
        \\[1ex]
        B_{p+q}(\R)  \backslash  \GL_{p+q}(\R) / \Ogroup(p,q)
        & \simeq & 
        B_{p+q}(\R)  \backslash  \GL_{p+q}(\R) / \GL_{p}(\R) \times \GL_{q}(\R).
    \end{array}
    \end{equation*}
\end{theorem}

The right-hand side of the first isomorphism of Theorem \ref{theorem-Matsuki-Duality-Intro} is classified combinatorially in \cite[\S4]{Matsuki.Oshima.1990} 
by the notion of clans.  
Let us explain it in detail (see Theorem \ref{theorem-KGB-clan} below for the statement).  

\begin{definition}[{\cite[\S4]{Matsuki.Oshima.1990}}]
\label{Def-Gamma(p,q)}
Let $p,q\in\Z_{\geq 0}$. Then, we define a set $\Gamma(p,q)$ of graphs consisting of $p+q$ vertexes,
which are decorated by $+,-$, or arcs according to the following rules.
\begin{enumerate}
	\item[\upshape{(arc)}] The two end points of each arc are different.
	\item[\upshape{(sgn)}] $ \# \{ \text{vertices decorated by $+$} \} - \# \{ \text{vertices decorated by $-$} \} = p - q $.
%%``The number of vertices decorated by $+$'' minus ``the number of vertices decorated by $-$'' is equal to $p-q$.
\end{enumerate}
More rigorously, we set (recall the notation $ [n] = \{ 1, 2, \dots, n \} $)
\begin{equation*}
	\Gamma(p,q) \coloneqq
	\left\{
        \gamma \colon [p + q] 
%%\{1,2,\dots,p+q\} 
\to [p + q] \cup \{+,-\}
%%\{1,2,\dots,p+q\} 
	\Bigm| 
	   \begin{array}{c}
        	\text{$\gamma$ satisfies} \\ \text{(arc) and (sgn)}
	   \end{array}
	\right\},
\end{equation*}
where
\begin{enumerate}
	\item[\upshape{(arc)}]  if $\gamma(i)=j \in [p+q]$, then we have $i\neq j$ and $\gamma(j)=i$,
	
	\item[\upshape{(sgn)}]  $ \# \left( \gamma^{-1}(+) \right)  - \# \left( \gamma^{-1}(-) \right) = p-q $,
\end{enumerate}
which paraphrase the same conditions above (so we use the same symbols).
We call a graph belonging to $\Gamma(p,q)$ a $(p,q)$-\textit{clan}.    
\end{definition}

\begin{example}\label{Ex-Gamma(p,q)}
Let $\gamma\colon \{1,2,\dots,5\} \to \{1,2,\dots,5\} \cup \{+,-\}$
be a function defined by
\begin{equation*}
        \gamma(1) = -,\quad
        \gamma(2) = 4,\quad
        \gamma(3) = +,\quad
        \gamma(4) = 2,\quad
        \gamma(5) = -.
\end{equation*}
Then we have $\gamma \in\Gamma(3,2)$. 
Its graphical representation is given by
\begin{equation*}
    \gamma
    =
    \xymatrix{
        1_{-} &
        2 \ar@{-}@/^10pt/[rr] &
        3_{+} &
        4     &
        5_{-}.
    }
\end{equation*}
\end{example}

\begin{theorem}[{\cite[\S4]{Matsuki.Oshima.1990}}]\label{theorem-KGB-clan}
There exists a bijection:
    \begin{equation*}
        B_{p+q}(\C)  \backslash  \GL_{p+q}(\C) / \GL_{p}(\C) \times \GL_{q}(\C)
        \simeq 
        \Gamma(p,q).
    \end{equation*}
\end{theorem}

Thus, we have a combinatorial description of $B_{H} \backslash G / P_{G} \simeq H \backslash \dblFV$ in Case \caseA and consequently the same description for Case \caseB 
(see Remark \ref{remark:orbit.inclusions}).  
%%of Setting \ref{Setting-DFVR}.

%%%%%%%% skipover:begin %%%%%%%%%%%%%%%%%%
\skipover{
\begin{corollary}
\label{Cor-BGP=coprod-Gamma(p,q)}
Let $(G,H,P_{G},B_{H})$ be in Case \caseA or Case \caseB of Setting \ref{Setting-DFVR}. Then, we have a bijection 
        \begin{equation*}
            B_{H} \backslash G /P_{G}
            \simeq
\coprod_{k = 0}^n 
\coprod_{\substack{p, q \geq 0 \\ p + q \leq k}} 
            \coprod_{\substack{I \in \binom{[n]}{k} \\ J \in \binom{[k]}{p + q}}}
	        \Gamma(p,q).
        \end{equation*}     
\end{corollary}

\begin{proof}
    This follows from Theorems \ref{theorem-DFV-contain-KGB}, \ref{theorem-Matsuki-Duality-Intro} and \ref{theorem-KGB-clan}.
\end{proof}
}
%%%%%%%% skipover:end   %%%%%%%%%%%%%%%%%%

We generalize the notion of Matsuki-Oshima's clans to give a combinatorial description of the orbit decomposition of the whole double flag varieties 
$H \backslash \dblFV \simeq B_{H} \backslash G / P_{G}$ using graphs.
For the precise statement, we need more notations.  
%%, which is inspired by Matsuki-Oshima's clan.  
%%(Definition \ref{Def-Gamma(p,q)})
%%to prepare some combinatorial language (Definition \ref{Def-Gamma(n)} below)

\begin{definition}\label{Def-Gamma(n)}
For a nonnegative integer $n\in\Z_{\geq 0}$, we define a set $\Gamma(n)$ of graphs consisting of $n$ vertices, all of which are decorated by one of $+,-,c,d$, or arcs according to the following rule.
\begin{enumerate}
	\item[\upshape{(arc)}] The two end points of each arc are different.
\end{enumerate}
More rigorously, we set
\begin{equation*}
	\Gamma(n)
	\coloneqq
	\{
        \gamma \colon [n] \to [n] \cup \{+,-,c,d\}
	\mid 
	   \text{$\gamma$ satisfies (arc)}
    \},
\end{equation*}
where $ [n] = \{1,2,\dots,n\} $ and 
\begin{enumerate}
	\item[\upshape{(arc)}] if $\gamma(i)=j \; (i, j \in [n]) $, then we have $i \neq j$ and $\gamma(j)=i$,
\end{enumerate}
which is a paraphrase of the same condition above (so we use the same symbol).
We call an element of $\Gamma(n)$ a \emph{decorated $n$-clan}.
\end{definition}

\begin{example}\label{Ex-Gamma(n)}
Let $\gamma\colon \{1,2,\dots,7\} \to \{1,2,\dots,7\} \cup \{+,-,c,d\}$ be a function in $ \Gamma(7) $ defined by
\begin{equation*}
        \gamma(1) = d,\quad
        \gamma(2) = -,\quad
        \gamma(3) = c,\quad
        \gamma(4) = 6,\quad
        \gamma(5) = +,\quad
        \gamma(6) = 4,\quad
        \gamma(7) = -.
\end{equation*}
%%Then we have $\gamma \in \Gamma(7)$.
Its graphical representation is given by
\begin{equation*}
    \gamma
    =
    \xymatrix{
        1_{d} &
        2_{-} &
        3_{c} &
        4 \ar@{-}@/^10pt/[rr] &
        5_{+} &
        6     &
        7_{-}.
    }
\end{equation*}
% We note that this $\gamma \in \Gamma(7)$ corresponds to $\omega_{C,D} \in \LregMat_{2n,n}/\GL_{n}(\C) \simeq G/P_{G}$ under the notation of Definition \ref{Def-LM-circ}, where 
% \begin{equation*}
%     C=\begin{pmatrix}
%         1&&&&&& \\ &-1&&&&& \\ &&0&&&& \\ &&&&&1& \\ &&&&1&&& \\ &&&1&&& \\ &&&&&&-1
%     \end{pmatrix},
%     \quad
%     D=\begin{pmatrix}
%         0&&&&&& \\ &1&&&&&& \\ &&1&&&& \\ &&&1&&&& \\ &&&&1&& \\ &&&&&1& \\ &&&&&&1
%     \end{pmatrix}.
% \end{equation*}
\end{example}

These decorated clans classify the orbits in double flag varieties 
both in Case \caseA and Case \caseB.

\begin{theorem}[Theorem \ref{Thm-Gamma(n)=DFVR}]\label{theorem-H-G-P_G=Gamma(n)}
Let $(G,H,P_{G},B_{H})$ be the quartet in Case \caseA or \caseB of Setting \ref{Setting-DFVR}. Then, we have a bijection
    \begin{equation*}
        H \backslash \dblFV \simeq B_{H} \backslash G / P_{G} \simeq \Gamma(n).
    \end{equation*}    
\end{theorem}

%%Theorem \ref{theorem-H-G-P_G=Gamma(n)} is proved in Section \ref{Section-BGP-via-SPP} 

\begin{corollary}\label{cor:number.of.orbits.HX}
The number of $ H $-orbits in $ \dblFV $ is given by 
	\begin{equation*}
\# H \backslash \dblFV = \#\Gamma(n)
%%=\sum_{k=0}^{ \lfloor\frac{n}{2}\rfloor }
%%4^{n-2k} \frac {n!} { (2!)^k (n-2k)! k! }
=	\sum_{k=0}^{ \lfloor\frac{n}{2}\rfloor }
(2 k - 1)!! \dbinom{n}{2k} 2^{2 (n - 2 k)}
=
		4^{n}n!
		\sum_{k=0}^{ \lfloor\frac{n}{2}\rfloor}
		\frac{1}{2^{5k}}  \frac{1}{(n-2k)!k!},
	\end{equation*}
	where $\lfloor\frac{n}{2}\rfloor$ is the integral part of $\frac{n}{2}$.
\end{corollary}

By chasing the proofs of relevant theorems, we obtain an algorithm for constructing an orbit representative $\omega_{\tau_{1},\tau_{2}} \in \regOmega_{n}$ from a decorated $ n $-clan $ \gamma \in \Gamma(n) $, which gives the bijections of Theorem \ref{theorem-H-G-P_G=Gamma(n)}. 
See Example \ref{Ex-Gamma-to-Omega}.

\subsection{Galois cohomology}

Finally, let us explain how to get information of the orbits by using the Galois cohomology.  
This is the third method for classifying the orbits $ H \backslash \dblFV $.   
The key point is the following proposition, which is easily obtained by combining well-known results.

\begin{proposition}\label{fact-real-orbits-via-Galois-cohomology}
Let $G$ be a real reductive algebraic group, $P$ its parabolic subgroup,
and $B$ a real algebraic subgroup of $G$ (not necessarily a parabolic subgroup of $G$).
Write $G_{\C}, P_{\C}$, and $B_{\C}$ for complexifications $G,P$ and $B$, respectively, and set $X_{\C} \coloneqq G_{\C} / P_{C}$.
Let $\Xi\subset B_{\C}\backslash X_{\C} $ be a subset of $B_{\C}$-orbits on $X_{\C}$ having an $\R$-rational point:
\begin{equation*}
	\Xi\coloneqq
	\{
        \mathcal{O} \in B_{\C} \backslash X_{\C}
    \mid
	   \mathcal{O}(\R)\neq \emptyset
    \}.
\end{equation*}
Suppose that the first Galois cohomology $H^{1}(\R, B_{\C})$ is trivial, 
\begin{equation*}
    H^{1}(\R, B_{\C}) = 1.
\end{equation*}
Then, we have a bijection
\begin{equation}
	B \backslash G / P 
	\simeq 
	\coprod_{\mathcal{O}\in\Xi}
	H^{1}(\R, (B_{\C})_{x_{\mathcal{O}}}),
\label{eq-fact-real-orbits-via-Galois-cohomology}
\end{equation}
where $(B_{\C})_{x_{\mathcal{O}}}$ denotes the stabilizer in $B_{\C}$ at $x_{\mathcal{O}} \in \mathcal{O}(\R)$.
\end{proposition}

\begin{proof}
    This is a special case of general results, which are summarized in Proposition \ref{Prop-real-orbits-via-Galois-cohomology-in-Appx} in Appendix \ref{Section-real-points-of-double-coset-decomposition}. Although this seems to be well known to the experts, we give a proof of the proposition 
    for the sake of completeness.
\end{proof}

\medskip 

Let us return to our basic settings, i.e., Cases \caseA and \caseB, and apply the above proposition to them.
We write $(G_{\C},H_{\C},P_{G_{\C}},B_{H_{\C}})$ for the complexification of $(G,H,P_{G},B_{H})$ (See Section \ref{Section-Preliminary} for more details).

\begin{convention}\label{Convention-complexifications} 
In our setting, the complexification $(G_{\C},H_{\C},P_{G_{\C}},B_{H_{\C}})$ becomes the following quartets 
according as Cases \caseA or \caseB.
\begin{enumerate}
    \item[\caseA] $(\GL_{2n}(\C), \GL_{n}(\C) \times \GL_{n}(\C), P_{(n,n)}, B_{n}^{+} \times B_{n}^{-})$, 
    where $P_{(n,n)}$ is the maximal parabolic subgroup of $G_{\C} = \GL_{2n}(\C)$ corresponding to the partition $2n=n+n$, 
    and $B_{n}^{+}$ (resp.~$B_{n}^{-}$) is a subgroup of upper (resp.~lower) triangular matrices of $\GL_{n}(\C)$.  
In this case, the flag variety $G_{\C} / P_{G_{\C}} = \GL_{2n}(\C) / P_{(n,n)}$ can be identified with the Grassmannian 
    \begin{equation*}
        \GL_{2n}(\C) / P_{(n,n)} \simeq  \Gr_{n}(\C^{2n}).
    % \label{eq-G/P=Gr}
    \end{equation*}

    \item[\caseB] $(\Sp_{2n}(\C), \GL_{n}(\C), P_{S}^{\C}, B_{n}(\C))$, where $P_{S}^{\C}$ is the Siegel parabolic subgroup of $G_{\C} = \Sp_{2n}(\C)$.
The flag variety $G_{\C}/P_{G_{\C}} = \Sp_{n}(\C) / P_{S}^{\C}$ is identified with the Lagrangian Grassmannian
%%is isomorphic to the closed subvariety $\LGr(\C^{2n})$ of the Grassmanian $\Gr_{n}(\C^{2n})$ consisting of maximally isotropic subspaces:
    \begin{equation*}
        \Sp_{n}(\C) / P_{S}^{\C} \simeq \LGr(\C^{2n})
        \coloneqq
        \{ W \in \Gr_{n}(\C^{2n}) \mid W \text{ is Lagrangian in }\C^{2n}_{\Sigma} \},
    % \label{eq-G/P=LGr-C}
    \end{equation*}
where $\C^{2n}_{\Sigma} $ is the symplectic vector space of dimension $ 2 n $ (the complexification of the real one, which we are considering).
\end{enumerate}    
\end{convention}

In both Cases \caseA and \caseB, 
%%of Setting \ref{Setting-DFVR}, 
the assumptions of Proposition \ref{fact-real-orbits-via-Galois-cohomology} are satisfied 
(see Lemma \ref{Fact-caseAB-H^1(B,R)=H^1(P,R)=1} below). 
Thus, by Proposition \ref{fact-real-orbits-via-Galois-cohomology}, the classification of $B_{H} \backslash G /P_{G}$ is reduced to the description of the double coset space $B_{H_{\C}} \backslash G_{\C} /P_{G_{\C}}$, and the computation of the Galois cohomologies in the right-hand side of \eqref{eq-fact-real-orbits-via-Galois-cohomology}.
The former was obtained in earlier works (\cite{Fresse.N.2020,Fresse.N.2021}), 
%%(Theorem \ref{Thm-complex-orbit-decomp-AIII-and-CI})
which we will summarize below in Theorem \ref{Thm-complex-orbit-decomp-AIII-and-CI}.  

To state the classification of the orbits in $ H_{\C} \backslash \dblFV_{\C} \simeq  B_{H_{\C}} \backslash G_{\C} /P_{G_{\C}}$, we prepare some notations.

\begin{definition}\label{Def-PP-regT-regR}  
    An $n\times n $ matrix $ \tau \in \Mat_{n}(\Z) $ is called a \emph{partial permutation} if 
    its entries belong to the set $ \{ 0, 1 \} $ and there is at most one nonzero entry in each row and also in each column.
    Then, we write $ \PP_{n} $ for the set of partial permutations:
    \begin{equation*}
        \PP_{n} \coloneqq \{ \tau \in \Mat_{n}(\Z) \mid \tau\text{ is a partial permutation}  \}.
    \end{equation*}
    Moreover, we set
    \begin{equation*}
    \begin{split}
        \regT_{n} &\coloneqq
        \left\{
            \tau \coloneqq   
            \begin{pmatrix}
                \tau_{1} \\ \tau_{2}
            \end{pmatrix} \in \Mat_{2n,n}(\Z)
        \biggm| 
            \begin{array}{c}
                \tau_{1}, \tau_{2} \in \PP_{n}  \\  \rank \tau = n
            \end{array}
        \right\},
        \\
        \regR_{n} &\coloneqq
        \left\{
            \tau \coloneqq   
            \begin{pmatrix}
                \tau_{1} \\ \tau_{2}
            \end{pmatrix} \in \Mat_{2n,n}(\Z)
        \Biggm| 
            \begin{array}{c}
                \tau_{1}, \tau_{2} \in \PP_{n}  \\ {}^{t}\tau_{1}\tau_{2} \in \Sym_{n}(\Z) \\ \rank \tau = n
            \end{array}
        \right\}  
        \subset \regT_{n}.
    \end{split} 
    \end{equation*}
    Note that the condition $\rank \tau = n$ implies that the image $[\tau] \coloneqq \operatorname{Im} \tau$, a subspace generated by the row vectors of $ \tau $,  
represents a point in $\Gr_{n}(\Q^{2n}) \subset \Gr_{n}(\C^{2n})$.
    Moreover, the condition ${}^{t}\tau_{1} \tau_{2} \in \Sym_{n}(\Z)$ implies that $[\tau]$ is a Lagrangian subspace belonging to $G_{\C}/P_{G_{\C}}$ in Case \caseB (See Section \ref{Section-Preliminary} for more details).
\end{definition}

\begin{remark}\label{Remark-Weyl-group-action-An-on-regT-regR}
    It is easy to see that $ \regT_{n} $ and $\regR_{n}$ are stable under the right multiplication by the permutation matrix in $ \Sgroup_{n} $ (or permuting rows). 
Thus, the Weyl group of type $A_{n-1}$ acts on $ \regT_{n} $ and $\regR_{n}$ naturally.
\end{remark}

The following theorem gives the orbit decomposition $H_{\C} \backslash \dblFV_{\C} \simeq B_{H_{\C}} \backslash G_{\C} / P_{G_{\C}}$.  

\begin{theorem}[{\cite[Thms.~0.1 and 4.6]{Fresse.N.2021}}]\label{Thm-complex-orbit-decomp-AIII-and-CI}
\begin{penumerate}
\item%%[\caseA] 
        Let $(G_{\C},H_{\C},P_{G_{\C}},B_{H_{\C}})$ be the complexification of the quartet in Case \caseA of Convention \ref{Convention-complexifications}. Then, we have a bijection which classifies the orbits.
        \begin{equation*}
        \begin{array}{ccc}
         \regT_{n} / \Sgroup_{n} & \xrightarrow{\sim} & B_{H_{\C}} \backslash G_{\C} / P_{G_{\C}}  
         \\
         \rotatebox{90}{$\in$} && \rotatebox{90}{$\in$}
         \\
         \tau & \mapsto & B_{H_{\C}} [\tau].
        \end{array}
        \end{equation*}

\item%%[\caseB] 
        Let $(G_{\C},H_{\C},P_{G_{\C}},B_{H_{\C}})$ be the complexification of the quartet in Case \caseB.   
%%of Setting \ref{Convention-complexifications}. 
Then, we have a bijection which classifies the orbits.
        \begin{equation*}
        \begin{array}{ccc}
         \regR_{n} / \Sgroup_{n}  & \xrightarrow{\sim} & B_{H_{\C}} \backslash G_{\C} / P_{G_{\C}}  
         \\
         \rotatebox{90}{$\in$}  && \rotatebox{90}{$\in$}
         \\
         \tau & \mapsto & B_{H_{\C}} [\tau].
        \end{array}
        \end{equation*}
    \end{penumerate}
\end{theorem}

Thus, in order to obtain the classification of $B_{H} \backslash G / P_{G}$ by using the Galois cohomology, 
first we have to know which $B_{H_{\C}}$-orbit has an $\R$-rational point, and then to determine the Galois cohomology corresponding to such orbits, 
which is the recipe in Proposition \ref{fact-real-orbits-via-Galois-cohomology}. 
%%These are described as follows:

\begin{definition}\label{Def-d-tau}
For $\tau = \begin{pmatrix}
    \tau_{1} \\ \tau_{2}
\end{pmatrix} \in \regT_{n}$, we define $d(\tau) \in \Z_{\geq 0}$ by
    \begin{equation*}
        d(\tau) \coloneqq 
        \#\{ 
            j \in [n] %%\{1,2,\dots,n\}
        \mid 
            \text{the $j$-th column of }\tau_{1}
            \text{ is equal to}
            \text{ the $j$-th column of }\tau_{2}
        \}.
    \end{equation*}

Note that $d(\tau)$ is invariant with respect to the action of $\Sgroup_{n}$ (described in Remark \ref{Remark-Weyl-group-action-An-on-regT-regR}) so that 
it gives a well defined function on $[\tau]\in\regT_{n} / \Sgroup_{n}$, denote by $ d([\tau])$. 
\end{definition}

We first prove the following lemma.

\begin{lemma}\label{Fact-caseAB-H^1(B,R)=H^1(P,R)=1}
    Let $(G_{\C}, H_{\C}, P_{G_{\C}}, B_{H_{\C}})$ be the complexification of the quartets in Case \caseA or \caseB 
    given in Convention \ref{Convention-complexifications}.
    Then, we have
    \begin{equation*}
        H^{1}(\R, B_{H_{\C}}) =1. %, \quad H^{1}(\R, P_{G_{\C}}) =1.
    \end{equation*}
    Namely, the assumption of Proposition \ref{fact-real-orbits-via-Galois-cohomology} is satisfied.
\end{lemma}

\begin{proof}
    In Case \caseA, this fact is proved in \cite[Lem.~2.6]{Nishiyama.Tauchi.2024}.  
%%\footnote{
    Note that there exists an error in the proof of \cite[Lem.~2.6]{Nishiyama.Tauchi.2024}.  
    For the correct proof, see arXiv version (arXiv:2504.09177).
%%}
    Case \caseB can be proved in the same way.
\end{proof}

\begin{remark}
    We note that there exists a Borel subgroup $B_{\C}$ of $\GL_{n}(\C)$ defined over $\R$ whose Galois cohomology is \textit{not} trivial.
%     For example, let $G\coloneqq \U(p,p+1)$ and $P$ be its minimal parabolic subgroup. Then, $P_{\C}$ is a Borel subgroup of $G_{\C}=\GL_{2p+1}(\C)$ because $G$ is quasi-split. Moreover, the Levi subgroup $T$ of $P$ is isomorphic to $S^{1} \times \GL(1,\R)^{p}$. Then, we have $H^{1}(\R, P_{\C}) \simeq H^{1}(\R,T_{\C}) = H^{1}(\R,(S^{1} \times \GL_{1}(\R)^{p})_{\C})$ by Lemma \ref{Lem-for-solvableB-H^1(B,R)=H^1(T,R)-flag}, and $H^{1}(\R,(S^{1} \times \GL_{1}(\R)^{p})_{\C}) \simeq H^{1}(\R, S^{1}_{\C})$ by $H^{1}(\R, \GL_{1}(\C)^{p}) = 1$ (Hilbert's Theorem 90).
%     Finally, the similar argument of Lemma \ref{Lem-H^1-tau-3cases-flag} implies $H^{1}(\R, S^{1}_{\C}) \simeq \{\pm1\} \neq 1$, which is the desired assertion.
% taito{It seems that there was an error in the previous assertion (which is the ``skipover'' now). Thus, I rewrote it in the form above.}
In fact, let $G$ be a quasi-split real form of a complex reductive group $G_{\C}$, which is not split, 
    and $P$ a minimal parabolic subgroup of $G$.  
%%    and suppose that $G$ is quasi-split (but not split). 
    Then, the complexification $P_{\C}$ of $P$ is a Borel subgroup of $G_{\C}$ since $G$ is quasi-split. Moreover the Levi subgroup $T$ of $P$ is isomorphic to $\GL_{1}(\R)^{a} \times (S^{1})^{b} \times \GL_{1}(\C)^{c}$ for some $a,b,c \in \Z_{\geq0}$ with $b>0$ or $c>0$.  
    %%because $G$ is quasi-split. 
Let us denote $ B_{\C} = P_{\C} $ and $ B = P $.  
    Then, we have $H^{1}(\R, B_{\C}) \simeq H^{1}(\R,T_{\C})$ by Lemma \ref{Lem-for-solvableB-H^1(B,R)=H^1(T,R)-flag}. Because the Galois cohomology preserves the direct product, we have 
    $H^{1}(\R,T_{\C}) \simeq H^{1}(\R,\GL_{1}(\R)_{\C})^{a} \times H^{1}(\R,S_{\C}^{1})^{b} \times H^{1}(\R,\GL_{1}(\C)_{\C})^{c}$. 
    We note that Hilbert's Theorem 90 implies $H^{1}(\R, \GL_{1}(\R)_{\C} ) = 1$. The similar argument as that in  Lemma \ref{Lem-H^1-tau-3cases-flag} implies
    $H^{1}(\R, \GL_{1}(\C)_{\C} ) = 1$ and $H^{1}(\R, S^{1}_{\C}) \simeq \{\pm 1\}$.
    Summarizing, we have $H^{1}(\R, B_{\C}) \simeq \{\pm 1\}^{b}$.
    Therefore, $H^{1}(\R, B_{\C}) \neq 1$ if $b>0$.
    (For example, consider $G = \U(p,p+1)$ so that $G_{\C} = \GL_{2p+1}(\C)$.)
\end{remark}

\begin{theorem}\label{theorem-Galois-RDFV}
\begin{penumerate}
\item\label{theorem-Galois-RDFV:item:1:nonempty.R.pts}
Suppose we are in Case \caseA of Convention \ref{Convention-complexifications}, 
and let $\mathcal{O}_{\tau}$ be the orbit corresponding to $\tau \in \regT_{n} / \Sgroup_{n}$ via the bijection of Theorem \ref{Thm-complex-orbit-decomp-AIII-and-CI}. 
Then, $\mathcal{O}_{\tau}(\R) \neq \emptyset$ if and only if $\tau \in \regR_{n} / \Sgroup_{n} \subset \regT_{n} / \Sgroup_{n}$.  
In Case \caseB, every orbit $\mathcal{O}_{\tau}$ corresponding to 
$\tau \in \regR_{n} / \Sgroup_{n}$ has nonempty real points.

\item\label{theorem-Galois-RDFV:item:2:Galois.tells.orbits.R}
Suppose we are in Case \caseA or Case \caseB, and  
%%        Let $(G_{\C},H_{\C},P_{G_{\C}},B_{H_{\C}})$ be in Case \caseA of Setting \ref{Convention-complexifications}. 
that the orbit $\mathcal{O}_{\tau}$ is corresponding to $\tau \in \regR_{n} / \Sgroup_{n}$ via the bijection of Theorem \ref{Thm-complex-orbit-decomp-AIII-and-CI}. 
Then the number of $B_{H}$-orbits in the real points $\mathcal{O}_{\tau}(\R)$ of $\mathcal{O}_{\tau}$ is $2^{d(\tau)}$ and they are parametrized by 
the first Galois cohomology 
%%$H^{1}(\R, (B_{H_{\C}})_{[\tau]})$ is given by
%
        \begin{equation}\label{eq-thoerem-Galois-RDFV-H^1(B,R)=pm1}
            H^{1}(\R, (B_{H_{\C}})_{[\tau]}) \simeq \{\pm1\}^{d(\tau)}, 
        \end{equation}
        where $d(\tau)\in \Z_{\geq 0}$ is given in Definition \ref{Def-d-tau},
        and $(B_{H_{\C}})_{[\tau]}$ is the stabilizer at $[\tau] \in G_{\C}/P_{G_{\C}}$ in $B_{H_{\C}}$.
As a consequence, we get a classification of 
orbits in the double flag variety:
        \begin{equation*}
            H \backslash \dblFV \simeq B_{H} \backslash G /P_{G}
            \simeq 
            \coprod_{\tau \in \regR_{n} / \Sgroup_{n}} \{\pm 1\}^{d(\tau)} .
        \end{equation*}
The classifications of orbits in Case \caseA and in Case \caseB are literally the same.
\end{penumerate}    
\end{theorem}

In particular, 
we get another enumeration formula of the orbits, namely,
\begin{equation*}
    \# H \backslash \dblFV 
    = \sum_{\tau \in \regR_{n} / \Sgroup_{n}} 2^{d(\tau)} .
\end{equation*}
Compare this to the one in Corollary \ref{cor:number.of.orbits.HX}.

\subsection{Closure relation}
Theorem \ref{theorem-H-G-P_G=Gamma(n)} states a combinatorial description of the orbit decomposition of the double flag variety $H \backslash \dblFV \simeq \Gamma(n)$ using graphs.
It is natural to ask about the closure relation of orbits and describe such relation in terms of graphs.
In this respect, we have the following theorem, which gives sufficient conditions for the closure relations.  
Note that the closure relation of $ H $-orbits in $ \dblFV $ and that of $ B_H $-orbits in $ G/P_G $ are the same (cf.~\cite[Lemma 3.1]{Fresse.N.2023}).

\begin{theorem}\label{Thm-closure-relation}
Let $\gamma, \gamma ' \in \Gamma (n)$ 
%%(Definition \ref{Def-Gamma(n)}) 
and denote by $\mathcal{O}_{\gamma}, \mathcal{O}_{\gamma'} \in H \backslash \dblFV$ the corresponding orbits under the isomorphism $\Gamma(n) \simeq H \backslash \dblFV $ of Theorem \ref{theorem-H-G-P_G=Gamma(n)}.  
We write $\overline{\mathcal{O}}_{\gamma}$ for the closure of $\mathcal{O}_{\gamma}$ in $\dblFV$.
Then, we have $\overline{\mathcal{O}}_{\gamma} \supset \mathcal{O}_{\gamma'}$ if one of the following conditions holds.
%% (see Remark \ref{Remark-graphic-closure-relation}):   
In the following conditions, we only indicate the difference between $ \gamma $ and $ \gamma' $.  So, if $ m \in [n] $ is not mentioned in the condition it means 
$ \gamma(m) = \gamma'(m) $.
\begin{enumerate}
\item[\upshape{(i)}] 
There exists $i \in [n]$ such that 
$    \gamma(i) \in \{+,-\} $;
$    \gamma'(i) \in\{c,d\} $.

\item[\upshape{(ii)}] 
There exist $i < j \in [n]$ such that
$    (\gamma(i)=c \textrm{ or } \gamma(j)=d ) $;    
$    \gamma'(i) =\gamma(j) $
and 
$    \gamma'(j) =\gamma(i) $.

\item[\upshape{(iii)}] 
There exist $i < j \in [n]$ such that 
$    \{\gamma(i),\gamma(j)\}=\{+,-\} $; 
$    \gamma'(i) =j $ and $ \gamma'(j) =i $.

\item[\upshape{(iv)}] 
There exist $i < j \in [n]$ such that
$    \gamma(i)=j $ and 
$     \gamma(j)=i $; 
$ (\gamma'(i) \in\{+,-\} $ and $ \gamma'(j) = c ) $ or 
$ (\gamma'(i) =d $ and $ \gamma'(j) \in\{+,-\} ) $.

\item[\upshape{(v)$_{c}$}] There exist $i < j < k \in [n]$ such that
$    \gamma(i)=c $
and 
$    \gamma(j)=k , \gamma(k)=j$;
$    \gamma'(k) = c $
and 
$    \{\gamma'(i),\gamma'(j)\} = \{+,-\} $.

\item[\upshape{(v)$_{d}$}] 
There exist $i < j < k \in [n]$ such that
$    \gamma(i)=j , \; \gamma(j)=i $ 
and 
$    \gamma(k)=d $; 
$    \gamma'(i) = d $
and 
$    \{\gamma'(j),\gamma'(k)\} = \{+,-\} $.

\item[\upshape{(vi)$_{1}$}] 
There exist $i < j < k \in [n]$ such that
$    \gamma(i)\in\{+,-\} $
and 
$    \gamma(j)=k $; 
$    \gamma'(i)=k $ 
and
$ \gamma'(j) = \gamma(i) $.

\item[\upshape{(vi)$_{2}$}] 
There exist $i < j < k \in [n]$ such that
$    \gamma(k)\in\{+,-\} $ 
and
$    \gamma(i)=j , \; \gamma(j) = i $; 
$    \gamma'(i)=k , \; \gamma'(k) = i $
and 
$    \gamma'(j) = \gamma(k) $.

\item[\upshape{(vii)$_{1}$}] 
There exist $i < j < k < \ell \in [n]$ such that
$    \gamma(i)=j, \;  \gamma(j)=i $ 
and 
$    \gamma(k)=\ell, \; \gamma(\ell)=k $;
$    \gamma'(i)=k , \; \gamma'(k) = i $
and 
$    \gamma'(j)=\ell, \; \gamma'(\ell) = j $.

\item[\upshape{(vii)$_{2}$}] 
There exist $i < j < k < \ell \in [n]$ such that
$    \gamma(i)=k , \; \gamma(k) = i $
and 
$    \gamma(j)=\ell , \; \gamma(\ell) = j $; 
$    \gamma'(i)=\ell , \; \gamma'(\ell) = i $
and 
$    \gamma'(j)=k, \; \gamma'(k) = j $.
\end{enumerate}
\end{theorem}

\begin{remark}\label{Remark-graphic-closure-relation}
The graphical interpretation of the conditions in Theorem \ref{Thm-closure-relation} is given by the following.
Here, we use the following notations:
\begin{equation*}
    \varepsilon \in \{+,-\},
	\quad
	\delta \in \{c,d\},
	\quad
	\star \in \{+,-,c,d,\textrm{arc}
	\},
	\quad
	i<j<k<\ell
\end{equation*}
\begin{equation*}
\begin{array}{|c|c|c|c|c|c|}
\hline 
\textrm{Case (i)}\\
\hline \hline
\varepsilon\mapsto\delta\\
\hline \hline
i_{\varepsilon}\rightsquigarrow i_{\delta}
\\
\hline
\end{array}
\quad
\begin{array}{|c|c|c|c|c|c|}
\hline 
\textrm{Case (ii)}\\
\hline \hline
\begin{array}{c}
(c,\star)\mapsto(\star,c)\\
(\star,d)\mapsto(d,\star)
\end{array}
\\
\hline \hline
% \begin{cases}
\xymatrix@R=5pt@C=5pt@M=2pt{
i_{c} & j_{\star}  \\
}
\rightsquigarrow
\xymatrix@R=5pt@C=5pt@M=2pt{
i_{\star}  & j_{c} \\
}\\ \hline
\xymatrix@R=5pt@C=5pt@M=2pt{
i_{\star} & j_{d} \\
}
\rightsquigarrow
\xymatrix@R=5pt@C=5pt@M=2pt{
i_{d} & j_{\star} \\
}
% \end{cases}
\\
\hline
\end{array}
\quad
\begin{array}{|c|c|c|c|c|c|}
\hline 
\textrm{Case (iii)}\\
\hline \hline
(\pm,\mp)\mapsto\textrm{arc}\\
\hline \hline
%\vcenter{
\begin{cases}
\xymatrix@R=5pt@C=5pt@M=2pt{
i_{+} & j_{-}
}\\
\xymatrix@R=5pt@C=5pt@M=2pt{
i_{-} & j_{+}
}
\end{cases}	
\rightsquigarrow
\xymatrix@R=5pt@C=5pt@M=2pt{
i\ar@{-}@/^10pt/[r] & j \\
}
\\
\hline
\end{array}
\quad
\begin{array}{|c|c|c|c|c|c|}
\hline 
\textrm{Case (iv)}\\
\hline \hline
\textrm{arc }\mapsto(\varepsilon,c),(d,\varepsilon)\\
\hline \hline
\xymatrix@R=5pt@C=5pt@M=2pt{
i\ar@{-}@/^10pt/[r] & j \\
}
\rightsquigarrow
\begin{cases}
\xymatrix@R=5pt@C=5pt@M=2pt{
i_{\varepsilon} & j_{c}
}\\
\xymatrix@R=5pt@C=5pt@M=2pt{
i_{d} & j_{\varepsilon}
}
\end{cases}
\\
\hline
\end{array}
\end{equation*}
\begin{equation*}
\begin{array}{|c|c|c|c|c|c|}
\hline 
\textrm{Case (v)}\\
\hline \hline
\begin{array}{c}
(c,\textrm{arc})\mapsto(\pm,\mp,c)\\
(\textrm{arc},d)\mapsto(d,\pm,\mp)
\end{array}
\\
\hline \hline
% \begin{cases}
\xymatrix@R=5pt@C=5pt@M=2pt{
i_{c}&j\ar@{-}@/^10pt/[r] & k \\
}
\rightsquigarrow
\begin{cases}
\xymatrix@R=5pt@C=5pt@M=2pt{
i_{+}  & j_{-} & k_{c} \\
} \\
\xymatrix@R=5pt@C=5pt@M=2pt{
i_{-}  & j_{+} & k_{c} \\
} 
\end{cases}
\\
\hline
\xymatrix@R=5pt@C=5pt@M=2pt{
i\ar@{-}@/^10pt/[r] & j &k_{d} \\
}
\rightsquigarrow
\begin{cases}
\xymatrix@R=5pt@C=5pt@M=2pt{
i_{d}  & j_{+} & k_{-} \\
} \\
\xymatrix@R=5pt@C=5pt@M=2pt{
i_{d}  & j_{-} & k_{+} \\
} 
\end{cases}
% \end{cases}
\\
\hline
\end{array}\quad
\begin{array}{|c|c|c|c|c|c|}
\hline 
\textrm{Case (vi)}\\
\hline \hline
\begin{array}{c}
(\varepsilon ,\textrm{arc})\mapsto(\textrm{arc},\varepsilon)\\
(\textrm{arc},\varepsilon)\mapsto(\varepsilon ,\textrm{arc})
\end{array}
\\
\hline \hline
% \begin{cases}
\xymatrix@R=5pt@C=5pt@M=2pt{
i_{\varepsilon}&j\ar@{-}@/^10pt/[r] & k 
%%\vphantom{\Big|}
\\
}
\rightsquigarrow
\xymatrix@R=5pt@C=5pt@M=2pt{
i\ar@{-}@/^10pt/[rr]  & j_{\varepsilon} & k \rule[10pt]{0pt}{0.4ex}\\
} \\ \hline
% \xymatrix@R=5pt@C=5pt@M=2pt{
% i_{-}&j\ar@{-}@/^10pt/[r] & k \\
% }
% \rightsquigarrow
% \xymatrix@R=5pt@C=5pt@M=2pt{
% i\ar@{-}@/^10pt/[rr]  & j_{-} & k \rule[10pt]{0pt}{0.4ex} \\
% }
% % \end{cases}
% \\
% \hline
% \begin{cases}
\xymatrix@R=5pt@C=5pt@M=2pt{
i\ar@{-}@/^10pt/[r]&j & k_{\varepsilon} \rule[10pt]{0pt}{0.4ex} \\
}
\rightsquigarrow
\xymatrix@R=5pt@C=5pt@M=2pt{
i\ar@{-}@/^10pt/[rr]  & j_{\varepsilon} & k \\
} \\ \hline
% \xymatrix@R=5pt@C=5pt@M=2pt{
% i\ar@{-}@/^10pt/[r]&j & k_{-} \rule[10pt]{0pt}{0.4ex}\\
% }
% \rightsquigarrow
% \xymatrix@R=5pt@C=5pt@M=2pt{
% i\ar@{-}@/^10pt/[rr]  & j_{-} & k \\
% }
% % \end{cases}
% \\
% \hline
\end{array}
\quad
\begin{array}{|c|c|c|c|c|c|}
\hline 
\textrm{Case (vii)}\\
\hline \hline 
(\textrm{arc, }\textrm{arc})\mapsto
(\textrm{arc, }\textrm{arc})\\
\hline \hline 
% \begin{cases}
\phantom{we}\vspace{-8pt}\\
\xymatrix@R=5pt@C=5pt@M=2pt{
i\ar@{-}@/^10pt/[r] & j & 
k\ar@{-}@/^10pt/[r] & \ell \\
}
\rightsquigarrow
\xymatrix@R=5pt@C=5pt@M=2pt{
i\ar@{-}@/^10pt/[rr] & j\ar@{-}@/^10pt/[rr] & 
k & \ell \\
}
\\ \hline
\phantom{we}\vspace{-8pt}\\
\xymatrix@R=5pt@C=5pt@M=2pt{
i\ar@{-}@/^10pt/[rr] & j\ar@{-}@/^10pt/[rr] & 
k & \ell 
}
\rightsquigarrow
\xymatrix@R=5pt@C=5pt@M=2pt{
i\ar@{-}@/^10pt/[rrr] & j\ar@{-}@/^6pt/[r] & 
k & \ell \\
}
% \end{cases}
\\
\hline
\end{array}
\end{equation*}
%
% \begin{equation*}
% 
% \end{equation*}
%
\end{remark}

\begin{example}
Closure relations %%of the orbits 
in {$B_{2}(\R)\backslash\Sp_{4}(\mathbb{R})/P_{S}^{\Sigma}$} and {$B_{2}(\C)\backslash\Sp_{4}(\mathbb{C})/P_{S}^{\C}$}.

In the diagram on the left, we write $11$ to indicate the arc 
$
\xymatrix@R=5pt@C=1pt{
1\ar@{-}@/^10pt/[r] & 2 \\
}$, 
and $c+$ instead of writing $ 1_c \; 2_+ \in \Gamma(2) \simeq B_{2}(\R)\backslash\Sp_{4}(\mathbb{R})/P_{S}^{\Sigma}$ (given in Theorem \ref{theorem-H-G-P_G=Gamma(n)}).  
The number of the leftmost side indicates the dimension of the corresponding $B_{2}(\R)$-orbits on $\Sp_{4}(\mathbb{R})/P_{S}^{\Sigma}$.
We note that $B_{2}(\C)\backslash\Sp_{4}(\mathbb{C})/P_{S}^{\C}$ is parametrized by the set obtained by changing the ``$+$'' and ``$-$'' of $\Gamma(2)$ to ``$\maru$'' (see \eqref{eq-comm-digagram-Omega-regR} for the detail) and use such notation in the diagram on the right.
\begin{equation*}
\begin{array}{c}
B_{2}(\R)\backslash\Sp_{4}(\mathbb{R})/P_{S}^{\Sigma}
\\[2ex]
\xymatrix{
3&
++\ar[d]\ar[drrr]& 
-+\ar[d]\ar[dr]\ar[drr]&  & 
+-\ar[dlll]\ar[dl]\ar[dr] & -- \ar[d]\ar[dlll]\\2&
+d\ar[d]\ar[drr] & -d\ar[d]\ar[dr] &
11\ar[drr]\ar[dr]\ar[dl]\ar[dll]& c+\ar[dl] \ar[d]& 
c-\ar[d]\ar[dll] \\1&
d+\ar[dr]\ar[drr]& d-\ar[dr]\ar[d] &cd\ar[d]& +c\ar[d]\ar[dl] & -c\ar[dll]\ar[dl]\\0&
&dd & dc & cc &
}%%\qquad
\end{array}
\qquad
\begin{array}{c}
B_{2}(\C)\backslash\Sp_{4}(\mathbb{C})/P_{S}^{\C}
\\[2ex]
\xymatrix{
   &  {\maru}{\maru}\ar[d]\ar[dl]\ar[dr]  &    \\
{\maru}d\ar[d]\ar[dr] &  11\ar[dr]\ar[dl]  & c{\maru}\ar[d]\ar[dl] \\
d{\maru}\ar[d]\ar[dr] &  cd\ar[d]  & {\maru}c\ar[d]\ar[dl]  \\
dd &  dc  & cc
}	
\end{array}
\end{equation*}
\end{example}

\begin{example}{Closure relations in $B_{2}(\C)\backslash\U(2,2)/P_{S}^{J}$.} 
\begin{equation*}
\xymatrix{
4&
++\ar[d]\ar[drrr]& 
-+\ar[d]\ar[dr]\ar[drr]&  & 
+-\ar[dlll]\ar[dl]\ar[dr] & -- \ar[d]\ar[dlll]\\3&
+d\ar[dd]\ar[drr] & -d\ar[dd]\ar[dr] &
11\ar[ddrr]\ar[ddr]\ar[ddl]\ar[ddll]& c+\ar[dl] \ar[dd]& 
c-\ar[dd]\ar[dll] \\2&
&&cd\ar[dd]&&\\1&
d+\ar[dr]\ar[drr]& d-\ar[dr]\ar[d] && +c\ar[d]\ar[dl] & -c\ar[dll]\ar[dl]\\0&
&dd & dc & cc &
}
\end{equation*}
In this diagram, we also write $11$ for $
\xymatrix@R=1pt@C=1pt{
1\ar@{-}@/^10pt/[r] & 2 \\
}$, and $c+$ for $1_{c}\:2_{+}\in \Gamma(2) \simeq B_{2}(\C)\backslash\U(2,2)/P_{S}^{J}$
% and 
% $\xymatrix@R=5pt@C=5pt{
% 1_{c} & 2_{+} \\
% } \in \Gamma(2) \simeq B_{2}(\C)\backslash\U(2,2)/P_{S}^{J}$ 
(given in Theorem \ref{theorem-H-G-P_G=Gamma(n)}), respectively,
and the number of the leftmost side indicates the dimension of the corresponding $B_{2}(\C)$-orbits on $\U(2,2)/P_{S}^{J}$ as in the previous diagram.
\end{example}

\begin{example}
Closure relations in {$B_{3}(\C)\backslash\Sp_{6}(\mathbb{C})/P_{S}^{\C}$} and {$B_{3}(\R)\backslash\Sp_{6}(\mathbb{R})/P_{S}^{\Sigma}$}. 
\medskip
\begin{equation*}
\begin{array}{c}
B_{3}(\C)\backslash\Sp_{6}(\mathbb{C})/P_{S}^{\C}
\\[-1ex]
\xymatrix@C=10pt{
&&&&&&&{\maru}{\maru}{\maru}\ar[dll]\ar[dl]\ar[d]\ar[dr]&\\
&&&&&{\maru}{\maru}d\ar[dlll]\ar[dl]\ar[d]&11{\maru}\ar[dll]\ar[drr]\ar[d]&{\maru}11\ar[dlllll]\ar[dl]\ar[d]&c{\maru}{\maru}\ar[dlll]\ar[dl]\ar[d]\\
&&{\maru}d{\maru}\ar[dl]\ar[d]\ar[dr]\ar[drr]&&11d\ar[dll]\ar[dl]\ar[dr]&c{\maru}d\ar[dl]\ar[d]&1{\maru}1\ar[dllll]\ar[dlll]\ar[d]\ar[drr]&c11\ar[dlll]\ar[dl]\ar[dr]&{\maru}c{\maru}\ar[dlll]\ar[dll]\ar[dl]\ar[d]\\
&{\maru}dd\ar[dl]\ar[d]&d{\maru}{\maru}\ar[dll]\ar[d]\ar[dr]&1d1\ar[dl]\ar[dr]&cd{\maru}\ar[dlll]\ar[dl]\ar[d]&{\maru}cd\ar[dll]\ar[dl]\ar[d]&1c1\ar[dlll]\ar[d]&cc{\maru}\ar[dll]\ar[d]&{\maru}{\maru}c\ar[dllll]\ar[dll]\ar[dl]\\
d{\maru}d\ar[d]\ar[dr]&cdd\ar[d]&d11\ar[dll]\ar[d]&dc{\maru}\ar[dll]\ar[dl]&{\maru}dc\ar[dll]\ar[d]&ccd\ar[dl]&11c\ar[dllll]\ar[d]&c{\maru}c\ar[dlll]\ar[dl]\\
dd{\maru}\ar[d]\ar[dr]&dcd\ar[d]&d{\maru}c\ar[dl]\ar[d]&&cdc\ar[dll]&&{\maru}cc\ar[dllll]\ar[d]\\
ddd&ddc&dcc&&&&ccc}	
\end{array}
\end{equation*}
In this diagram, the dimension is constant along each horizontal row, and increases by one from bottom to top.
For example, the dimensions of $B_{3}(\C)$-orbits on $\Sp_{6}(\mathbb{C})/P_{S}^{\C}$ corresponding to $ddd$, $d{\maru}d$ and ${\maru}{\maru}{\maru}$ are 
0, 2, and 6, respectively.

\noindent
{$B_{3}(\R)\backslash\Sp_{6}(\mathbb{R})/P_{S}^{\Sigma}$}
\nopagebreak
\begin{equation*}
\resizebox{\linewidth}{!}{
\xymatrix@R=120pt@C=1pt{
&&&&&&&&&&
+++\ar[dllll]\ar[drrrr]&
++-\ar[dlllll]\ar[dr]\ar[drrrr]&
+-+\ar[dlllll]\ar[dll]\ar[d]\ar[drrrr]&
-++\ar[dlllll]\ar[dlll]\ar[dr]&
+--\ar[dlllllll]\ar[dlll]\ar[drrr]&
-+-\ar[dlllllll]\ar[dllll]\ar[dll]\ar[d]&
--+\ar[dlllllll]\ar[dlll]\ar[d]&
---\ar[dllllllll]\ar[d]
\\
&&&&&&
++d\ar[dll]\ar[drrr]&
+-d\ar[dll]\ar[dr]\ar[drrr]&
-+d\ar[dll]\ar[d]\ar[dr]&
--d\ar[dll]\ar[dr]&
11+\ar[dll]\ar[dr]\ar[drrrr]\ar[drrrrrr]&
11-\ar[dlll]\ar[dr]\ar[drrrr]\ar[drrrrrr]&
+11\ar[dllllllll]\ar[dlllllll]\ar[dl]\ar[dr]&
-11\ar[dlllllll]\ar[dllllll]\ar[dl]\ar[d]&
c++\ar[dlllll]\ar[d]&
c+-\ar[dllllll]\ar[dll]\ar[d]&
c-+\ar[dllllll]\ar[dlll]\ar[d]&
c--\ar[dlllllll]\ar[d]
\\
&&&&
+d+\ar[dllll]\ar[dll]\ar[drrr]&
+d-\ar[dlllll]\ar[dll]\ar[dr]\ar[drrr]&
-d+\ar[dlllll]\ar[dll]\ar[d]\ar[dr]&
-d-\ar[dllllll]\ar[dll]\ar[dr]&
11d\ar[dlllll]\ar[dllll]\ar[dll]\ar[dr]\ar[drr]&
c+d\ar[dll]\ar[d]&
c-d\ar[dll]\ar[d]&
1+1\ar[dlllllllll]\ar[dllllllll]\ar[dlllll]\ar[d]\ar[drrr]\ar[drrrrr]&
1-1\ar[dllllllll]\ar[dlllllll]\ar[dllllll]\ar[dl]\ar[drrr]\ar[drrrrr]&
c11\ar[dllllll]\ar[dlllll]\ar[dll]\ar[drr]\ar[drrr]&
+c+\ar[dlllll]\ar[dll]\ar[d]&
+c-\ar[dllllll]\ar[dllll]\ar[dll]\ar[d]&
-c+\ar[dllllll]\ar[dlllll]\ar[dllll]\ar[d]&
-c-\ar[dlllllll]\ar[dllll]\ar[d]
\\
+dd\ar[d]\ar[drr]&
-dd\ar[d]\ar[dr]&
d++\ar[dll]\ar[drr]&
d+-\ar[dlll]\ar[d]\ar[drr]&
d-+\ar[dlll]\ar[dl]\ar[d]&
d--\ar[dllll]\ar[d]&
1d1\ar[dlll]\ar[d]\ar[dr]&
cd+\ar[dlllll]\ar[dlll]\ar[dl]&
cd-\ar[dllllll]\ar[dlll]\ar[dl]&
+cd\ar[dlllll]\ar[dlll]\ar[dl]&
-cd\ar[dlllll]\ar[dlll]\ar[dll]&
1c1\ar[dlllllll]\ar[dllllll]\ar[dll]&
cc+\ar[dllll]\ar[dll]&
cc-\ar[dlllll]\ar[dll]&
++c\ar[dllllllll]\ar[dllll]&
+-c\ar[dlllllllll]\ar[dllllll]\ar[dllll]&
-+c\ar[dlllllllll]\ar[dlllllll]\ar[dllllll]&
--c\ar[dllllllllll]\ar[dllllll]
\\
d+d\ar[d]\ar[drr]&
d-d\ar[d]\ar[dr]&
cdd\ar[d]&
d11\ar[dlll]\ar[dll]\ar[d]\ar[dr]&
dc+\ar[dll]\ar[dl]&
dc-\ar[dlll]\ar[dl]&
+dc\ar[dlll]\ar[dl]&
-dc\ar[dlll]\ar[dll]&
ccd\ar[dlll]&
11c\ar[dllllll]\ar[dlllll]\ar[dlll]\ar[dll]&
c+c\ar[dlllll]\ar[dllll]&
c-c\ar[dllllll]\ar[dllll]
\\
dd+\ar[dr]\ar[drrr]&
dd-\ar[d]\ar[drr]&
dcd\ar[dr]&
d+c\ar[d]\ar[drr]&
d-c\ar[dl]\ar[dr]&
cdc\ar[d]&
+cc\ar[dl]\ar[dr]&
-cc\ar[dll]\ar[d]
\\
&ddd&&ddc&&dcc&&ccc
}	}
\end{equation*}
In this diagram, the dimension is also constant along each horizontal row, and increases by one from bottom to top.
For example, the dimensions of $B_{3}(\R)$-orbits on $\Sp_{6}(\mathbb{R})/P_{S}^{\Sigma}$ corresponding to $ddd$, $d+d$ and $+++$ are 
0, 2, and 6, respectively.
\end{example}

By examining these examples, it seems that there are no other closure relations than the combinations of those stated in Theorem \ref{Thm-closure-relation}.  

\begin{remark}
By Theorem \ref{theorem-DFV-contain-KGB}, the closure relation of orbits in $B_{H} \backslash G / P_{G}$ contains that of 
$B_{p+q}(\C)  \backslash  \GL_{p+q}(\C) / \U(p,q)$ and 
$B_{p+q}(\R)  \backslash  \GL_{p+q}(\R) / \Ogroup(p,q)$ in Case \caseA and Case \caseB, respectively.
Moreover, in Case \caseA, this is equivalent to that of the $B_{p+q}(\C)$-orbits on the symmetric space $\GL_{p+q}(\C) / \GL_{p}(\C)\times \GL_{q}(\C)$ over the complex numbers $ \C $ by the Matsuki duality (Theorem \ref{theorem-Matsuki-Duality-Intro}).
For details on the closure relation of Borel orbits on symmetric spaces over an algebraically closed field of characteristic zero, see \cite{Richardson.Springer.1990}.  
\end{remark}

\subsection{Future works}
It seems that the orbit decomposition of the real double flag variety has a number of applications. 
One of them is an application to the study of the intertwining operators between (degenerate) principal series representations, namely, \textit{symmetry breaking operators}.
In fact, there exist observations that the orbit decomposition $B_{H} \backslash G /P_{G}$ can be regarded as a ``parameter space'' of the intertwining operators from the (degenerate) principal series representation on $G/P_{G}$ to that on $H/B_{H}$ 
(See \cite{Kobayashi.Speh.2015,Kobayashi.Speh.2018} and \cite{Nishiyama.Orsted.2018}, for example).
In this respect, we are preparing another article, in which we construct an integral operator corresponding to each open $B_{H}$-orbit on $G/P_{G}$. The results will appear soon.

Besides the parametrization of the orbits, 
closure ordering among $ H $-orbits and their geometric properties (such as singularities of the closure) are to be explored.  
Since, in our case, $ G /P_{G} $ is a \emph{real} spherical $ H $-variety, the harmonic analysis on it will be fruitful.  
See \cite{Delorme.Knop.Kroetz.Schrichtkrull.2021,Kroetz.Schrichtkrull.2017} and \cite{Knop.CMH1995} for example.

\subsection{Brief summary of sections}
Let us explain each section briefly.  
In Section \ref{Section-Preliminary}, we fix some notations about the symmetric pairs $(\U(n,n), \GL_{n}(\C))$ and $(\Sp_{2n}(\R),\GL_{n}(\R))$.
In Section \ref{Section-Orbit-decomposition},
we reduce the problem about the orbit decomposition $H\backslash (H/B_{H}\times G/P_{G})$ of the real double flag variety to the classification of $B_{H}$-orbits on the Lagrangian Grassmannians, and then further reduce 
it to the space of Hermitian (or, symmetric) matrices.
% In Section \ref{Section-Orbit-decomp-B-backslash-Her},
% we give a proof of the classification of $B_{H}$-orbit on the space of Hermitian (or, symmetric) matrices.
In Section \ref{Section-Gaussian-elimination}, 
the classification of $B_{H}$-orbits on the space of Hermitian (or, symmetric) matrices is given by signed partial involutions and its proof is based on the Gaussian elimination.
In Section \ref{Section-BGP-via-SPP}, we give a proof of Theorem \ref{theorem-H-G-P_G=Gamma(n)},
which gives a combinatorial description of the orbit decomposition $H\backslash (H/B_{H} \times G/P_{G})$.
The classification of orbits can be interpreted in combinatorial way using certain graphs, which are related to clans.  This is explained in Section \ref{Section:orbit-graphs}.  

In Section \ref{Section-proof-Thm-Galios-RDFV}, we prove Theorem \ref{theorem-Galois-RDFV} using the theory of Galois cohomology.
The Matsuki duality and the Matsuki-Oshima's clans are the key features in Section \ref{Section-Matsuki-duality}, in which we prove Theorem \ref{theorem-DFV-contain-KGB}, and we also give another proof of the classification of $B_{H}$-orbits on the space of Hermitian (or, symmetric) matrices.
Theorem \ref{Thm-closure-relation} (closure relation of orbits) is proved in Section 
\ref{section:proof.closure.relations}.  

We also give a proof of a lemma of fiber-wise decomposition of orbits 
in Appendix \ref{Section-proof-of-Lem-orbit-decomp-fiber-bundle}.
Finally, we give a proof of Proposition \ref{fact-real-orbits-via-Galois-cohomology} and also
discuss the properties of real points of double coset decomposition
in Appendix \ref{Section-real-points-of-double-coset-decomposition}.

\section*{Acknowledgments}
The authors thank the anonymous referee for useful suggestions.  
The first author was partially supported by Grant-in-Aid for Scientific Research C (21K03184 and 25K06938), Japan Society for the Promotion of Science (JSPS). 
The second author was partially supported by Grant-in-Aid for Early-Career Scientists (25K17259), Japan Society for the Promotion of Science (JSPS). 

\section{Preliminary}
\label{Section-Preliminary}

\skipover{If we use Section \ref{Section-proof-Thm-Galios-RDFV}, the notation ``$\tau$'' of the involution in this section should be replaced with ``$\gamma$''.}

In this section, we fix some notations about the indefinite unitary group $\U(n,n)$ and 
the symplectic group  $\Sp_{2n}(\R)$.  
%% in Sections \ref{Section-Preliminary-caseA} and \ref{Section-Preliminary-caseB}, respectively. 

\subsection{Case \caseA: the indefinite unitary group \texorpdfstring{$\U(n,n)$}{U(n,n)}}\label{Section-Preliminary-caseA}
%%In this section, we fix some notations about the indefinite unitary group $\U(n,n)$, namely, in Case \caseA of Setting \ref{Setting-DFVR}.
Let $ \C^{(n|n)} $ be a complex vector space of dimension $ 2 n $ endowed with an indefinite Hermitian inner product of signature $ (n,n) $  
given by 
\begin{equation}
\label{eq:Hermitian.form.and.Jn}
    \LL u, v \RR = u^* J_n v ,  \quad \text{where } \quad J_{n}\coloneqq\sqrt{-1}
    \begin{pmatrix}
        0 & \unitmatrix_{n} \\
        -\unitmatrix_{n} & 0
    \end{pmatrix}, 
\end{equation}
where we write $ A^{*} = \transpose{\conjugate{A}} $.  
The indefinite unitary group $\U(n,n)$ preserving this inner product is given by
\begin{equation*}
    G  \coloneqq \U(n,n) 
    \coloneqq
    \left\{
        g\in \GL_{2n}(\C)
    \mid 
        g^{*}J_{n}g=J_{n}
    \right\}.
\end{equation*}
Define a maximally isotropic subspace $W_{0}\subset \C^{(n|n)}$ by
\begin{equation*}
    W_{0} \coloneqq \C^{n} \oplus \{0\} \subset \C^{n} \oplus \C^{n} = \C^{(n|n)}.
\end{equation*}
We call a maximally isotropic subspace as a Hermitian Lagrangian subspace.  
The Siegel parabolic subgroup $P_{S}^{J} \subset G = \U(n,n)$ is defined to be the stabilizer of $W_{0}$:  
\begin{equation*} 
    P_{S}^{J} \coloneqq \Stab_{G}(W_{0}) =
    \left\{ 
        \begin{pmatrix}
            a & u \\
            0 & (a^*)^{-1}
        \end{pmatrix} 
    \Bigm|  
        \begin{array}{c}
            a \in \GL_n(\C)   \\
            a^{-1} u \in \Her_n(\C)  
        \end{array}
    \right\} ,
\end{equation*}
where $ \Her_n(\C) $ denotes the space of Hermitian matrices of degree $ n $.  
Since $ G $ acts on the Hermitian Lagrangian subspaces transitively by Witt theorem, 
$ G/P_{S}^{J} $ is naturally identified with the set of Hermitian Lagrangian subspaces in $ \C^{(n|n)} $:
\begin{equation}\label{eq-G/P_S(C)=HLGr}
    G/P_{S}^{J} \simeq \HLGr(\C^{2n}) \coloneqq \{ W\in \Gr_{n}(\C^{2n}) \mid W \text{ is maximally isotropic in } \C^{(n|n)} \}.
\end{equation}
Let $H \subset G = \U(n,n)$ be a symmetric subgroup defined by
\begin{equation}\label{eq-def-H-and-h}
    H \coloneqq 
    \{ 
        h(a) \coloneqq \diag(a, (a^{*})^{-1}) 
    \mid 
        a \in \GL_{n}(\C) 
    \} 
    \simeq \GL_{n}(\C) , 
\end{equation}
which is fixed by an involution $ \sigma(g) = I_{n,n} g I_{n,n} $, 
where $ I_{n,n} = \diag(\unitmatrix_{n}, -\unitmatrix_{n}) $.  
Write $B_{H}$ for the Borel subgroup of $ H $ given by
\begin{equation}\label{eq-def-B_H-caseA}
    B_{H} \coloneqq 
    \{ 
        h(b) = \diag(b, (b^{*})^{-1}) 
    \mid 
        b \in B_{n}(\C) 
    \} 
    \simeq B_{n}(\C),
\end{equation}
where $B_{n}(\C)$ is the subgroup of upper triangular matrices in $\GL_{n}(\C)$.

\subsubsection{Complexification in Case \caseA}

We write
\begin{equation*}
	G_{\C}\coloneqq
	\GL_{2n}(\C)
\end{equation*}
for the complexification of $ G = \U(n,n)$, 
and define an involution $\gamma$ on $G_{\C}$ by
\begin{equation}\label{eq-def-gamma-caseA}
	\gamma(g)
	\coloneqq
	J_{n}^{-1} (g^{*})^{-1} J_{n},
\end{equation}
where $J_{n}$ is defined in \eqref{eq:Hermitian.form.and.Jn}.  
The involution $\gamma \colon G_{\C} \to G_{\C}$ gives an action of the Galois group $\Gal(\C/\R)$ on $G_{\C}$, and by which we regard $G_{\C}$ as an algebraic group defined over $\R$.
Then, we get 
$
	G
	=
	G_{\C}^{\gamma}
    \coloneqq 
    \{ g \in G_{\C} \mid \gamma(g)=g\}
$, 
%%by definition of $ G $, 
and thus $ G $ is a real form of $G_{\C}$.  
For the complexification of $ H $, we can take 
%%We also define a symmetric subgroup $H_{\C}\simeq \GL_{n}(\C) \times \GL_{n}(\C) $ of $G_{\C}$ by
\begin{equation*}
	H_{\C}\coloneqq
	\left\{
	   \diag(g_{1},g_{2})
	\mid  
	   g_{1},g_{2}\in \GL_{n}(\C)
	\right\}.
% \label{eq-def-H-C}
\end{equation*}
For the complexification of $ P_{G} $ and $ B_{H} $, 
we can take a $\gamma$-stable maximal parabolic subgroup $P_{G_{\C}}$ of $G_{\C}$,
and a Borel subgroup $B_{H_{\C}}$ of $H_{\C}$ in the following respectively.
\begin{equation*}
	P_{G_{\C}} \coloneqq
	\left\{
	   \begin{pmatrix}
	       g_{1} & u \\
	       0 & g_{2}
	   \end{pmatrix}
	\Bigm| 
	   \begin{array}{c}
	       g_{1},g_{2} \in \GL_{n}(\C) \\
	       u \in \Mat_{n}(\C)
	   \end{array}
	\right\},
	\qquad
	B_{H_{\C}} \coloneqq
	\left\{
	   \begin{pmatrix}
	       b & 0\\
           0 & c
        \end{pmatrix}
	\Bigm| 
	   \begin{array}{cc}
	       b\in B_{n}^{+} \\
	       c\in B_{n}^{-}
	   \end{array}
	\right\},
% \label{eq-def-P_C-B_C}
\end{equation*}
where $B_{n}^{+}$ (resp.~$B_{n}^{-}$) is the subgroup of $\GL_{n}(\C)$ of upper (resp.~lower) triangular matrices.
\skipover{
Then, we have
\begin{equation*}
	G_{\C}^{\gamma}=G,
	\quad
	H_{\C}^{\gamma}=H,
	\quad
	P_{G_{\C}}^{\gamma}=P_{S}^{J},
	\quad
	B_{H_{\C}}^{\gamma}=B_{H},
	% \label{eq-def-G-H-P-B}
\end{equation*}
where the groups on the right-hand sides are real forms defined in Section \ref{Section-Preliminary-caseA}.
}

\subsection{Case \caseB: the symplectic group \texorpdfstring{$ \Sp_{2n}(\R) $}{Sp(R)}}\label{Section-Preliminary-caseB}
Let us consider the symplectic group $\Sp_{2n}(\R)$, namely, in Case \caseB of Setting \ref{Setting-DFVR}.

Let $ \R^{2n}_{\Sigma} $ be a vector space of dimension $ 2 n $ endowed with a symplectic form given by 
\begin{equation*}
    \LL u, v \RR = {}^{t}u \Sigma_{n} v , \quad  \text{where }  \quad
    \Sigma_{n}\coloneqq
    \begin{pmatrix}
        0 & \unitmatrix_{n} \\
        -\unitmatrix_{n} & 0
    \end{pmatrix}.
\end{equation*}
The symplectic group $\Sp_{2n}(\R)$ preserving this inner product is given by
\begin{equation*}
    G  \coloneqq \Sp_{2n}(\R) 
    \coloneqq
    \left\{
        g\in \GL_{2n}(\R)
    \Bigm|
        \transpose{g} \Sigma_{n}g=\Sigma_{n}
    \right\} .
\end{equation*}
Let us choose a Lagrangian subspace (i.e., a maximal isotropic subspace)
\begin{equation*}
    W_{0} \coloneqq \R^{n} \oplus \{0\} \subset \R^{n} \oplus \R^{n} = \R^{2n}_{\Sigma}.
\end{equation*}
Our Siegel parabolic subgroup $P_{S}^{\Sigma} \subset G = \Sp_{2n}(\R)$ is the stabilizer of $ W_0 $:
\begin{equation*}
    P_{S}^{\Sigma} \coloneqq \Stab_{G}(W_{0}) =
    \left\{ 
        \begin{pmatrix}
            a & u \\
            0 & {}^{t} a^{-1}
        \end{pmatrix} 
    \Bigm|  
        \begin{array}{c}
            a \in \GL_{n}(\R)   \\
            a^{-1} u \in \Sym_{n}(\R) 
        \end{array}
    \right\}, 
\end{equation*}
so that we have an isomorphism
    \begin{equation*}
        \Sp_{n}(\R) / P_{S}^{\Sigma} \simeq \LGr(\R^{2n})
        \coloneqq
        \{ W \in \Gr_{n}(\R^{2n}) \mid W \text{ is maximally isotropic in }\R^{2n}_{\Sigma} \}.
    \end{equation*}
We also write $H \subset G = \Sp_{2n}(\R)$ for the symmetric subgroup defined by
\begin{equation*}
    H \coloneqq 
    \{ 
        h(a) \coloneqq \diag(a, {}^{t}a^{-1}) 
    \mid 
        a \in \GL_{n}(\R) 
    \} 
    \simeq \GL_{n}(\R) 
\end{equation*}
and write $B_{H}$ for its Borel subgroup given by
\begin{equation*}
    B_{H} \coloneqq 
    \{ 
        h(b) = \diag(b, {}^{t}b^{-1}) 
    \mid 
        b \in B_{n}(\R) 
    \} 
    \simeq B_{n}(\R),
\end{equation*}
where $B_{n}(\R)$ is the group of upper triangular matrices of $\GL_{n}(\R)$.

\subsubsection{Complexification in Case \caseB}

We write
\begin{equation*}
	G_{\C} \coloneqq \Sp_{2n}(\C)
\end{equation*}
for the complexification of $ G = \Sp_{2n}(\R)$, and define an involution $\gamma$ on $G_{\C}$ by
\begin{equation*}
	\gamma(g) \coloneqq \overline{g},
\end{equation*}
which we consider as the Galois action of $ \Gal(\C/\R) $ on $ G_{\C} $.  
Then, 
$
	G = G_{\C}^{\gamma} :=\{g\in G_{\C} \mid \gamma(g)=g\}
$
is the real form of $ G_{\C} $.  
For the complexifications of $ H, P_G, B_H $, we can take: 
%%We also define a symmetric subgroup $H_{\C}\simeq \GL_{n}(\C)$ of $G_{\C}$ by
\begin{align*}
	H_{\C}
	&\coloneqq
	\left\{
	   \diag(a,{}^{t}a^{-1})
	\mid
	   a\in \GL_{n}(\C)
	\right\}, 
\\[1ex]
	P_{G_{\C}} &\coloneqq
	\left\{
	   \begin{pmatrix}
	       a & u \\
	       0 & {}^{t}a^{-1}
	   \end{pmatrix}
	\Bigm| 
	   \begin{array}{c}
	       a \in \GL_{n}(\C) \\
	       a^{-1}u \in \Mat_{n}(\C)
	   \end{array}
	\right\}, \qquad
	B_{H_{\C}} \coloneqq
	\left\{
	    \begin{pmatrix}
	        b & 0 \\
            0 & {}^{t}b^{-1}
        \end{pmatrix}
	\Bigm| 
	    b\in B_{n}(\C)
	\right\}.
\end{align*}
\skipover{
Then, we have
\begin{equation*}
	G_{\C}^{\gamma}=G,
	\quad
	H_{\C}^{\gamma}=H,
	\quad
	P_{G_{\C}}^{\gamma}=P_{S}^{\Sigma},
	\quad
	B_{H_{\C}}^{\gamma}=B_{H},
\end{equation*}
where the groups on the right-hand sides are real forms 
defined in Section \ref{Section-Preliminary-caseB}.
}

\section{Reduction of the orbit space}\label{Section-Orbit-decomposition}

In this section, we reduce the whole orbit space $ H\backslash \dblFV \simeq B_{H} \backslash G /P_{G} $ to 
smaller, well known affine spaces.  
This process is literally the same for Cases \caseA and \caseB, 
so we only treat Case \caseA.  
For the unified notation, let us write 
\begin{equation*}
    P_{S} \coloneqq P_{S}^{J} \:\:(= P_{G}),
    \qquad
    B_{n} \coloneqq B_{n}(\C) \:\:(\simeq B_{H}),
\end{equation*}
where the right-hand sides are defined in Section \ref{Section-Preliminary-caseA}.  
We can apply the same argument in this section \textit{mutatis mutandis} to Case \caseB by rereading the symbols according to the list in Notation \ref{Notation-caseA-and-caseB} below.

\begin{notation}\label{Notation-caseA-and-caseB}
The table of corresponding notations for Case \caseA and Case \caseB 
%%of Setting \ref{Setting-DFVR} 
is as follows:
\begin{equation*}
    \begin{array}{|c|c|c|c|c|c|c|c|c|c|} \hline
        &  & G & H & P_{S} & B_{n} &  & a^{*}  
        \\ \hline
        \text{\caseA} & \C & \U(n,n)      & \GL_{n}(\C) & P_{S}^{J}      & B_{n}(\C) & \Her_{n}(\C) & a^{*}  
        \\ \hline
        \text{\caseB} & \R & \Sp_{2n}(\R) & \GL_{n}(\R) & P_{S}^{\Sigma} & B_{n}(\R) & \Sym_{n}(\R) & \transpose{a}  
        \\ \hline
    \end{array}        
\end{equation*}
%%where the groups in the list are defined in Section \ref{Section-Preliminary}.
\end{notation}

The first step toward the orbit decomposition of $B_{H} \backslash G / P_{S}$ is to reduce it to that of $B_{m} \backslash \Her_{m}(\C)$.   
%%(Lemma \ref{Lem-B-G-P_S=coprod-B-Her}).
For this reduction,
we take three steps divided into Lemmas \ref{Lem-BGP=BLMGL}, \ref{Lem-orbit-decomp-B-LM-GL=coprod-B-LMI-GL},  and \ref{Lem-(B-Her)=(B-LM-I/GL)}.  
The first step is to consider a matrix realization of the Grassmannian $G/P_{S}$.

\subsection{Matrix realization of  \texorpdfstring{$G /P_{S}$}{G/P}}

In this subsection, we consider a matrix realization of the Grassmannian $ G/P_{S}$, which is well known.

\begin{definition}\label{Def-LM-circ}
For $n\in \Z_{\geq 0}$, we define $\LregMat_{2n,n} \subset \Mat_{2n,n}(\C)$ by
\begin{equation*}
    \LregMat_{2n,n} \coloneqq	
    \left\{
        \omega_{C,D}=
        \begin{pmatrix}
            C \\ D
        \end{pmatrix}       \in \Mat_{2n,n}(\C)
    \Biggm| 
         \begin{array}{c}
              C,D\in\Mat_{n}(\C) \\
              C^{*}D \in  \Her_{n}(\C) \\
              \rank \omega_{C,D} = n
         \end{array}
    \right\}.
\end{equation*}
\end{definition}

\begin{lemma}\label{Lem-G/P_S=LM/GL}\label{Lem-BGP=BLMGL}
   Let $n\in\Z_{\geq 0}$. Then, we have a $G$-equivariant isomorphism
   \begin{equation}
   \begin{array}{ccc}
        \LregMat_{2n,n} / \GL_{n}(\C) & \xrightarrow{\sim} & G/P_{S}  \\
        \rotatebox{90}{$\in$} && \rotatebox{90}{$\in$} \\
        \omega & \mapsto & [\omega],
   \end{array}
   \label{eq-Lem-G/P_S=LM/GL}
   \end{equation}
   where $[\omega] \in G/P_{S} \simeq \HLGr(\C^{2n})$ is the $n$-dimensional subspace of $\C^{(n|n)}$ defined by the image of $\omega \colon \C^{n} \to \C^{2n}$.  
   As a consequence, we get an isomorphism 
\begin{equation*}
        B_{H} \backslash G/P_{S} \simeq B_{H} \backslash \LregMat_{2n,n} / \GL_{n}(\C).
\end{equation*}
\end{lemma}

\begin{proof}
	Recall that the isomorphism $G/P_{S} \simeq \HLGr(\C^{2n})$ given in the equation \eqref{eq-G/P_S(C)=HLGr}.
	Under this identification, it is easy to see the map is well defined.  
    The rows of $ \omega $ gives a basis of $ [\omega] $, which is determined up to the conjugacy by $ \GL_n(\C) $.
\skipover{
    For the inverse, the map sending an element of $ \HLGr(\C^{2n})$ to the matrix whose columns consisting of the basis of it determines the inverse of the map in \eqref{eq-Lem-G/P_S=LM/GL}.
}
\end{proof}

Following Lemma \ref{Lem-BGP=BLMGL}, we will consider the orbit decomposition $B_{H} \backslash \LregMat_{2n,n} / \GL_{n}(\C)$ instead of $B_{H} \backslash G / P_{S}$.

\subsection{Reduction of \texorpdfstring{$B_{H}\backslash\mathrm{LM}^{\circ}_{2n,n}/\GL_{n}(\C)$}{the orbit decomposition}}\label{Section-Orbit-decomp-of-B-LM-GL}

In order to consider the double coset space $B_{H}\backslash\mathrm{LM}^{\circ}_{2n,n}/\GL_{n}(\C)$,
we divide $\mathrm{LM}^{\circ}_{2n,n}$ into $(B_{H},\GL_{n}(\C))$-invariant subsets according as the degeneracy of the orbits.
For this purpose, we prepare some terminologies.

%%\Today[2025/01/14]

Let $D\in \mathrm{M}_{n}(\C)$.
Then, by the right translation of an appropriate element $g\in \GL_{n}(\C)$, we transform $D$ into reduced column echelon form:
\begin{equation}
Dg =
\begin{pmatrix}
          & & & & & \\
        1 &  & & & & \\
        * & 1 &  & & & \\
        \vdots & * & \ddots  & & & \\
        \vdots & \vdots & \vdots & 1 &  \\
        * & * & * & * &
    \end{pmatrix}.
    \label{eq-Row-echelon-form}
\end{equation}
For $k\in\{1,2,\dots,\rank D\}$, if the $(i_{k},k)$-entry is the leading entry of the $k$-th column of \eqref{eq-Row-echelon-form}, we define
\begin{equation}
	\ech D \coloneqq
	\{i_{1},i_{2},\dots,i_{\operatorname{rank} D}\}.
\label{eq-def-ech}
\end{equation}
In the case that $D=0 \in \Mat_{n}(\C)$, we define  $\ech D\coloneqq \emptyset$.

\begin{example}
For example, we have
    \begin{equation*}
        \ech
        \begin{pmatrix}
            1 & 0 & 0 & 0 \\
            2 & 0 & 0 & 0 \\
            0 & 1 & 0 & 0 \\
            0 & 0 & 1 & 0
        \end{pmatrix}
        =\{1,3,4\},
    \quad
        \ech
        \begin{pmatrix}
            0 & 0 & 0 & 0 \\
            0 & 0 & 0 & 0 \\
            1 & 0 & 0 & 0 \\
            0 & 0 & 0 & 0
        \end{pmatrix}
        =\{3\},
    \quad
        \ech
        \begin{pmatrix}
            0 & 0 & 0 & 0 \\
            0 & 0 & 0 & 0 \\
            0 & 0 & 0 & 1 \\
            0 & 0 & 0 & 0
        \end{pmatrix}
        =\{3\}.
    \end{equation*}
\end{example}

\begin{definition}\label{Def-LM-I}
For $I\subset [n] = \{1,2,\dots,n\}$, we define a subset $\mathrm{LM}_{2n,n}^{I}$ of $\LregMat_{2n,n}$ by
\begin{equation*}
\mathrm{LM}_{2n,n}^{I} \coloneqq	
\left\{
    \omega_{C,D}=
    \begin{pmatrix}
        C \\ D
    \end{pmatrix}  \in \LregMat_{2n,n}
\Bigm| 
        \ech  D = I
\right\}.
\end{equation*}
\end{definition}

\begin{lemma}\label{Lem-LM=coprod-LM-I}\label{Lem-LM-I-stable}\label{Lem-orbit-decomp-B-LM-GL=coprod-B-LMI-GL}
For $I\subset [n]$, 
the subset $ \mathrm{LM}_{2n,n}^{I} $ is stable under the action of $(B_{H},\GL_{n}(\C))$ and 
we have a disjoint decomposition
\skipover{
\begin{equation*}
    \LregMat_{2n,n} = \coprod_{I\subset[n]} \mathrm{LM}_{2p,p}^{I}.
\end{equation*}
}
\begin{equation*}
B_{H}\backslash \LregMat_{2n,n} / \GL_{n}(\C)
= \coprod_{I\subset[n]} B_{H}\backslash \mathrm{LM}_{2n,n}^{I} / \GL_{n}(\C).
\end{equation*}
\end{lemma}

\begin{proof}
The right hand side is a disjoint union indexed by $ I = \ech(D) $ of $ \omega_{C,D} $, for which each subset 
$\mathrm{LM}_{2n,n}^{I} $ is $ (B_{H},\GL_{n}(\C)) $-stable.  Thus follows the lemma.
\end{proof}

\subsection{Space of Hermitian matrices}\label{section:space.Hermitian.matrices}

By Lemma \ref{Lem-orbit-decomp-B-LM-GL=coprod-B-LMI-GL},
it is sufficient to identify $B_{H}\backslash \mathrm{LM}_{2n,n}^{I}/ \GL_{n}(\C)$ in order to describe 
the double coset space $B_{H}\backslash \LregMat_{2n,n} / \GL_{n}(\C)$.
%%For that purpose, we need a bit of notations.

%%\begin{convention}\label{Convention-sigma}
As usual, we identify an element $\sigma$ of the symmetric group $\Sgroup_{n}$ with a permutation matrix $\sigma \in \GL_{n}(\C)$ determined by
	\begin{equation*}
		\sigma e_{i} = e_{\sigma(i)}  \qquad (1 \leq i \leq n),
	\end{equation*}
	where $ \{ e_{i} \} $ is the standard basis of $ \C^n $ (or $ \R^n $).  
%%\end{convention}

\begin{definition}\label{Def-sigma-I}
For $I=\{i_{1}<i_{2}<\dots<i_{m}\}\subset[n]
%%:=\{1,2,\dots,n\}
$, we define 
% (the inverse of) 
a Grassmannian permutation $\sigma_{I} \in \Sgroup_{n}$ by
\begin{equation*}
    \sigma_{I} \coloneqq 
    \begin{pmatrix}
        1 & 2 & \dots & m & m + 1 & \dots & n
        \\
        i_{1} & i_{2} & \dots & i_{m} & i'_{1} & \dots & i'_{n-m}
    \end{pmatrix},
\end{equation*}
where $\{i'_{1}<i'_{2}<\dots<i'_{n-m}\}\coloneqq[n]\backslash I$.  
The element $ \sigma_I $ is also identified with a permutation matrix, and further, 
by the isomorphism $h\colon \GL_{n}(\C) \to H$, it is considered to be an element in $ H $.
%%In particular, we have $h(\sigma) \in H$ under Convention \ref{Convention-sigma}, where $h\colon \GL_{n}(\C) \to H$ is defined in \eqref{eq-def-H-and-h}.
\end{definition}

\begin{remark}\label{Remark-sigma-I-correspond-coprod}
    The index $I$ in the isomorphism of Theorem \ref{theorem-DFV-contain-KGB} corresponds to $I \in \mathcal{P}([n])$ of Definition \ref{Def-sigma-I}, where $\mathcal{P}([n])$ is the power set. 
    Thus, we can interpret the index $I$ of Theorem \ref{theorem-DFV-contain-KGB} as the representatives of $\mathcal{P}([n]) \simeq  \Z_{2}^{n} \simeq (\Z_{2}^{n} \rtimes \Sgroup_{n})/\Sgroup_{n}$,
    which is the Weyl group corresponding to the flag variety $G/P_{S}$. 
    (Note that the restricted root systems of the real groups $\U(n,n)$ and $\Sp_{2n}(\R)$ are type $C_{n}$ and $\GL_{n}(\C)$ and $\GL_{n}(\R)$ are type $A_{n-1}$).
\end{remark}

Now let us give a description of $B_{H}\backslash \mathrm{LM}_{2n,n}^{I}/ \GL_{n}(\C)$.  
We denote 
\begin{equation*}
\dbinom{[n]}{m} = \{ I \subset [n] \mid \# I = m \} .
\end{equation*}

\begin{lemma}\label{Lem-(B-Her)=(B-LM-I/GL)}
Let $I\subset [n] $ with $m\coloneqq \#I$, i.e., $ I \in \dbinom{[n]}{m} $.
Then, the following map is well-defined and bijective:
\begin{equation}
	\begin{array}{ccc}
		B_{m}\backslash \Her_{m}(\C)
		&\xrightarrow{\sim} &
		B_{H}\backslash \mathrm{LM}_{2n,n}^{I} / \GL_{n}(\C)
		\\
		\rotatebox{90}{$\in$}  && \rotatebox{90}{$\in$}
		\\
		B_{m}z & \mapsto & 
        B_{H} h(\sigma_{I})
		\begin{pmatrix}
			z     & 0 \\
			0     & \unitmatrix_{n-m} \\
			\unitmatrix_{m} & 0 \\
			0     & 0
		\end{pmatrix}  \GL_{n}(\C),
	\end{array}
\label{eq-def-phi_I-bar}
\end{equation}
where the action of $B_{m}$ on $\Her_{m}(\C)$ is given by
	\begin{equation}
		b \cdot z \coloneqq b z b^{*}
	\label{eq-def-B-action-Her}
	\end{equation}
and $h(\sigma_{I}) = \diag(\sigma_{I}, (\sigma_{I}^{*})^{-1}) \in H$.  
%% is defined in Definition \ref{Def-sigma-I}.
\end{lemma}

To prove this lemma, 
%%\ref{Lem-(B-Her)=(B-LM-I/GL)}, 
we need one more lemma.

\begin{lemma}\label{Lem-B-p-cap-GL-times-GL}
For $ I \in \dbinom{[n]}{m} $, we have 
%%Let $I\subset \{1,2,\dots,n\}$ with $m \coloneqq \# I$. Then, we have
%
\begin{equation}\label{eq:BsigmaI.contains.smallerBorels}
    B_{n}^{\sigma_{I}^{-1}} \cap  (\GL_{m}(\C) \times  \GL_{n-m}(\C))
    =
    \left\{
        \begin{pmatrix}
            b_{m} & 0 \\
            0     & b_{n-m}
        \end{pmatrix}
    \Bigm| 
        b_{m} \in B_{m}, b_{n-m} \in B_{n-m}
    \right\},
\end{equation}
where $B_{n}^{\sigma_{I}^{-1}} \coloneqq \sigma_{I}^{-1} B_{n} \sigma_{I}$.  
%%,and $\sigma_{I}\in\Sgroup_{n}$ is defined in Definition \ref{Def-sigma-I}.
\end{lemma}

\begin{proof}
First we prove the Lie algebras of the both hands side agree.  

Take a standard (diagonal) torus $ T $.  
Then the Lie algebras of the both hands side are stable under the conjugation by $ T $.  
Thus they are decomposed into the root spaces with respect to the adjoint action of $ \lie{t} = \Lie T $, and 
we compare the root spaces.  
Define $Z\in \lie{t} $ by
\begin{equation*}
	Z \coloneqq \operatorname{diag}(1,2,\dots,n) \in  \gl_{n}(\C).
\end{equation*}
Then, the root space decomposition of the Lie algebra $\fb_{n}$ of $B_{n}$ is given by
\begin{equation*}
    \fb_{n}
	= \lie{t} \oplus 
\sum_{\alpha(Z) < 0} \lie{g}_{\alpha} , 
\end{equation*}
where $ \alpha $ moves over the roots in $ \Delta(\lie{gl}_n, \lie{t}) $, and 
$ \lie{g}_{\alpha} $ denotes the corresponding root space.  
Moreover, for any $ \alpha \in \Delta(\lie{gl}_n, \lie{t}) $, we have
\begin{equation*}
\alpha(Z) \leq 0 
\iff 
\sigma_{I}^{-1}(\alpha)(\sigma_{I}^{-1} Z \sigma_{I}) \leq 0 
\end{equation*}
which implies
\begin{equation*}
\fb_{n}^{\sigma_{I}^{-1}}  =  \lie{t} \oplus \sum_{\alpha(Z) < 0} \lie{g}_{\alpha}^{\sigma_{I}^{-1}}
= \lie{t} \oplus \sum_{\alpha(Z) < 0} \lie{g}_{\sigma_{I}^{-1}(\alpha)}
= \lie{t} \oplus \sum_{\alpha(\sigma_{I}^{-1} Z \sigma_{I}) < 0} \lie{g}_{\alpha} .
\end{equation*}
Write $I=\{i_{1}<i_{2}<\dots<i_{m}\}$ and
$ [n] \setminus I = \{i'_{1}<i'_{2}<\dots<i'_{n-m}\} $.  
Then, the definition of $\sigma_{I}\in\Sgroup_{n}$  (Definition \ref{Def-sigma-I}) implies
\begin{equation*}
\sigma_{I}^{-1}Z\sigma_{I}
=
\operatorname{diag} (i_{1},i_{2},\dots,i_{m},  i'_{1},i'_{2},\dots,i'_{n-m}) ,
\end{equation*}
and we get
\begin{equation*}
    \fb_{n}^{\sigma_{I}^{-1}} \cap 	
    ( \fg \mathfrak{l}_{m}(\C)  \oplus  \fg \mathfrak{l}_{n-m}(\C) )
    =
    \fb_{m} \oplus \fb_{n-m} .
\end{equation*}
Thus, the left hand side of \eqref{eq:BsigmaI.contains.smallerBorels} contains a Borel subgroup of $ \GL_{m}(\C) \times  \GL_{n-m}(\C) $, and 
hence it is a parabolic subgroup (in fact, it coincides with the Borel subgroup itself).  Since a parabolic subgroup is connected, it coincides with the right hand side.  
\end{proof}

Let us prove Lemma \ref{Lem-(B-Her)=(B-LM-I/GL)}.

\begin{proof}[Proof of Lemma \ref{Lem-(B-Her)=(B-LM-I/GL)}]
We will prove well-definedness, surjectivity, and injectivity of the map \eqref{eq-def-phi_I-bar} in this order.

To prove the well-definedness, 
%% of the map \eqref{eq-def-phi_I-bar}.
pick $b_{m}\in B_{m}$ and $z\in \Her_{m}(\C)$.  
We want to show that the images of $z$ and $b_{m}\cdot z= b_{m}zb_{m}^{*}$ 
by the map \eqref{eq-def-phi_I-bar} are equal.  
The image of $b_{m}\cdot z$ 
%%by the map \eqref{eq-def-phi_I-bar} 
is 
\begin{align*}
	&
	B_{H}h(\sigma_{I})
	\begin{pmatrix}
		b_{m}zb_{m}^{*} & 0 \\
		0               & \unitmatrix_{n-m} \\
		\unitmatrix_{m}           & 0 \\
		0               & 0
    \end{pmatrix} \GL_{n}(\C)
	\\
	=&
	B_{H} h(\sigma_{I}) h\left(
	   \begin{pmatrix}
		b_{m} & 0 \\
		0     & \unitmatrix_{n-m} 
	   \end{pmatrix}
	\right)
	\begin{pmatrix}
		z     & 0 \\
		0     & \unitmatrix_{n-m} \\
		\unitmatrix_{m} & 0 \\
		0     & 0
	\end{pmatrix}
	\begin{pmatrix}
		b_{m}^{*} & 0 \\
		0         & \unitmatrix_{n-m}
	\end{pmatrix}         \GL_{n}(\C)
	\\
	=&
	B_{H} h 
    \left(
		\sigma_{I}
		\begin{pmatrix}
			b_{m} & 0 \\
			0     & \unitmatrix_{n-m} 
		\end{pmatrix}
		\sigma_{I}^{-1}
	\right)
	h(\sigma_{I})
		\begin{pmatrix}
			z     & 0 \\
			0     & \unitmatrix_{n-m} \\
			\unitmatrix_{m} & 0 \\
			0     & 0
		\end{pmatrix}
    \GL_{n}(\C).
\end{align*}
Note that Lemma \ref{Lem-B-p-cap-GL-times-GL} tells that 
%%Because we have to prove that this coincides with the image of $z$ by \eqref{eq-def-phi_I-bar},
%%it is sufficient to prove the following by the definition of $B_{H}$ (which is given in \eqref{eq-def-B_H-caseA}):
\begin{align*}
	\sigma_{I}
		\begin{pmatrix}
			b_{m} & 0 \\
			0     & \unitmatrix_{n-m}
		\end{pmatrix}
	\sigma_{I}^{-1}
	\in B_{n},
\quad
\text{ hence we get }
\quad
%%\\
%%	\intertext{so that}
	B_{H} h 
    \left(
		\sigma_{I}
		\begin{pmatrix}
			b_{m} & 0 \\
			0     & \unitmatrix_{n-m} 
		\end{pmatrix}
		\sigma_{I}^{-1}
	\right) = B_H.  
\end{align*}

Let us prove the surjectivity.  
%% of the map \eqref{eq-def-phi_I-bar}.
Let $\omega_{C,D} = \vectwo{C}{D} \in \mathrm{LM}_{2n,n}^{I}$.
%%(Definition \ref{Def-LM-I}).
We will prove that $\omega_{C,D} \in \mathrm{LM}_{2n,n}^{I}$ can be transformed into
\begin{equation*}
    h(\sigma_{I})
		\begin{pmatrix}
			z     & 0 \\
			0     & \unitmatrix_{n-m} \\
			\unitmatrix_{m} & 0 \\
			0     & 0
		\end{pmatrix}
\end{equation*}
for some $z \in \Her_{m}(\C)$ by the actions of $B_{H}$ and $\GL_{n}(\C)$.
Recall that the actions of  $B_{H}$ and $\GL_{n}(\C)$ on $\mathrm{LM}_{2n,n}^{I}$ are given by
\begin{equation*}
	h(b) \omega_{C,D} g
	=
	\begin{pmatrix}
		bCg \\
		(b^{*})^{-1}Dg
	\end{pmatrix}.
\end{equation*}
Thus, after the right translation by an appropriate element $g\in \GL_{n}(\C)$, we can assume that $D$ is in the reduced column echelon form \eqref{eq-Row-echelon-form}.
Moreover, by the left multiplication of some $(b^{*})^{-1}$, 
we can also assume that
\begin{equation*}
    D = 
    \begin{pmatrix}
	e_{i_{1}} & 
	e_{i_{2}} & 
	\dots     & 
	e_{i_{m}} &  
	0         &
	\dots     & 
	0
    \end{pmatrix} ,
\end{equation*}
where $I=\{i_{1}<i_{2}<\dots<i_{m}\}$.
Then, by the definition of $\sigma_{I}\in\Sgroup_{n}$ (Definition \ref{Def-sigma-I}), we have
\begin{equation}
    \sigma_{I}^{-1} D
    =
    \begin{pmatrix}
        \unitmatrix_{m} & 0 \\
        0     & 0
    \end{pmatrix},
\qquad
    \text{or,}
\qquad
    D=
    \sigma_{I}
    \begin{pmatrix}
        \unitmatrix_{m} & 0 \\
        0     & 0
    \end{pmatrix}.
\label{eq-def-sigma_{I}D}
\end{equation}
Let us write 
%%Moreover, we define $z_{ij}\:(i,j\in\{1,2\})$ by
\begin{equation}
    \sigma_{I}^{-1} C = 
    \begin{pmatrix}
        z_{11} & z_{12} \\
        z_{21} & z_{22}
    \end{pmatrix}
		\quad
			\text{so that,}
		\quad
    C=
    \sigma_{I}
    \begin{pmatrix}
        z_{11} & z_{12} \\
        z_{21} & z_{22}
    \end{pmatrix}.
\label{eq-def-sigma_{I}C}
\end{equation}
%%Note that $\sigma_{I} \in \OO(n)$, and thus, we have 
Since $\sigma_{I}^{*} ={}^{t}\sigma_{I} =\sigma_{I}^{-1}$,
we get 
\begin{equation}\label{eq-Lem-prf-C^{*}D}
    C^{*}D
    =
    \begin{pmatrix}
        z_{11} & z_{12} \\
        z_{21} & z_{22}
    \end{pmatrix}{\vphantom{\Bigm|}}^{*}
    (\sigma_{I}^{*})  \sigma_{I}
    \begin{pmatrix}
        \unitmatrix_{m} & 0 \\
        0     & 0
    \end{pmatrix}
    =
    \begin{pmatrix}
        z_{11}^{*} & z_{21}^{*} \\
        z_{12}^{*} & z_{22}^{*}
    \end{pmatrix}
    \begin{pmatrix}
        \unitmatrix_{m} & 0 \\
        0     & 0
    \end{pmatrix}     
    =
    \begin{pmatrix}
        z_{11}^{*} & 0 \\
        z_{12}^{*} & 0 
    \end{pmatrix}.
\end{equation}
Moreover, $\omega_{C,D}\in \mathrm{LM}_{2n,n}^{I} \subset \LregMat_{2n,n}$ implies that \eqref{eq-Lem-prf-C^{*}D} is an element of $\Her_{n}(\C)$, and thus, we have 
\begin{equation}\label{eq-z11=z11*-z12=0}
    z_{11} = z_{11}^{*},
    \quad
    z_{12}=0.
\end{equation}
By combining this equation \eqref{eq-z11=z11*-z12=0} with \eqref{eq-def-sigma_{I}D} and \eqref{eq-def-sigma_{I}C},
we have
\begin{equation*}
    \omega_{C,D}
    =
    \begin{pmatrix}
        C \\ D
    \end{pmatrix}
    =
    h(\sigma_{I})
    \begin{pmatrix}
        z_{11} & 0 \\
        z_{21} & z_{22} \\
        \unitmatrix_{m}  & 0 \\
        0      & 0
    \end{pmatrix}.
\end{equation*}
%
%%Note that the multiplication by $h(\sigma_{I})^{-1}\in H \subset \GL_{2n}(\C) $ preserves the rank of $\Mat_{2n,n}(\C)$.
%%Thus, because $\operatorname{rank} \omega_{C,D} = n$ by $\omega_{C,D}\in \mathrm{LM}_{2n,n}^{I} \subset   \LregMat_{2n,n}$, we have
Since $ \rank \omega_{C,D} = n $, 
%%%%
\skipover{
\begin{equation*}
\operatorname{rank}
      \begin{pmatrix}
        z_{11} & 0 \\
        z_{21} & z_{22} \\
        \unitmatrix_{m}  & 0 \\
        0      & 0
      \end{pmatrix}  
= \operatorname{rank} \omega_{C,D}
= n,
\end{equation*}
}
%%%%%
the matrix $z_{22}$ is invertible.
Therefore, by multiplying appropriate element of $\GL_{n}(\C)$ from the right again, we can assume that 
$ z_{22} = \unitmatrix_{n -m} $ and $ z_{21} = 0 $.  
%
%%%%%%%% skipover:begin %%%%%%%%%%%%%%%%%%
\skipover{
\begin{equation*}
    \omega_{C,D} =
    \begin{pmatrix}
        C \\ D
    \end{pmatrix}
    = 
    h(\sigma_{I})^{-1}
    \begin{pmatrix}
        z_{11} & 0       \\
        0      & \unitmatrix_{n-m} \\
        \unitmatrix_{m}  & 0       \\
        0      & 0
    \end{pmatrix}.
\end{equation*}
}
%%%%%%%% skipover:end   %%%%%%%%%%%%%%%%%%
%
Thus we have proved the desired surjectivity because $z_{11}\in\Her_{m}(\C)$.
%%\eqref{eq-z11=z11*-z12=0}.

Finally, we prove the injectivity of the map \eqref{eq-def-phi_I-bar}.
Let $z,z'\in\Her_{n}(\C)$.
Suppose
\skipover{
\begin{equation*}
	B_{H} h(\sigma_{I})
		\begin{pmatrix}
			z     & 0       \\
			0     & \unitmatrix_{n-m} \\
			\unitmatrix_{m} & 0       \\
			0     & 0
		\end{pmatrix}
	\GL_{n}(\C)
	=
	B_{H} h (\sigma_{I})
		\begin{pmatrix}
			z'    & 0       \\
			0     & \unitmatrix_{n-m} \\
			\unitmatrix_{m} & 0       \\
			0     & 0
		\end{pmatrix}
	\GL_{n}(\C).
\end{equation*}
In other words, suppose }
that 
\begin{equation}\label{eq-inj-comparison}
	h(b)
	h(\sigma_{I})
	\begin{pmatrix}
        z     & 0       \\
		0     & \unitmatrix_{n-m} \\
		\unitmatrix_{m} & 0       \\
		0     & 0
	\end{pmatrix} g
	=
	h(\sigma_{I})
	\begin{pmatrix}
		z'    & 0 \\
		0     & \unitmatrix_{n-m} \\
		\unitmatrix_{m} & 0 \\
		0     & 0
	\end{pmatrix}.
\end{equation}
for some $h(b)\in B_{H}$ and $g\in \GL_{n}(\C)$.  
The left-hand side of \eqref{eq-inj-comparison} is equal to
\begin{equation}\label{eq-inj-comparison-2}
	h(b) h(\sigma_{I})
	\begin{pmatrix}
		z     & 0 \\
		0     & \unitmatrix_{n-m} \\
		\unitmatrix_{m} & 0 \\
		0     & 0
	\end{pmatrix} g
    =
    h(\sigma_{I}) h (\sigma_{I}^{-1} b \sigma_{I})
	\begin{pmatrix}
		z     & 0 \\
		0     & \unitmatrix_{n-m} \\
		\unitmatrix_{m} & 0 \\
		0     & 0
	\end{pmatrix} g .
\end{equation}
In order to compute it,
%%this \eqref{eq-inj-comparison-2}, 
we set
\begin{equation}\label{eq-def-beta-gamma-g}
    \sigma_{I}^{-1} b \sigma_{I}
    \coloneqq
    \begin{pmatrix}
        \beta_{1} & \beta_{2} \\
        \beta_{3} & \beta_{4}
    \end{pmatrix},
\quad
    ((\sigma_{I}^{-1} b \sigma_{I})^{*})^{-1}
    \coloneqq
    \begin{pmatrix}
        \gamma_{1} & \gamma_{2}\\
        \gamma_{3} & \gamma_{4}
    \end{pmatrix},
\quad
    g
    \coloneqq
    \begin{pmatrix}
        g_{1} & g_{2}\\
        g_{3} & g_{4}
    \end{pmatrix}.
\end{equation}
After some computation, we get 
%%Then, the definition \eqref{eq-def-H-and-h} of $h$ implies that 
\begin{align*}
%%%%
\skipover{
    &=	
    h(\sigma_{I})
	\begin{pmatrix}
		\beta_{1} & \beta_{2} & 0 & 0 \\
		\beta_{3} & \beta_{4} & 0 & 0 \\
		0 & 0 & \gamma_{1} & \gamma_{2}\\
		0 & 0 & \gamma_{3} & \gamma_{4}\\
	\end{pmatrix}
	\begin{pmatrix}
		z     & 0 \\
		0     & \unitmatrix_{n-m} \\
		\unitmatrix_{m} & 0 \\
		0     & 0
	\end{pmatrix}
    \begin{pmatrix}
        g_{1} & g_{2}\\
        g_{3} & g_{4}
    \end{pmatrix}
\\}
%%%%
\eqref{eq-inj-comparison-2} 
    &=	
    h(\sigma_{I})
    \begin{pmatrix}
        \beta_{1}zg_{1} + \beta_{2}g_{3} 
        &   
        \beta_{1}zg_{2} + \beta_{2}g_{4} 
    \\
        \beta_{3}zg_{1} + \beta_{4}g_{3} 
        & 
        \beta_{3}zg_{2} + \beta_{4}g_{4} 
    \\
        \gamma_{1}g_{1} & \gamma_{1}g_{2} 
    \\
        \gamma_{3}g_{1} & \gamma_{3}g_{2}
    \end{pmatrix}.
\end{align*}
%
%%Because we supposed that this is equal to the right-hand of 
Comparing with \eqref{eq-inj-comparison}, we get
\begin{equation}\label{eq-comparizon-2}
    \begin{pmatrix}
        \beta_{1}zg_{1} + \beta_{2}g_{3} 
        & 
        \beta_{1}zg_{2} + \beta_{2}g_{4} 
    \\
        \beta_{3}zg_{1} + \beta_{4}g_{3} 
        & 
        \beta_{3}zg_{2} + \beta_{4}g_{4} 
    \\
        \gamma_{1}g_{1} & \gamma_{1}g_{2} 
    \\
        \gamma_{3}g_{1} & \gamma_{3}g_{2}
    \end{pmatrix}	
=	
	\begin{pmatrix}
		z'    & 0 \\
		0     & \unitmatrix_{n-m} \\
		\unitmatrix_{m} & 0 \\
		0     & 0
    \end{pmatrix}, 
\end{equation}
which implies 
%%%%
\skipover{
By comparing 
the $(3,1)$-entry,
the $(4,1)$-entry,
the $(3,2)$-entry,
the $(2,2)$-entry,
and
the $(1,2)$-entry
of \eqref{eq-comparizon-2} in this order, we have
}
%%%%
\begin{equation*}
	\gamma_{1}=g_{1}^{-1},
	\quad
	\gamma_{3}=0,
	\quad
	g_{2}=0,
	\quad
	\beta_{4}=g_{4}^{-1},
	\quad
	\beta_{2}=0.
\end{equation*}
Combining these and \eqref{eq-def-beta-gamma-g}, we have
\begin{equation*}
    \begin{pmatrix}
        \beta_{1} & 0\\
        \beta_{3} & \beta_{4}
    \end{pmatrix}
    =
    \left(
        \begin{pmatrix}
            \gamma_{1} & \gamma_{2}\\
            0          & \gamma_{4}
        \end{pmatrix}^{*}
    \right)^{-1}
    =
    \begin{pmatrix}
        \gamma_{1}^{*} & 0\\
        \gamma_{2}^{*} & \gamma_{4}^{*}
    \end{pmatrix}^{-1},
\end{equation*}
which implies
\begin{equation*}
    \beta_{1}=(\gamma_{1}^{*})^{-1},
    \quad
    \beta_{4}=(\gamma_{4}^{*})^{-1}.
\end{equation*}
%
%%%%
\skipover{
%%By using these formula and comparing 
From the $(2,1)$-entry of \eqref{eq-comparizon-2}, we get
\begin{equation*}
	g_{3}
	=
	- \beta_{4}^{-1} \beta_{3} z \beta_{1}^{*}.
\end{equation*}
}
%%%%
%
Thus, back to the equation \eqref{eq-comparizon-2}, we have
\begin{equation*}
    \begin{pmatrix}
        \beta_{1}z\beta_{1}^{*}  &  0  \\
        0                        & \unitmatrix_{n-m}\\
        \unitmatrix_{m}                    & 0 \\
        0                        & 0
    \end{pmatrix}	
    =	
	\begin{pmatrix}
		z'    & 0 \\
		0     & \unitmatrix_{n-m} \\
		\unitmatrix_{m} & 0 \\
		0     & 0
	\end{pmatrix}.
\end{equation*}
Therefore, for the injectivity of the map \eqref{eq-def-phi_I-bar}, it is sufficient to prove $\beta_{1}\in B_{m}$, which follows from Lemma \ref{Lem-B-p-cap-GL-times-GL} and the definition 
of $\beta_{1}$. 
%%(given in \eqref{eq-def-beta-gamma-g}).
\end{proof}

Summarizing the results, we get 

\begin{lemma}\label{Lem-B-G-P_S=coprod-B-Her}
%%    Let $n\in\Z_{\geq 0}$ and set $[n] = \{1,2,\dots,n\}$. 
    The collection of the maps for $ I \subset [n] $ 
    \begin{equation*}
	\begin{array}{ccc}
        \varphi_{I} \colon B_{\# I}\backslash \Her_{\# I}(\C)
		&\xrightarrow{\sim} &
		B_{H}\backslash \mathrm{LM}_{2n,n}^{I} / \GL_{n}(\C)
		\\
		\rotatebox{90}{$\in$}  && \rotatebox{90}{$\in$}
		\\
		B_{\# I} z & \mapsto &
		B_{H}h(\sigma_{I})
		\begin{pmatrix}
			z     & 0 \\
			0     & \unitmatrix_{n-\# I} \\
			\unitmatrix_{\# I} & 0 \\
			0     & 0
		\end{pmatrix}
		\GL_{n}(\C),
	\end{array}
    \end{equation*}
gives an isomorphism
    \begin{equation*}
\begin{aligned}
        \coprod_{I \subset [n]} \varphi_{I} \colon
        \coprod_{I\subset [n]} B_{\# I} \backslash \Her_{\# I}(\C) 
        \xrightarrow{\;\;\sim\;\;} &
        \coprod_{I\subset [n]} B_{H}\backslash \mathrm{LM}_{2n,n}^{I} / \GL_{n}(\C) 
\\
&= B_{H} \backslash \LregMat_{2n,n} /\GL_{n}(\C) \simeq B_{H} \backslash G /P_{S}.
\end{aligned}
    \end{equation*}
\end{lemma}

\begin{proof}
    This follows from Lemmas \ref{Lem-BGP=BLMGL}, \ref{Lem-orbit-decomp-B-LM-GL=coprod-B-LMI-GL}, and \ref{Lem-(B-Her)=(B-LM-I/GL)}.
\end{proof}

In fact, the cosets appear in the above lemma are finite, which we will prove later.  
So the ``isomorphism'' in the lemma means just a bijection.  

Now the full double coset decomposition $B_{H} \backslash G /P_{S}$ reduces to 
the orbit decomposition $B_{m} \backslash \Her_{m}(\C)$ for various $ m \geq 0 $, which we will consider in the next section.
%% (Proposition \ref{Prop-SPI=(B-Her)}).

%%\Today[January 15, 2025]

\subsection{Orbits of Borel subgroups in Hermitian and symmetric matrices}\label{Section-Orbit-decomp-B-backslash-Her}

Let us give orbit decompositions $B_{m}(\C)\backslash \Her_{m}(\C)$ (Case \caseA) and $B_{m}(\R) \backslash \Sym_{m}(\R)$ (Case \caseB) in Proposition \ref{Prop-SPI=(B-Her)}.  
To do so, we need signed partial involutions defined below. 

\begin{definition}\label{Def-SPI-and-SPIst}
For $m\in\Z_{\geq 0}$, we write $\SPI_{m}$ for the set of \emph{signed partial involutions}, defined by 
    \begin{equation*}
    \begin{split}
        \SPI_{m} &\coloneqq 
        \{ 
            \tau \in \SPP_{m} 
        \mid 
            \tau^{2} \text{ is a diagonal matrix whose entries belong to the set }\{0,1\}
        \}\\
        &=
        \SPP_{m} \cap \Sym_{m}(\Z),
    \end{split}
    \end{equation*}
    where $\mathrm{SPP}_{m}$ is the set of signed partial permutations defined in Definition \ref{Def-SPP-and-regOmega}.
    Moreover, we write $\SPIst_{m}$ for the subset of $\SPI_{m}$ consisting all elements whose off-diagonal entries are all nonnegative.
    More precisely, 
    \begin{equation*}
        \SPIst_{m} \coloneqq \{ \tau \in \SPI_{m} \mid  \tau_{ij} \geq 0 \text{ if } i\neq j\},
    \end{equation*}
    where $\tau_{ij}$ is the $(i,j)$-th entry of $\tau \in \Mat_{m}(\Z)$.
    Note that we have inclusions
    \begin{equation*}
        \SPIst_{m} \subset \SPI_{m} \subset \Sym_{m}(\R) \subset \Her_{m}(\C).
    \end{equation*}
\end{definition}

    There exists an action of $\Z_{2}^{m} \coloneqq \{\pm 1\}^{m}$ on $\SPI_{m}$ given by
    \begin{equation}\label{Rem-SPI^st}
        \varepsilon \cdot \tau = \varepsilon \tau \varepsilon,
    \end{equation}
where we regard $ \ee \in \Z_{2}^{m} $ as $ \varepsilon = \diag (\varepsilon_{1},\dots,\varepsilon_{m}) \in \GL_{m}(\C) $ by abuse of notations.
%%%%
\skipover{
    \begin{equation*}
        \begin{array}{ccc}
            \Z_{2}^{m} & \hookrightarrow &  \GL_{m}(\C)\\
            \rotatebox{90}{$\in$} & & \rotatebox{90}{$\in$}\\
            \varepsilon=(\varepsilon_{1},\dots,\varepsilon_{m}) & \mapsto &  \diag(\varepsilon_{1},\dots,\varepsilon_{m}).
        \end{array}
    \end{equation*}
}
%%%%
Then there is a natural isomorphism (see Example \ref{Example-SPI^st} given below).
    \begin{equation}\label{eq-Z2-SPI=SPIst}
        \Z_{2}^{m} \backslash \SPI_{m} 
        \simeq 
        \SPIst_{m} .
    \end{equation}
%
%%where $\SPIst_{m}$ is defined in Definition \ref{Def-SPI-and-SPIst} 

\begin{example}\label{Example-SPI^st}
For $(-1,-1,1) \in \Z_{2}^{3}$, we have
    \begin{equation*}
    \begin{split}
        &
        (-1,1,1)
        \cdot 
        \begin{pmatrix}
            0 & 0 & -1  \\
            0 & -1 & 0 \\
            -1& 0 & 0 \\
        \end{pmatrix}
        =
        \begin{pmatrix}
            0 & 0 & 1 \\ 
            0 & -1 & 0 \\
            1& 0 & 0 \\
        \end{pmatrix}  \in \SPIst_{m}.
    \end{split}
    \end{equation*}
Note that the $ (-1) $ on the diagonal cannot be resolved into $ 1 $.
\end{example}

\begin{proposition}
\label{Prop-SPI=(B-Her)}
    Consider the action of the group $B_{m}(\C)$ (resp.~$B_{m}(\R)$) of upper triangular matrices on $\Her_{m}(\C)$ (resp.~$\Sym_{m}(\R)$) given by
    \begin{equation}
        b \cdot z \coloneqq b \, z \, b^{*}
        \qquad
        (\text{resp. } b \cdot z \coloneqq b \, z \transpose{b}).
    \label{eq-def-action-of-B-on-Her(C)-in-Prop}
    \end{equation}
%%    Then the orbit decomposition of this action is given by:
Then we have bijections: 
%%    \begin{enumerate}
%%        \item[\caseA] We have a bijection:
    \begin{equation}\label{eq-map-from-SPIst-to-B-Her(C)}
\text{\upshape{Case \caseA}} \qquad\qquad
        \begin{array}{ccc}
            \SPIst_{m}  & \xrightarrow{\sim} &   B_{m}(\C) \backslash \Her_{m}(\C)
            \\
            \rotatebox{90}{$\in$}  && \rotatebox{90}{$\in$} 
            \\
            \tau & \mapsto & B_{m}(\C) \tau ,
        \end{array}
        \hspace*{.3\textwidth}
    \end{equation}
%%        \item[\caseB] We have a bijection: 
    \begin{equation}\label{eq-map-from-SPIst-to-B-Sym(R)}
    \;\;
\text{\upshape{Case \caseB}} \qquad\qquad
        \begin{array}{ccc}
            \SPIst_{m} & \xrightarrow{\sim} &  B_{m}(\R) \backslash \Sym_{m}(\R) 
            \\
            \rotatebox{90}{$\in$} && \rotatebox{90}{$\in$} 
            \\
            \tau & \mapsto & B_{m}(\R) \tau ,
        \end{array}
        \hspace*{.3\textwidth}
    \end{equation}
%%    \end{enumerate}
i.e., the set of signed partial involutions $ \SPIst_{m} $ is a complete system of representatives for the both of $ B_{m}(\C) \backslash \Her_{m}(\C) $ and $ B_{m}(\R) \backslash \Sym_{m}(\R) $.
\end{proposition}

The proof is given in the next section.
%%Proposition \ref{Prop-SPI=(B-Her)} 

%%\Today[January 16, 2025]

\section{Gaussian elimination}\label{Section-Gaussian-elimination}

In this section, we prove Proposition \ref{Prop-SPI=(B-Her)} by using a version of Gaussian elimination.   
We only deal with Case \caseA, since we can apply the same proof \textit{mutatis mutandis} to Case \caseB by rereading the symbols according to the list in Notation \ref{Notation-caseA-and-caseB}.  
For the simplicity, we set 
\begin{equation*}
    B_{m} \coloneqq B_{m}(\C)
\end{equation*}
in this section.  

%%For the proof of Proposition \ref{Prop-SPI=(B-Her)}, we need some preparations.

\begin{notation}\label{Convention-submatix}
Recall $ [m] = \{1,2,\dots,m\}$ and $ \dbinom{[m]}{k} $.  
For $ 0 \leq p \leq m $, we put $ [p]^{c} := [m] \setminus [p] = \{ p + 1, \dots, m \} $.  
For a matrix  
\begin{equation*}
A =(a_{ij})_{i\in[m], j \in [n]} \in \Mat_{m, n}(\C) \qquad
\text{ and } \qquad 
I \in \dbinom{[m]}{k}, \;\; J \in \dbinom{[n]}{\ell} , 
\end{equation*}
we write $A_{IJ} \in \Mat_{k, \ell}(\C) $ for the submatrix of $A$ corresponding to the indices in $I$ and $J$:
\begin{equation*}
    A_{IJ} \coloneqq (a_{ij} )_{i\in I, j\in J} \quad \in \Mat_{k, \ell}(\C).
\end{equation*}
%
%%(see Example \ref{Ex-formula-of-k}).
In the case $I=J$, we also write
\begin{equation*}
    A_{I} \coloneqq A_{I,I} = (a_{ij} )_{i,j\in I} \quad \in \Mat_{k}(\C).
\end{equation*}
\end{notation}

Recall the action of $B_{m}$ on $\Her_{m}(\C)$: 
\(
        b\cdot z = b\,z\,b^{*} \; (b \in B_{m}, \; z \in \Her_{m}(\C)).
\)

\begin{lemma}\label{Lemma-rank-sign-B-invariant}
    The collection of the quantities $\{\rank z_{[p]^{c}[q]^{c}}\}_{0\leq p\leq q < m}$ and $\{\sign z_{[p]^{c}}\}_{0\leq p < m}$ are 
$ B_m $-invariants on $\Her_{m}(\C)$, and they separate the $ B_m $-orbits.  
Here    $\sign$ denotes the signature of a Hermitian matrix,
%% (See Example \ref{Ex-formula-of-k}), 
and $[p]^{c} = [m] \setminus [p] $ (see Notation \ref{Convention-submatix}).

In other words, for $ z, z' \in \Her_{m}(\C) $, 
$ z $ and $ z' $ generate the same $ B_m $-orbit if and only if 
    \begin{align*}
        \rank z_{[p]^{c}[q]^{c}} &= \rank {z'}_{[p]^{c}[q]^{c}} \quad (0\leq p \leq q< m),
\qquad\text{ and }\;
\\
        \sign z_{[p]^{c}} &= \sign {z'}_{[p]^{c}} \quad (0\leq p< m).
    \end{align*}
%%    where $[p]^{c} \coloneqq \{p+1, p+2, \dots, m\}$ and we used Notation \ref{Convention-submatix}.
\end{lemma}

\begin{proof}%%[Proof of Lemma \ref{Lemma-rank-sign-B-invariant}]
Let $b\in B_{m}$ and $z\in \Her_{m}(\C)$, and write these matrices as block matrices corresponding to the partition $m=p+q+(m-p-q)$:
\begin{equation*}
    b=
    \begin{pmatrix}
        b_{11} & b_{12} & b_{13} \\
        0      & b_{22} & b_{23} \\
        0      & 0      & b_{33}
    \end{pmatrix},
    \quad
    z 
    = \begin{pmatrix}
        z_{11}     & z_{12}     & z_{13} \\
        z_{21} & z_{22}     & z_{23} \\
        z_{31} & z_{32} & z_{33} \\
    \end{pmatrix}
    = \begin{pmatrix}
        z_{11}^{*} & z_{12}     & z_{13} \\
        z_{12}^{*} & z_{22}^{*} & z_{23} \\
        z_{13}^{*} & z_{23}^{*} & z_{33}^{*} \\
    \end{pmatrix}.
\end{equation*}
We get
\begin{equation*}
    \begin{split}
        b \cdot z 
        & =
        \begin{pmatrix}
          \sum_{i = 1}^3 b_{1i}z_{i1} \smallvstrut
%%            b_{11}z_{11} + b_{12}z_{12}^{*} + b_{13}z_{13}^{*}
            &
          \sum_{i = 1}^3 b_{1i}z_{i2}
%%            b_{11}z_{12} + b_{12}z_{22}     + b_{13}z_{23}^{*}
            &
          \sum_{i = 1}^3 b_{1i}z_{i3}
%%            b_{11}z_{13} + b_{12}z_{23}     + b_{13}z_{33}
            \\
          \sum_{i = 2}^3 b_{2i}z_{i1} \smallvstrut
%%            b_{22}z_{12}^{*} + b_{23}z_{13}^{*}
            &
          \sum_{i = 2}^3 b_{2i}z_{i2}
%%            b_{22}z_{22}     + b_{23}z_{23}^{*}
            &
          \sum_{i = 2}^3 b_{2i}z_{i3}
%%            b_{22}z_{23}     + b_{23}z_{33}
            \\
            b_{33}z_{31} & b_{33}z_{32} & b_{33}z_{33}
        \end{pmatrix}
        \begin{pmatrix}
            b_{11}^{*} & 0           & 0 \\
            b_{12}^{*} & b_{22}^{*}  & 0 \\
            b_{13}^{*} & b_{23}^{*}  & b_{33}^{*}
        \end{pmatrix}
        \\
        &=
        \begin{pmatrix}
            \star & \star & (b_{11}z_{13} + b_{12}z_{23}     + b_{13}z_{33})b_{33}^{*} 
            \\
            \star & \star & (b_{22}z_{23}     + b_{23}z_{33})b_{33}^{*}
            \\
            \star & \star & b_{33}z_{33}b_{33}^{*}
        \end{pmatrix},
    \end{split}
\end{equation*}
where ``$\star$'' denotes a matrix that does not matter in this proof.
Thus we know
\begin{equation*}
    \sign (b\cdot z)_{[p+q]^{c}} = \sign b_{33}z_{33}b_{33}^{*} = \sign z_{33} = \sign z_{[p+q]^{c}}
\end{equation*}
and
\begin{equation*}
    \rank (b\cdot z)_{[p]^{c}[p+q]^{c}} 
    =
    \rank \begin{pmatrix}
            (b_{22}z_{22} + b_{23}z_{33})b_{33}^{*}
            \\
            b_{33}z_{33}b_{33}^{*}
        \end{pmatrix}
    =
    \rank \begin{pmatrix}
            z_{23}
            \\
            z_{33}
        \end{pmatrix}
    = \rank z_{[p]^{c}[p+q]^{c}},
\end{equation*}
which proved $ \rank z_{[p]^{c}[p+q]^{c}} $ and $ \sgn z_{[p+q]^{c}} $ are in fact $ B_m $-invariants.  
We postpone the proof that they are complete invariants separating the orbits, and prove it later 
after the proof of Proposition \ref{Prop-SPI=(B-Her)}.  
\end{proof}

\begin{lemma}\label{Lemma-rank-and-sign-separate-SPIst}
    The invariants $\{\rank z_{[p]^{c}[q]^{c}}\}_{0\leq p \leq q < m}$ and $\{\sign z_{[p]^{c}}\}_{0\leq p < m}$ of $B_{m}$-orbits separate elements in $\SPIst_{m} \subset \Her_{m}(\C)$, 
so that signed partial involutions in $ \SPIst_{m} $ generate distinct orbits.
\skipover{
    In other words, for $\tau, \tau' \in \SPIst_{m}$ with $\tau \neq \tau'$, there exists $0\leq p<q< m$ such that we have
    \begin{equation*}
        \rank \tau_{[p]^{c}[q]^{c}} \neq \rank \tau_{[p]^{c}[q]^{c}}',
    \end{equation*}
     or $0\leq p< m$ such that 
     \begin{equation*}
        \sign \tau_{[p]^{c}} \neq \sign \tau_{[p]^{c}}'.
    \end{equation*}
}
%%    where $[p]^{c} \coloneqq \{p+1, p+2, \dots, m\}$ and we used Notation \ref{Convention-submatix}.
\end{lemma}

\begin{proof}
Let us prove that $\tau \in \SPIst_{m}$ can be recovered from $\{\rank \tau_{[p]^{c}[q]^{c}}\}_{0\leq p \leq q < m}$ and $\{\sign \tau_{[p]^{c}}\}_{0\leq p < m}$.
We will reconstruct the submatrices $\tau_{[p]^{c}}$ of $\tau$ starting from $p=m-1$ to $p=0$ in decreasing order.  
%%by using the invariants $\{\rank \tau_{[p]^{c}[q]^{c}}\}_{0\leq p \leq q < m}$ and $\{\sign \tau_{[p]^{c}}\}_{0\leq p < m}$. 
(Note that $\tau_{[0]^{c}}=\tau$.)

For $p=m-1$, since $\tau_{[p]^{c}} = \tau_{\{m\}}$, it is determined by the signature, namely:
    \begin{equation*}
        \tau_{[p]^{c}} = 
        \begin{cases}
            1 & \text{if }\sign \tau_{[p]^{c}} = (1,0,0), \\
            -1& \text{if }\sign \tau_{[p]^{c}} = (0,1,0), \\
            0 & \text{if }\sign \tau_{[p]^{c}} = (0,0,1). 
        \end{cases}
    \end{equation*}

    Assume that we already construct $\tau_{[p+1]^{c}} = \tau_{\{p+2, p+3, \dots, m\}}$. 
    Then, the submatrix $\tau_{[p]^{c}} = \tau_{\{p+1, p+2, \dots, m\}}$ is determined by the formula:
    \begin{equation}
        \tau_{[p]^{c}} =
        \begin{cases}
            \begin{pmatrix}
                1 & 0 \\ 0 & \tau_{[p+1]^{c}}
            \end{pmatrix}
            & \text{ if } \sign \tau_{[p]^{c}} = (1,0,0 ) + \sign\tau_{[p+1]^{c}}, \\
            \begin{pmatrix}
                -1 & 0 \\ 0 & \tau_{[p+1]^{c}}
            \end{pmatrix}
            & \text{ if } \sign \tau_{[p]^{c}} = (0,1,0 ) + \sign\tau_{[p+1]^{c}}, \\
            \begin{pmatrix}
                0 & 0 \\ 0 & \tau_{[p+1]^{c}}
            \end{pmatrix}
            & \text{ if } \sign \tau_{[p]^{c}} = (0,0,1 ) + \sign\tau_{[p+1]^{c}}, \\
            \begin{pmatrix}
                0 & {}^{t}e_{k} \\ e_{k} & \tau_{[p+1]^{c}}
            \end{pmatrix}
            & \text{ if } \sign \tau_{[p]^{c}} = (1,1,-1 ) + \sign\tau_{[p+1]^{c}} ,
        \end{cases}
    \label{eq-def-tau_p-c}
    \end{equation}
    where $e_{k} $ denotes the $k$-th standard basis vector, 
and $k\geq 1$ is an integer defined by
%%    where $e_{k-p}\in\Mat_{m-p, 1}(\C)$ is the $(k-p)$th unit vector, and $k > p$ is the unique integer satisfying
    \begin{equation}
        k \coloneqq \max \{ 1\leq j \leq m-(p+1) \mid \rank \tau_{[p]^{c}[p+j]^{c}} - \rank \tau_{[p+1]^{c}[p+j]^{c}} = 1 \}.
        % \rank \tau_{[p]^{c}[p+k]^{c}} - \rank \tau_{[p+1]^{c}[p+k]^{c}} = 1.
        % \quad
        % \text{ and }\quad
        % \rank \tau_{[p]^{c}[p+k+1]^{c}} - \rank \tau_{[p+1]^{c}[p+k+1]^{c}} = 0.
    \label{eq-def-k-in-the-proof}
    \end{equation}
%%    (See also Example \ref{Ex-formula-of-k}).
Note that $\tau \in \SPIst_{m}$ implies that 
\begin{equation*}
    \rank \tau_{[p]^{c}[p+j]^{c}} - \rank \tau_{[p+1]^{c}[p+j]^{c}}
    =
    \rank \tau_{\{p+1\}[p+j]^{c}}
\end{equation*}
under Notation \ref{Convention-submatix}.
Therefore, \eqref{eq-def-k-in-the-proof} means that 
the nonzero entry of the $(p+1)$-th row of $\tau$ is the $(p+1, p+1+k)$-th entry (which is equal to the $(1,1+k)$-th entry of $\tau_{[p]^{c}}$).
This completes the proof.
\end{proof}

\begin{example}\label{Ex-formula-of-k}
    Let us consider 
%%$\tau \in \SPIst_{4} \subset  \Her_{4}(\C)$
%
    \begin{equation*}
        \tau = 
        \begin{pmatrix}
            0 & 0 & 0 & 0 \\
            0 & 0 & 1 & 0 \\
            0 & 1 & 0 & 0 \\
            0 & 0 & 0 & 1 \\
        \end{pmatrix} \in \SPIst_{4} \subset  \Her_{4}(\C).
    \end{equation*}
    Then, we get
    \begin{equation*}
        \tau_{[0]^{c}} 
%%= \tau_{\{1,2,3,4\}} 
= \tau,
        \quad
        \tau_{[1]^{c}} 
%%= \tau_{\{2,3,4\}} 
=
        \begin{pmatrix}
            0 & 1 & 0 \\
            1 & 0 & 0 \\
            0 & 0 & 1 \\
        \end{pmatrix},
        \quad
        \tau_{[2]^{c}} 
%%= \tau_{\{3,4\}} 
= 
        \begin{pmatrix}
            0 & 0 \\
            0 & 1 \\
        \end{pmatrix},
        \quad
        \tau_{[3]^{c}} 
%%= \tau_{\{4\}} 
=
        1,
    \end{equation*}
since $ [0]^{c} = [4], [1]^{c} = \{2,3,4\}, [2]^{c} = \{3,4\}, [3]^{c} = \{4\} $.  
The signatures are given by 
    \begin{equation*}
\begin{aligned}
        \sign \tau_{\{1,2,3,4\}} &= (2,1,1),
        \quad &
        \sign \tau_{\{2,3,4\}} &= (2,1,0),
\\
        \sign \tau_{\{3,4\}} &= (1,0,1),
        \quad &
        \sign \tau_{\{4\}} &= (1,0,0).
\end{aligned}
    \end{equation*}
    Thus, we have
    \begin{equation*}
        \sign \tau_{[1]^{c}} = (1,1,-1) + \sign \tau_{[2]^{c}}
    \end{equation*}
    (namely, if we set $p=1$, then we are in the setting of the fourth case of \eqref{eq-def-tau_p-c}). 
    In this case, consider 
    \begin{equation*}
        \tau_{[1]^{c}[2]^{c}} = 
        \begin{pmatrix}
            1 & 0 \\
            0 & 0 \\
            0 & 1 \\
        \end{pmatrix},
        \quad
        \tau_{[1]^{c}[3]^{c}} = 
        \begin{pmatrix}
            0 \\
            0 \\
            1 \\
        \end{pmatrix}
\quad   \text{ and } \quad
        \tau_{[2]^{c}[2]^{c}} = 
         \begin{pmatrix}
            0 & 0 \\
            0 & 1 \\
        \end{pmatrix},
        \quad
        \tau_{[2]^{c}[3]^{c}} = 
        \begin{pmatrix}
            0 \\
            1
        \end{pmatrix},
    \end{equation*}
    which implies 
    \begin{equation*}
        \rank \tau_{[1]^{c}[2]^{c}} = 2,
        \;\;
        \rank \tau_{[1]^{c}[3]^{c}} = 1
\quad \text{ and } \quad
        \rank \tau_{[2]^{c}[2]^{c}} = 1,
        \;\;
        \rank \tau_{[2]^{c}[3]^{c}} = 1.
    \end{equation*}
    Thus, we have
    \begin{equation*}
        % \rank \tau_{[1]^{c}[2]^{c}} - \rank \tau_{[2]^{c}[2]^{c}} = 1,
        \rank \tau_{[1]^{c}[2]^{c}} - \rank \tau_{[2]^{c}[2]^{c}} = 1\quad
        \text{ and }\quad
        \rank \tau_{[1]^{c}[3]^{c}} - \rank \tau_{[2]^{c}[3]^{c}} = 0,
    \end{equation*}
    which implies $k= 1$ in this case (under the notation of \eqref{eq-def-k-in-the-proof} in the proof of Lemma \ref{Lemma-rank-and-sign-separate-SPIst}).
\end{example}

%%Then, we prove Proposition \ref{Prop-SPI=(B-Her)} from now on.

\begin{proof}[Proof of Proposition \ref{Prop-SPI=(B-Her)}]
From the lemma above, it is clear that 
any two elements $\tau$ and $\tau'$ of $\SPIst_{m}$ belongs to different $B_{m}$-orbits, 
which implies the injectivity of the map \eqref{eq-map-from-SPIst-to-B-Her(C)}.
Thus, let us prove the surjectivity.  

We prove that any element $ z \in \Her_{m}(\C)$ can be transformed into an element of $\SPIst_{m}$ by the action of $B_{m}$.
%%namely, the surjectivity of the map \eqref{eq-map-from-SPIst-to-B-Her(C)}. 
We prove it by the induction on $ m $.  
Let us show that any $z\in \Her_{m}(\C)$ can be transformed by the action of $B_{m}$ to the one of the elements of the following form:
\begin{equation*}
    \begin{pmatrix}
        \star & 0 \\
        0 & 1
    \end{pmatrix},
    \quad
    \begin{pmatrix}
        \star & 0 \\
        0 & -1
    \end{pmatrix},
    \quad
    \begin{pmatrix}
        \star & 0 \\
        0 & 0
    \end{pmatrix},
    \quad
    \begin{pmatrix}
        \star & 0 & \star & 0 \\
        0     & 0 & 0     & 1 \\
        \star & 0 & \star & 0 \\
        0     & 1 & 0     & 0
    \end{pmatrix},
\end{equation*}
where ``$\star$'' stands for some matrix.
If this is proven, 
%%then considering the action of $B_{m-1} \times \{1\} \subset B_{m}$ (in the first three case) 
%%or $B_{k-1} \times \{1\} \times B_{m-k-1} \times \{1\} \subset B_{m}$ (for some $k$ in the last case), 
an easy induction implies the  proposition.

%%Let $z\in \Her_{m}(\C)$ and $b\in B_{n}$. 
We write 
\begin{equation*}
    z=
    \begin{pmatrix}
        z_{11}     & z_{12} \\
        z_{12}^{*} & z_{22}
    \end{pmatrix},
    \text{ where }
    z_{11} \in \Her_{m-1}(\C),\:
    z_{12} \in \mathrm{M}_{m-1,1}(\C),\:
    z_{22} \in \Her_{1}(\C) = \mathbb{R}.
\end{equation*}
Similarly, for $b\in B_{m}$, we write 
\begin{equation*}
    b=
    \begin{pmatrix}
        b_{11} & b_{12} \\
        0      & b_{22}
    \end{pmatrix},
    \text{ where }
    b_{11} \in B_{m-1},\:
    b_{12} \in \mathrm{M}_{m-1,1}(\C),\:
    b_{22} \in B_{1} = \GL_{1}(\C).
\end{equation*}
Using these notations, we calculate 
\begin{equation}
    \begin{split}
        b z b^{*} &=
        \begin{pmatrix}
            b_{11} & b_{12} \\
            0      & b_{22}
        \end{pmatrix}
        \begin{pmatrix}
            z_{11}     & z_{12} \\
            z_{12}^{*} & z_{22}
        \end{pmatrix}
        \begin{pmatrix}
            b_{11}^{*} & 0 \\
            b_{12}^{*} & b_{22}^{*}
        \end{pmatrix}
        \\
        &=
        \begin{pmatrix}
            b_{11}z_{11} + b_{12}z_{12}^{*}
            &
            b_{11}z_{12} + b_{12}z_{22} 
            \\
            b_{22}z_{12}^{*} 
            &
            b_{22}z_{22}
        \end{pmatrix}
        \begin{pmatrix}
            b_{11}^{*} & 0 \\
            b_{12}^{*} & b_{22}^{*}
        \end{pmatrix}
        \\
        &=
        \begin{pmatrix}
            (b_{11}z_{11} + b_{12}z_{12}^{*})b_{11}^{*}
            +
            (b_{11}z_{12} + b_{12}z_{22})b_{12}^{*}
            &
            (b_{11}z_{12} + b_{12}z_{22})b_{22}^{*}
            \\
            b_{22}(z_{12}^{*}b_{11}^{*} + z_{22}b_{12}^{*})
            &
            z_{22}|b_{22}|^{2}
        \end{pmatrix}.
    \end{split}
\label{eq-bzb^*}
\end{equation}
Now let us divide the cases.

First we assume $z_{22}\neq 0$. Then by putting $b_{22}=|z_{22}|^{-\frac{1}{2}}$, we can assume $z_{22}=\pm 1$,  
%%(by \eqref{eq-bzb^*}).
%%In this case ($z_{22}=\pm 1$ case), 
and the equation \eqref{eq-bzb^*} becomes
\begin{equation*}
    \begin{pmatrix}
        (b_{11}z_{11} + b_{12}z_{12}^{*})b_{11}^{*}
        +
        (b_{11}z_{12} \pm b_{12})b_{12}^{*}
        &
        (b_{11}z_{12} \pm b_{12})b_{22}^{*}
        \\
        b_{22}(z_{12}^{*}b_{11}^{*} \pm b_{12}^{*})
        &
        \pm 1
    \end{pmatrix}.  
\end{equation*}
Thus taking $b_{12} = \mp b_{11}z_{12}$, it becomes 
\begin{equation*}
    =
    \begin{pmatrix}
        (b_{11}z_{11} + b_{12}z_{12}^{*})b_{11}^{*} + (b_{11}z_{12} \pm b_{12})b_{12}^{*}
        & 0 \\
        0 & \pm 1
    \end{pmatrix},
\end{equation*}
and we are done.
%%which completes the proof in this case ($z_{22}\neq 0$ case),

Next, let us assume $z_{22}=0$.  
In this case, the equation \eqref{eq-bzb^*} becomes
\begin{equation*}
    bzb^{*} = 
    \begin{pmatrix}
        (b_{11}z_{11} + b_{12}z_{12}^{*})b_{11}^{*}   +   b_{11}z_{12}b_{12}^{*}  &  b_{11}z_{12}b_{22}^{*}
        \\
        b_{22}z_{12}^{*}b_{11}^{*} & 0
    \end{pmatrix}.
% \label{eq-bzb*=lower-right=0}
\end{equation*}
If $z_{12}=0$, we are done.   
%%which completes the proof in the case that $z_{12}=0$ is also satisfied (namely, $z_{12}=z_{22}=0$ case).
%%Therefore, 
Thus we can assume $z_{12}\neq 0$ (and $z_{22}=0)$.  
By choosing appropriate elements $b_{11} \in B_{m-1}$ and $b_{22} \in \GL_{1}(\C)$, we get
%%we can take $k\in\{1,2,\dots,m-1\}$ such that
%
\begin{equation}
    b_{11}z_{12}b_{22}^{*}=e_{k} \in \mathrm{M}_{m-1,1}(\C),
\end{equation}
where $e_{k}$ is the $k$-th standard basis vector.  
So we can assume $z_{12}=e_{k}$ from the beginning.  
In this case, where $z_{22}=0 $ and $ z_{12}=e_{k}$, we rewrite $z\in \Her_{m}(\C)$ as
\begin{equation*}
    z=
    \begin{pmatrix}
        z_{11}     & z_{12}     & z_{13} & 0 \\
        z_{12}^{*} & z_{22}     & z_{23} & 1 \\
        z_{13}^{*} & z_{23}^{*} & z_{33} & 0 \\
        0          & 1          & 0      & 0
    \end{pmatrix},
    \text{ where }
    z_{11} \in \Her_{k-1}(\C),\:
    z_{22} \in \Her_{1}(\C),\:
    z_{33} \in \Her_{m-k-1}(\C).
\end{equation*}
Then, the conjugation by 
\begin{equation*}
    b=
    \begin{pmatrix}
        1 & 0 & 0 & -z_{12} \\
        0 & 1 & 0 & -z_{22}/2 \\
        0 & 0 & 1 & -z_{23}^{*} \\
        0 & 0 & 0 & 1
    \end{pmatrix} \in B_{m}
%%\end{equation*}
%
\quad \text{ results in } \quad 
%
%%\begin{equation*}
%%\begin{split}
    bzb^{*} 
    = 
%%%%%%
\skipover{
    \begin{pmatrix}
        1 & 0 & 0 & -z_{12} \\
        0 & 1 & 0 & -z_{22}/2 \\
        0 & 0 & 1 & -z_{23}^{*} \\
        0 & 0 & 0 & 1
    \end{pmatrix}
    \begin{pmatrix}
        z_{11}     & z_{12}     & z_{13} & 0 \\
        z_{12}^{*} & z_{22}     & z_{23} & 1 \\
        z_{13}^{*} & z_{23}^{*} & z_{33} & 0 \\
        0          & 1          & 0      & 0
    \end{pmatrix}
    \begin{pmatrix}
        1 & 0 & 0 & 0 \\
        0 & 1 & 0 & 0 \\
        0 & 0 & 1 & 0  \\
        -z_{12}^{*} & -z_{22}^{*}/2 & -z_{23} & 1
    \end{pmatrix}
    \\
    &=
    \begin{pmatrix}
        z_{11}     & 0          & z_{13} & 0 \\
        z_{12}^{*} & z_{22}/2   & z_{23} & 1 \\
        z_{13}^{*} & 0          & z_{33} & 0 \\
        0          & 1          & 0      & 0
    \end{pmatrix}
    \begin{pmatrix}
        1 & 0 & 0 & 0 \\
        0 & 1 & 0 & 0 \\
        0 & 0 & 1 & 0  \\
        -z_{12}^{*} & -z_{22}^{*}/2 & -z_{23} & 1
    \end{pmatrix}
    \\
    &=
}
%%%%%%%%%%%%
    \begin{pmatrix}
        z_{11}     & 0          & z_{13} & 0 \\
        0          & 0          & 0      & 1 \\
        z_{13}^{*} & 0          & z_{33} & 0 \\
        0          & 1          & 0      & 0
    \end{pmatrix},
%%\end{split}
\end{equation*}
after tedious calculations, which we omit.
%%which completes the proof in this case ($z_{22}=0,z_{12}\neq 0$ case).
%%This completes the proof of the surjectivity of \eqref{eq-map-from-SPIst-to-B-Her(C)}, and also completes the proof of the proposition.
\end{proof}

Now Lemma \ref{Lemma-rank-and-sign-separate-SPIst} together with the proof of Proposition \ref{Prop-SPI=(B-Her)} 
show the invariants in Lemma \ref{Lemma-rank-sign-B-invariant} are complete (we did not use the completeness to prove Lemma \ref{Lemma-rank-and-sign-separate-SPIst} and Proposition \ref{Prop-SPI=(B-Her)}).

\section{Orbits of Borel subgroups in the Hermitian Lagrangian Grassmannian via the signed partial permutations}\label{Section-BGP-via-SPP}
%%$B_{H} \backslash G / P_{S}$ 

In this section, we prove Theorem \ref{theorem-Omega=H-G-PS}.
We note that the arguments in this section are similar to that of Section \ref{Section-Orbit-decomp-of-B-LM-GL}.
Until Theorem \ref{theorem-Omega=H-G-PS-2}, we do not specify that we are in Case \caseA or \caseB (of Setting \ref{Setting-DFVR}).  
The arguments below equally work for the both cases.
%%because it does not matter for the arguments given below (until Theorem \ref{theorem-Omega=H-G-PS-2}).

Recall the definition of $\regOmega_{n}$.

\begin{definition}\label{RE:Def-SPP-and-regOmega}\label{Def-Omega-I}
    A matrix $ \tau \in \Mat_{n}(\Z) $ is called a {signed partial permutation} if 
    its entries belong to the set $ \{-1,  0, 1 \} $ and there is at most one nonzero entry in each row and also in each column.  
    We write $ \SPP_n $ for the set of all signed partial permutations and put 
    \begin{equation*}
        \regOmega_{n}
        =
        \left\{
            \omega_{\tau_{1},\tau_{2}} =   
            \begin{pmatrix}
                \tau_{1} \\ \tau_{2}
            \end{pmatrix}
            \in \Mat_{2n,n}(\Z)
        \Bigm|
        \begin{array}{c}
             \tau_{1}, \tau_{2} \in \SPP_{n}  \\
             {}^{t}\tau_{1}\tau_{2} \in \Sym_{n}(\Z) \\
             \rank (\omega_{\tau_{1},\tau_{2}}) = n
        \end{array}
        \right\} .
    \end{equation*}
For $I\subset [n] = \{1,2,\dots,n\}$, we define a subset $\Omega^{I}_{n}$ of $\regOmega_{n}$ by
\begin{equation*}
\Omega^{I}_{n} \coloneqq	
    \left\{
        \omega_{\tau_{1},\tau_{2}}=
        \begin{pmatrix}
            \tau_{1} \\ \tau_{2}
        \end{pmatrix}   \in \regOmega_{n}
    \Bigm| 
         \ech  \tau_{2} = I
    \right\} = \mathrm{LM}_{2n,n}^{I} \cap \Mat_{2n,n}(\Z),
\end{equation*}
where the \'{e}chlon map $\ech $ is defined in the equation \eqref{eq-def-ech}.
\end{definition}

Note that $ \Mat_{2n, n}(\Z) $ inherits the natural action of $ \GL_{2n}(\Z) \times \GL_n(\Z) $ by the left and right multiplications.  
We identify $ \ee = (\ee_1, \dots, \ee_n) \in \Z_{2}^{n} $ with a diagonal matrix in $ \GL_n(\Z) $ and denote it by the same notation.  
Thus $ \ee \in \Z_{2}^{n} $ acts on $ \Mat_{2n, n}(\Z) $, 
and at the same time $ h(\ee) = \diag(\ee, \ee) \in \GL_{2n}(\Z) $ also  acts on it by the left multiplication.  
The set $ \regOmega_n $ is stable under the action of $(\Z_{2}^{n},\Z_{2}^{n}\rtimes \Sgroup_{n})$.

%%First, we divide $\regOmega_{n}$ into $(\Z_{2}^{n},\Z_{2}^{n}\rtimes \Sgroup_{n})$-invariant subsets, where the action is given in Remark \ref{Remark-reg-Omega-W-Cn-set}.

The Lemma \ref{Lem-Omega=coprod-Omega-I} given below is clear from the definitions.
%%an analogue of Lemma \ref{Lem-LM=coprod-LM-I}.

\begin{lemma}\label{Lem-Omega=coprod-Omega-I}\label{Lem-Omega-I-stable}\label{Lem-orbit-decomp-Omega=coprod-Omega-I}
%%For $I\subset [n] $,  
%%\coloneqq\{1,\dots,n\}$. 
There is a natural disjoint decomposition 
\begin{equation*}
\regOmega_{n} = \coprod_{I\subset[n]} \Omega^{I}_{n}, 
\end{equation*}
which is $(\Z_{2}^{n},\Z_{2}^{n}\rtimes \Sgroup_{n})$-stable.  
\end{lemma}

\begin{proof}
Regarding $\ech $ as a function $\ech \colon \regOmega_{n} \to \mathcal{P}([n])$, where $\mathcal{P}([n])$ is the power set of $[n]$,
we have $\Omega^{I}_{n}=\ech ^{-1}(I)$, which completes the proof.
\end{proof}

Thus, we have a partition of $\regOmega_{n}$ by $(\Z_{2}^{n},\Z_{2}^{n}\rtimes \Sgroup_{n})$-invariant subsets.
For the double coset decomposition $\Z_{2}^{n} \backslash  \Omega^{I}_{n} /(\Z_{2}^{n} \rtimes \Sgroup_{n})$, 
we have an analogue of Lemma \ref{Lem-(B-Her)=(B-LM-I/GL)}.

\begin{proposition}\label{Prop-Omega-I=SPI}
Pick $I \subset [n] $
%%    \coloneqq \{1,2,\dots,n\}$
and put $m\coloneqq \#I$. Then, we get a bijection
    \begin{equation}
        \begin{array}{ccc}
             \Z_{2}^{m} \backslash \SPI_{m}
             & \xrightarrow{\sim} &  
             \Z_{2}^{n} \backslash \Omega^{I}_{n} / (\Z_{2}^{n} \rtimes \Sgroup_{n})
             \\
             \rotatebox{90}{$\in$} & & \rotatebox{90}{$\in$}
             \\
             \tau & \mapsto & 
             \Z_{2}^{n} h(\sigma_{I})
             \begin{pmatrix}
                 \tau & 0 \\ 
                 0    & \unitmatrix_{n-m} \\
                 \unitmatrix_{m}& 0 \\
                 0    & 0
             \end{pmatrix}
             (\Z_{2}^{n} \rtimes \Sgroup_{n}),
        \end{array}
    \label{eq-def-phi_I-bar-tau}
    \end{equation}
    where %%$\Omega_{n}^{I}$ is defined in Definition \ref{Def-Omega-I}, 
    the action of $\Z_{2}^{m}$ on $\SPI_{m}$ is defined in \eqref{Rem-SPI^st}  
    and $h(\sigma_{I}) = \diag(\sigma_{I}, \sigma_{I}) \in H$.  
%%    is defined in Definition \ref{Def-sigma-I}.
\end{proposition}

\begin{proof}
We note that the proof given below is similar to that of Lemma \ref{Lem-(B-Her)=(B-LM-I/GL)}.
We will prove the well-definedness, surjectivity, and injectivity of the map \eqref{eq-def-phi_I-bar-tau} in this order.

First, we prove the well-definedness of the map \eqref{eq-def-phi_I-bar-tau}.
Let $\varepsilon\in \Z_{2}^{m}$ and $\tau \in\SPI_{m}$.
Then, it is sufficient to prove that the image of $\tau$ by the map \eqref{eq-def-phi_I-bar-tau} coincides with that of $\varepsilon\cdot \tau= \varepsilon \tau \varepsilon$.
The image of $\varepsilon\cdot \tau= \varepsilon \tau \varepsilon$ by the map \eqref{eq-def-phi_I-bar-tau} is equal to
\begin{align*}
	&
	\Z_{2}^{n} h(\sigma_{I})
	\begin{pmatrix}
		\varepsilon \tau \varepsilon & 0 \\
		0               & \unitmatrix_{n-m} \\
		\unitmatrix_{m}           & 0 \\
		0               & 0
    \end{pmatrix}
    (\Z_{2}^{n} \rtimes \Sgroup_{n})
	\\
	=&
	\Z_{2}^{n} h(\sigma_{I}) h
    \left(
	\begin{pmatrix}
		\varepsilon & 0 \\
		0     & \unitmatrix_{n-m} 
	\end{pmatrix}
	\right)
	\begin{pmatrix}
		\tau     & 0 \\
		0     & \unitmatrix_{n-m} \\
		\unitmatrix_{m} & 0 \\
		0     & 0
	\end{pmatrix}
	\begin{pmatrix}
		\varepsilon & 0 \\
		0         & \unitmatrix_{n-m}
	\end{pmatrix}
	(\Z_{2}^{n} \rtimes \Sgroup_{n})
	\\
	=&
	\Z_{2}^{n} h 
    \left(
		\sigma_{I}
		\begin{pmatrix}
			\varepsilon & 0 \\
			0     & \unitmatrix_{n-m} 
		\end{pmatrix}
		\sigma_{I}^{-1}
	\right)
	h(\sigma_{I})
		\begin{pmatrix}
			\tau     & 0 \\
			0     & \unitmatrix_{n-m} \\
			\unitmatrix_{m} & 0 \\
			0     & 0
		\end{pmatrix}
    (\Z_{2}^{n} \rtimes \Sgroup_{n}).
\end{align*}
Since $ \sigma_{I} \diag(\ee, \unitmatrix_{n - m}) \sigma_{I}^{-1} \in \Z_{2}^{n} $,
the last expression coincides with the image of $\tau$ in \eqref{eq-def-phi_I-bar-tau}.  

Next, we prove the surjectivity.
Let $\omega_{\tau_{1},\tau_{2}} \in \Omega^{I}_{n}$. 
We will prove that this $\omega_{\tau_{1},\tau_{2}} \in \Omega^{I}_{n}$ can be transformed into
\begin{equation*}
    h(\sigma_{I})
		\begin{pmatrix}
			z     & 0 \\
			0     & \unitmatrix_{n-m} \\
			\unitmatrix_{m} & 0 \\
			0     & 0
		\end{pmatrix}
\end{equation*}
for some signed partial involution $z \in \SPI_{m}$ (Definition \ref{Def-SPI-and-SPIst}) by the actions of $\Z_{2}^{n}$ and $\Z_{2}^{n} \rtimes \Sgroup_{n}$.
Recall that the actions of $\Z_{2}^{n}$ and $\Z_{2}^{n} \rtimes \Sgroup_{n}$ on $\Omega^{I}_{n}$ are given by
\begin{equation*}
	\varepsilon \omega_{\tau_{1},\tau_{2}} (\eta,w)
	=
	\begin{pmatrix}
		\varepsilon \tau_{1}\eta w \\
		\varepsilon \tau_{2}\eta w
	\end{pmatrix}
\end{equation*}
where $\varepsilon \in \Z_{2}^{n} $ and $(\eta,w) \in \Z_{2}^{n} \rtimes \Sgroup_{n}$.
Then, by the right action of $\Z_{2}^{n} \rtimes \Sgroup_{n}$, we can assume that 
\begin{equation*}
    \tau_{2} = \begin{pmatrix}
        e_{i_{1}} & e_{i_{2}} & \dots & e_{i_{m}} & 0 & \dots & 0
    \end{pmatrix},
\end{equation*}
where $I=\{i_{1} < \dots < i_{m}\}$.
Then, by the definition of $\sigma_{I}\in\Sgroup_{n}$ (Definition \ref{Def-sigma-I}), we have
\begin{equation}
    \sigma_{I}^{-1} \tau_{2}
    =
    \begin{pmatrix}
        \unitmatrix_{m} & 0 \\
        0     & 0
    \end{pmatrix},
    \quad
    \text{namely,}
    \quad
    \tau_{2}=
    \sigma_{I} 
    \begin{pmatrix}
        \unitmatrix_{m} & 0 \\
        0     & 0
    \end{pmatrix}.
\label{eq-def-sigma_{I}tau_2}
\end{equation}
Moreover, we define $z_{ij}\:(i,j\in\{1,2\})$ by
\begin{equation}
    \begin{pmatrix}
        z_{11} & z_{12} \\
        z_{21} & z_{22}
    \end{pmatrix}
    \coloneqq
    \sigma_{I}^{-1} \tau_{1},
		\quad
			\text{ so that }
		\quad
    \tau_{1}=
    \sigma_{I}
    \begin{pmatrix}
        z_{11} & z_{12} \\
        z_{21} & z_{22}
    \end{pmatrix}.
\label{eq-def-sigma_{I}tau_1}
\end{equation}
Since ${}^{t}\sigma_{I} =\sigma_{I}^{-1}$, we get
\begin{equation}
    {}^{t}\tau_{1}\tau_{2}
    =
    \begin{pmatrix}
        {}^{t}z_{11} & {}^{t}z_{21} \\
        {}^{t}z_{12} & {}^{t}z_{22}
    \end{pmatrix}
    {}^{t}\sigma_{I}
    \sigma_{I}
    \begin{pmatrix}
        \unitmatrix_{m} & 0 \\
        0     & 0
    \end{pmatrix}
    %%\\
=
    \begin{pmatrix}
        {}^{t}z_{11} & {}^{t}z_{21} \\
        {}^{t}z_{12} & {}^{t}z_{22}
    \end{pmatrix}
    \begin{pmatrix}
        \unitmatrix_{m} & 0 \\
        0     & 0
    \end{pmatrix}     
    =
    \begin{pmatrix}
        {}^{t}z_{11} & 0 \\
        {}^{t}z_{12} & 0 
    \end{pmatrix}
\label{eq-Lem-prf-tau_1^{*}tau_2}
\end{equation}
Moreover, $\omega_{\tau_{1},\tau_{2}}\in \regOmega_{n} $ implies that \eqref{eq-Lem-prf-tau_1^{*}tau_2} is an element of $\Sym_{n}(\Z)$,
and thus, we have 
\begin{equation}
    z_{11} = {}^{t}z_{11}
    \quad
    z_{12}=0.
\label{eq-z11=z11*-z12=0-tau}
\end{equation}
By combining this equation \eqref{eq-z11=z11*-z12=0-tau} with \eqref{eq-def-sigma_{I}tau_2} and \eqref{eq-def-sigma_{I}tau_1},
we have
\begin{equation*}
    \omega_{\tau_{1},\tau_{2}}
    =
    \begin{pmatrix}
        \tau_{1} \\ \tau_{2}
    \end{pmatrix}
    =
    h(\sigma_{I})
    \begin{pmatrix}
        z_{11} & 0 \\
        z_{21} & z_{22} \\
        \unitmatrix_{m}  & 0 \\
        0      & 0
    \end{pmatrix}.
\end{equation*}
Note that the multiplication by $h(\sigma_{I})\in H \subset \GL_{2n}(\Q) $ preserves the rank of $\Mat_{2n,n}(\Q)$.
Thus, because $\operatorname{rank} \omega_{\tau_{1},\tau_{2}} = n$ by $\omega_{\tau_{1},\tau_{2}}\in \Omega^{I}_{n} \subset \regOmega_{n} $,
we have
\begin{equation*}
\operatorname{rank}
      \begin{pmatrix}
        z_{11} & 0 \\
        z_{21} & z_{22} \\
        \unitmatrix_{m}  & 0 \\
        0      & 0
    \end{pmatrix}  
= \operatorname{rank} \omega_{\tau_{1},\tau_{2}}
= n,
\end{equation*}
which implies that $z_{22}$ is invertible.
Therefore, by multiplying appropriate element of $\Sgroup_{n}$ from the right, we can assume that
\begin{equation*}
    \omega_{\tau_{1},\tau_{2}} =
    \begin{pmatrix}
        \tau_{1} \\ \tau_{2}
    \end{pmatrix}
    = 
    h(\sigma_{I})
    \begin{pmatrix}
        z_{11} & 0       \\
        0      & \unitmatrix_{n-m} \\
        \unitmatrix_{m}  & 0       \\
        0      & 0
    \end{pmatrix}.
\end{equation*}
This completes the proof of the surjectivity of the map \eqref{eq-def-phi_I-bar-tau} because
\eqref{eq-z11=z11*-z12=0-tau} implies $z_{11}\in \Sym_{m}(\Z) \cap \SPP_{m} = \SPI_{m}$. 
%%(Definition \ref{Def-SPI-and-SPIst}).

Finally, we prove the injectivity of the map \eqref{eq-def-phi_I-bar-tau}. Let $\tau,\tau'\in\SPI_{m}$.
Suppose
\begin{equation*}
	\Z_{2}^{n} h(\sigma_{I})
		\begin{pmatrix}
			z     & 0       \\
			0     & \unitmatrix_{n-m} \\
			\unitmatrix_{m} & 0       \\
			0     & 0
		\end{pmatrix}
	(\Z_{2}^{n}\rtimes \Sgroup_{n})
	=
	\Z_{2}^{n} h (\sigma_{I})
		\begin{pmatrix}
			z'    & 0       \\
			0     & \unitmatrix_{n-m} \\
			\unitmatrix_{m} & 0       \\
			0     & 0
		\end{pmatrix}
	(\Z_{2}^{n} \rtimes \Sgroup_{n}).
\end{equation*}
In other words, suppose that there exist $\varepsilon \in \Z_{2}^{n}$ and $(\eta,w)\in \Z_{2}^{n}\rtimes \Sgroup_{n}$ such that
\begin{equation}
	h(\varepsilon)
	h(\sigma_{I})
	\begin{pmatrix}
        z     & 0       \\
		0     & \unitmatrix_{n-m} \\
		\unitmatrix_{m} & 0       \\
		0     & 0
	\end{pmatrix}
	(\eta, w)
	=
	h(\sigma_{I})
	\begin{pmatrix}
		z'    & 0 \\
		0     & \unitmatrix_{n-m} \\
		\unitmatrix_{m} & 0 \\
		0     & 0
	\end{pmatrix}.
\label{eq-inj-comparison-tau}
\end{equation}
The left-hand side of this equation \eqref{eq-inj-comparison-tau} is equal to
\begin{equation}
	h(\varepsilon) h(\sigma_{I})
	\begin{pmatrix}
		z     & 0 \\
		0     & \unitmatrix_{n-m} \\
		\unitmatrix_{m} & 0 \\
		0     & 0
	\end{pmatrix}
	(\eta,w)
    =
    h(\sigma_{I}) h (\sigma_{I}^{-1} \varepsilon \sigma_{I})
	\begin{pmatrix}
		z     & 0 \\
		0     & \unitmatrix_{n-m} \\
		\unitmatrix_{m} & 0 \\
		0     & 0
	\end{pmatrix}
    \eta w.
\label{eq-inj-comparison-2-tau}
\end{equation}
In order to compute \eqref{eq-inj-comparison-2-tau}, we set
\begin{equation}
    \sigma_{I}^{-1} \varepsilon \sigma_{I}
    \eqqcolon
    \begin{pmatrix}
        \varepsilon' & 0 \\
        0 & \varepsilon''
    \end{pmatrix},
    \quad
    \eta \eqqcolon
    \begin{pmatrix}
        \eta' & 0 \\
        0 & \eta''
    \end{pmatrix}.
\label{eq-def-varepsilon-eta}
\end{equation}
Then, the right-hand side of \eqref{eq-inj-comparison-2-tau} is equal to
\begin{align*}
    h(\sigma_{I})
	\begin{pmatrix}
		\varepsilon'  & 0 & 0 & 0 \\
		0 & \varepsilon'' & 0 & 0 \\
		0 & 0             & \varepsilon' & 0\\
		0 & 0             & 0 & \varepsilon''\\
	\end{pmatrix}
	\begin{pmatrix}
		z     & 0 \\
		0     & \unitmatrix_{n-m} \\
		\unitmatrix_{m} & 0 \\
		0     & 0
	\end{pmatrix}
    \begin{pmatrix}
        \eta' & 0 \\
        0     & \eta''
    \end{pmatrix} w
    &=	
    h(\sigma_{I})
    \begin{pmatrix}
        \varepsilon'z\eta' & 0 \\
        0                  & \varepsilon''\eta'' \\
        \varepsilon'\eta'  & 0 \\ 
        0                  & 0
    \end{pmatrix} w.
\end{align*}
Because we supposed that this is equal to the right-hand side of \eqref{eq-inj-comparison-tau}, we have
\begin{equation}
    \begin{pmatrix}
        \varepsilon'z\eta' & 0 \\
        0                  & \varepsilon''\eta'' \\
        \varepsilon'\eta'  & 0 \\ 
        0                  & 0
    \end{pmatrix} w	
=	
	\begin{pmatrix}
		z'    & 0 \\
		0     & \unitmatrix_{n-m} \\
		\unitmatrix_{m} & 0 \\
		0     & 0
	\end{pmatrix}.
\label{eq-comparizon-2-tau}
\end{equation}
By comparing entries, we have
\begin{equation*}
	w = 1,\quad
    \varepsilon'=\eta',\quad
    \varepsilon''=\eta''.
\end{equation*}
Thus, we have
\begin{equation*}
    \begin{pmatrix}
        \varepsilon'z\varepsilon'  &  0  \\
        0                        & \unitmatrix_{n-m}\\
        \unitmatrix_{m}                    & 0 \\
        0                        & 0
    \end{pmatrix}	
    =	
	\begin{pmatrix}
		z'    & 0 \\
		0     & \unitmatrix_{n-m} \\
		\unitmatrix_{m} & 0 \\
		0     & 0
	\end{pmatrix},
\end{equation*}
which completes the proof.
\end{proof}

Summarizing, we get 

\begin{lemma}\label{Lem-ZSPI=coprod-ZOmegaZS}
%%    Let $n\in\Z_{\geq 0}$ and set $[n]\coloneqq \{1,2,\dots,n\}$. 
    The union of the maps
    \begin{equation*}
        \begin{array}{ccc}
            \varphi_{I}' \colon
             \Z_{2}^{\# I} \backslash \SPI_{\# I} & \xrightarrow{\sim} &  \Z_{2}^{n} \backslash \Omega^{I}_{n} / (\Z_{2}^{n} \rtimes \Sgroup_{n})
             \\
             \rotatebox{90}{$\in$} & & \rotatebox{90}{$\in$}
             \\
             \tau & \mapsto & 
             \Z_{2}^{n} h(\sigma_{I})
             \begin{pmatrix}
                 \tau & 0 \\ 
                 0    & \unitmatrix_{n-\# I} \\
                 \unitmatrix_{\# I}& 0 \\
                 0    & 0
             \end{pmatrix}
             (\Z_{2}^{n} \rtimes \Sgroup_{n})
        \end{array}
    \end{equation*}
    defines a bijection 
    \begin{equation*}
        \coprod_{I \subset [n]} \varphi_{I}' 
        \colon
        \coprod_{I\subset [n]} \Z_{2}^{\# I} \backslash \SPI_{\# I} 
        \xrightarrow{\;\sim\;}
        \coprod_{I\subset [n]} \Z_{2}^{n} \backslash \Omega^{I}_{n} / (\Z_{2}^{n} \rtimes \Sgroup_{n}) 
        = 
        \Z_{2}^{n} \backslash \regOmega_{n} / (\Z_{2}^{n} \rtimes \Sgroup_{n}).  
    \end{equation*}
\end{lemma}

%%%%%%%% skipover:begin %%%%%%%%%%%%%%%%%%
\skipover{
\begin{proof}
    This follows from Proposition \ref{Prop-Omega-I=SPI} and Lemma \ref{Lem-orbit-decomp-Omega=coprod-Omega-I}.
\end{proof}
}
%%%%%%%% skipover:end   %%%%%%%%%%%%%%%%%%

Now let us give a proof of Theorem \ref{theorem-Omega=H-G-PS} in Introduction.  
This is one of our main theorems in this article.  
Recall that $ H $-orbits on the double flag variety $ \dblFV $ is in bijection with 
$ B_{H} \backslash G / P_{G}  $.

\begin{theorem}\label{theorem-Omega=H-G-PS-2}
Let $(G,H,P_{G},B_{H})$ be in Case \caseA or \caseB of Setting \ref{Setting-DFVR}. Then, we have bijections
    \begin{equation}
    \begin{array}{cccc}
         \Z_{2}^{n} \backslash \regOmega_{n} / ( \Z_{2}^{n} \rtimes \Sgroup_{n}  ) 
         & \xrightarrow{\;\;\sim\;\;} & 
         \makebox[10ex][l]{$B_{H} \backslash \LregMat_{2n,n} / \GL_{n}(\C) \simeq B_{H} \backslash G / P_{G}  \simeq H \backslash \dblFV$}
         & \hspace*{.35\textwidth}
         \\
         \rotatebox{90}{$\in$} && \rotatebox{90}{$\in$}
         \\
         \omega_{\tau_{1},\tau_{2}}
         & \longmapsto &
         B_{H} [\omega_{\tau_{1},\tau_{2}}]
    \end{array}
    \label{eq-map-theorem-Omega=H-H-PS-2}
    \end{equation}
    where $\regOmega_{n}$ is defined in Definition \ref{Def-SPP-and-regOmega} and the latter isomorphism is given by Lemma \ref{Lem-BGP=BLMGL}.
\end{theorem}

\begin{proof}
We only give a proof of Case \caseA 
%%of Setting \ref{Setting-DFVR}
because the same proof works for Case \caseB.
Recall that there exists an inclusion
    \begin{equation}
        \regOmega_{n} \hookrightarrow \LregMat_{2n,n}
    \label{eq-inclusion-regOmega-to-LregMat}
    \end{equation}
by the definitions of $\regOmega_{n}$ and $\LregMat_{2n,n}$ (Definitions \ref{Def-SPP-and-regOmega} and \ref{Def-LM-circ}, respectively).
Obviously, this inclusion induces the map \eqref{eq-map-theorem-Omega=H-H-PS-2}.
Moreover, the definitions of the maps $\varphi_{I}$ and $\varphi_{I}'$ defined in Lemmas \ref{Lem-B-G-P_S=coprod-B-Her} and \ref{Lem-ZSPI=coprod-ZOmegaZS} imply 
that we have a commutative diagram 
\begin{equation*}
    \xymatrix@C=50pt{
    \coprod_{I\subset [n]} B_{\# I} \backslash \Her_{\# I}(\C) 
    \ar[r]^-{\sim}_-{\coprod \varphi_{I}}
    &
    B_{H} \backslash \LregMat_{2n,n} / \GL_{n}(\C)
    \\
    \coprod_{I\subset [n]} \Z_{2}^{\# I} \backslash \SPI_{\# I}
    \ar[u]^-{\eqref{eq-Z2-SPI=SPIst}, \eqref{eq-map-from-SPIst-to-B-Her(C)}}_-{\wr}
    \ar[r]^-{\sim}_-{\coprod \varphi_{I}'}
    &
    \Z_{2}^{n} \backslash \regOmega_{n} / (\Z_{2}^{n} \rtimes \Sgroup_{n}),
    \ar[u]^-{\eqref{eq-map-theorem-Omega=H-H-PS-2}}
    }
\end{equation*}
which shows that \eqref{eq-map-theorem-Omega=H-H-PS-2} is an isomorphism.
\end{proof}

In the above theorem, the symmetry group 
$ \Z_{2}^{n} \rtimes \Sgroup_{n} $ is isomorphic to the type C Weyl group $ W(C_n) $, which 
is the little Weyl group of both of $ G = \U(n, n) $ and $ \Sp_{2n}(\R) $.  
Though there is not a nice interpretation of $ \regOmega_n $ using (restricted) root systems up to now, 
we are expecting such kind of description for the orbit spaces.

\section{Combinatorial descriptions of orbits by graphs}\label{Section:orbit-graphs}

Let us give a combinatorial description of the double coset space $ B_{H} \backslash G /P_{S}$ or 
$ \dblFV/H $ using certain kind of graphs.  
This is a natural generalization of the graphs in \cite{Fresse.N.2023,Fresse.N.2020}.  

Recall the set $ \Gamma(n) $ of decorated graphs with $ n $-vertices from Introduction.  
In a graph $ \gamma \in \Gamma(n) $, each vertex of $ \gamma $ is decorated by $ \{ \pm, c, d \} $ or arcs (loops are not allowed).  In other words, a graph $ \gamma $ is corresponding to a map 
\begin{equation*}
    \gamma : [n] \to [n] \cup \{ +, -, c, d \} ,
\end{equation*}
with the property that 
\begin{itemize}
    \item[\upshape{(arc)}]
    $\gamma(i)=j \in [n]$ implies $i\neq j$ and $\gamma(j)=i$.
\end{itemize}
This condition corresponds to an arc connecting two different vertices $ i $ and $ j $.  
For a subset $ I \subset [n] $, we also define 
\begin{equation}\label{eq-Def-Gamma'(I)}
    \Gamma'(I) = \{ \gamma' : I \to I \cup \{ +, -, c \} \mid \text{with the property (arc)}  \} .
\end{equation}
Then $ \gamma' \in \Gamma'(I) $ extends to $ \gamma \in \Gamma(n) $ by putting 
$ \gamma(k) = d $ for $ k \in [n] \setminus I $.  
In this way, we always identify $ \Gamma'(I) $ with a subset of $ \Gamma(n) $.  Clearly we get 
\begin{equation*}
    \Gamma(n) = \coprod_{I \in [n]} \Gamma'(I).
\end{equation*}

\begin{example}\label{Ex-Gamma'(m)-to-Gamma(n)}
Let $ n = 7 $ and 
consider $ I = [7] \setminus \{ 1 \} = \{2,3,4,5,6,7\} $.  
%%Let $\gamma' \in \Gamma'(I)$ be the same as in Example \ref{Ex-Gamma'(m)} 
Define $ \gamma' \in \Gamma'(I) $ as a map 
$\gamma'\colon I \to I \cup \{+,-,c\}$ by
\begin{equation*}
        \gamma'(2) = -,\quad
        \gamma'(3) = c,\quad
        \gamma'(4) = 6,\quad
        \gamma'(5) = +,\quad
        \gamma'(6) = 4,\quad
        \gamma'(7) = -.
\end{equation*}
%
% Its graphical representation is given by
% \begin{equation*}
%     \gamma'
%     =
%     \xymatrix{
%         2_{-} &
%         3_{c} &
%         4 \ar@{-}@/^10pt/[rr] &
%         5_{+} &
%         6     &
%         7_{-}.
%     }
% \end{equation*}
%
If we extend this $ \gamma'$ to $ \gamma \in \Gamma(n) $ indicated above, 
we obtain the $\gamma$ given in Example \ref{Ex-Gamma(n)}.
Namely the graphical representations of these $\gamma$ and $\gamma'$ are given by
     \begin{equation*}
     \begin{array}{c}
        \xymatrix{
            \gamma' = &
            \:    &
        2_{-} &
        3_{c} &
        4 \ar@{-}@/^10pt/[rr] &
        5_{+} &
        6     &
        7_{-},
            \\
            \gamma = &
            1_{d} &
            2_{-} &
            3_{c} &
            4 \ar@{-}@/^10pt/[rr] &
            5_{+} &
            6     &
            7_{-}.
        }
     \end{array}
     \end{equation*}    
\end{example}

We get a graphical interpretation of our main theorem, Theorem \ref{theorem-Omega=H-G-PS-2}.

\begin{theorem}\label{Thm-Gamma(n)=DFVR}
    Let $(G,H,P_{G},P_{H})$ be in Case \caseA or \caseB in Setting \ref{Setting-DFVR}. Then, we have
    \begin{equation*}
        \Gamma(n)
        \simeq
        \Z_{2}^{n} \backslash \regOmega_{n} / (\Z_{2}^{n} \rtimes \Sgroup_{n})
        \simeq 
        B_{H} \backslash G /P_{S} \simeq H \backslash \dblFV ,
    \end{equation*}
    where $\Gamma(n)$ and $\regOmega_{n}$ are defined in Definitions \ref{Def-Gamma(n)} and \ref{Def-SPP-and-regOmega}, respectively.
\end{theorem}

\begin{proof}
    From $ \omega_{\tau_1, \tau_2} \in \regOmega_{n}$, we associate a graph $ \gamma \in \Gamma(n) $ by the following manner.  We denote the $ j $-th column of a matrix $ \tau $ by $ \tau(j) $ in this proof.

\begin{itemize}
\item[\upshape{(i)}]\ 
If $ \tau_1(j) = 0 $, we must have $ \tau_2(j) = \pm e_k $ for some $ k $.  In this case we put $ \gamma(k) = c $.  

\item[\upshape{(ii)}]\ 
Similarly, if $ \tau_2(j) = 0 $, we must have $ \tau_1(j) = \pm e_k $ for some $ k $.  We put $ \gamma(k) = d $. 
\end{itemize}

If $ \tau_1(j) $ and $ \tau_2(j) $ are both nonzero, there are $ k_1, k_2 \in [n] $ such that 
$ \tau_1(j) = \delta_1 e_{k_1} $ and $ \tau_2(j) = \delta_2 e_{k_2} $ for some 
$ \delta_1, \delta_2 = \pm 1 $.  

\begin{itemize}
\item[\upshape{(iii)}]\ 
If $ k_1 = k_2 $, we put $ \gamma(k) = \delta_1 \delta_2 $ for $ k = k_1 = k_2 $. 

\item[\upshape{(iv)}]\ 
If $ k_1 \neq k_2 $, we put $ \gamma(k_1) = k_2 $ and $ \gamma(k_2) = k_1 $.  
\end{itemize}

It is obvious that the procedure (i)--(iv) determines $ \gamma \in \Gamma(n) $ and 
the image $ \gamma $ does not depend on the choice of $ \omega_{\tau_1, \tau_2} $ 
from an orbit of $ \Z_2^n \times (\Z_2^n \ltimes \Sgroup_n) $.  
Thus we have a well-defined map $ \Z_{2}^{n} \backslash \regOmega_{n} / (\Z_{2}^{n} \rtimes \Sgroup_{n}) \to \Gamma(n) $.  

This map is clearly surjective (we can reconstruct $ \omega_{\tau_1, \tau_2} $).  
The injectivity can be shown in the same way as Proposition \ref{Prop-Omega-I=SPI}.  
%%\S \ref{Section-BGP-via-SPP}.  
We omit the details.    
\end{proof}

\begin{example}\label{Ex-Gamma-to-Omega}
    Let $\gamma\in\Gamma(7)$ be the same as Example \ref{Ex-Gamma(n)}.
    Then, this $\gamma$ corresponds to $\omega_{\tau_{1},\tau_{2}}\in\regOmega_{n}$ under the isomorphism in Theorem \ref{Thm-Gamma(n)=DFVR} given by
    \begin{equation*}
        \tau_{1}=
        \begin{pmatrix}
            1 & 0 & 0 & 0 & 0 & 0 & 0 \\
            0 &-1 & 0 & 0 & 0 & 0 & 0 \\
            0 & 0 & 0 & 0 & 0 & 0 & 0 \\
            0 & 0 & 0 & 0 & 0 & 1 & 0 \\
            0 & 0 & 0 & 0 & 1 & 0 & 0 \\
            0 & 0 & 0 & 1 & 0 & 0 & 0 \\
            0 & 0 & 0 & 0 & 0 & 0 & -1 \\
        \end{pmatrix},
        \quad
        \tau_{2}=
        \begin{pmatrix}
            0 & 0 & 0 & 0 & 0 & 0 & 0 \\
            0 & 1 & 0 & 0 & 0 & 0 & 0 \\
            0 & 0 & 1 & 0 & 0 & 0 & 0 \\
            0 & 0 & 0 & 1 & 0 & 0 & 0 \\
            0 & 0 & 0 & 0 & 1 & 0 & 0 \\
            0 & 0 & 0 & 0 & 0 & 1 & 0 \\
            0 & 0 & 0 & 0 & 0 & 0 & 1 \\
        \end{pmatrix},
    \end{equation*}
\end{example}

Let us fix $ I \subset [n] $ and denote $ m = \# I $.  

\begin{lemma}
There is a bijection between $ \Gamma'(I) $ and 
$ \Z_2^n \times (\Z_2^n \ltimes \Sgroup_n) $-orbits in 
$ \Omega_n^I $, which classifies the $ B_H $-orbits 
$ B_H \backslash \mathrm{LM}_{2n, n}^I / \GL_n(\C) $ in Case \caseA.  
While in Case \caseB, it classifies $ B_H \backslash \mathrm{LM}_{2n, n}^I(\R) / \GL_n(\R) $,
where $\mathrm{LM}_{2n, n}^I(\R) \coloneqq \mathrm{LM}_{2n, n}^{I} \cap \Mat_{2n,n}(\R)$.
\end{lemma}

\begin{proof}
    This directly follows from 
    Lemma \ref{Lem-(B-Her)=(B-LM-I/GL)}, Proposition \ref{Prop-SPI=(B-Her)}
    and Proposition \ref{Prop-Omega-I=SPI}. 
    In fact, we have already established the following bijections:
\begin{equation*}
    \Gamma'(I) \simeq \Z_{2}^{n} \backslash \Omega_{n}^I / (\Z_{2}^{n} \rtimes \Sgroup_{n}) 
    \simeq B_H \backslash \mathrm{LM}_{2n, n}^I / \GL_n(\C) \simeq B_m \backslash \Her_m(\C) 
    \simeq \Z_{2}^{m} \backslash \SPI_m .
\end{equation*}
\end{proof}

\section{Galois cohomology of orbits}\label{Section-proof-Thm-Galios-RDFV}

As explained in Introduction, the orbit decomposition for double flag varieties 
\begin{align*}
\dblFV_{\C} &= \bigl(\GL_n(\C) \times \GL_n(\C) \bigr) \bigm/ \bigl( B_n^+(\C) \times B_n^-(\C) \bigr) \times \GL_{2n}(\C)/ P_{(n,n)} 
\quad & & (\text{Case \caseA)}
\\
\intertext{ and }
\dblFV_{\C} &= \GL_n(\C) / B_n(\C) \times \Sp_{2n}(\C)/P_S^{\C}
\quad& & (\text{Case \caseB)}
\end{align*}
over complex numbers are already known (\cite{Fresse.N.2023,Fresse.N.2021}).  
Our double flag varieties over $ \R $ are real forms of these complex double flag varieties 
and we are interested in the decomposition of a complex orbit on $ \dblFV_{\C} $ 
when restricted to the real forms.  
In this section, we study those real orbits in $ \dblFV $ using Galois cohomology and 
prove Theorem \ref{theorem-Galois-RDFV}.

\subsection{\texorpdfstring{$\R$}{R}-rational point of \texorpdfstring{$B_{H_{\C}}$}{B}-orbit on \texorpdfstring{$G_{\C}/P_{\C}$}{the complex flag variety}}\label{Section-proof-of-(1)-of-Thm-RDFV}

In this subsection, we prove \eqref{theorem-Galois-RDFV:item:1:nonempty.R.pts} of Theorem \ref{theorem-Galois-RDFV}.  
%%which is essentially proved in \cite[Fact 3.1]{Nishiyama.Tauchi.2024}.
Before the proof, we note that Theorem \ref{Thm-complex-orbit-decomp-AIII-and-CI} implies that there exists a commutative diagram
\begin{equation}\label{eq-diagram-of-regR-regT-flags}
\vcenter{\xymatrix{
    \regR_{n} / \Sgroup_{n} 
    \ar@{^{(}-_>}[r]
    \ar[d]^-{\wr}_{\text{Thm.}~\ref{Thm-complex-orbit-decomp-AIII-and-CI}}
    &
    \regT_{n} / \Sgroup_{n}
    \ar[d]^-{\wr}_{\text{Thm.}~\ref{Thm-complex-orbit-decomp-AIII-and-CI}}
    \\
    B_{n}(\C) \backslash \Sp_{2n}(\C) / P_{S}^{\C}
    \ar@{^{(}-_>}[r]
    &
    B_{n}^{+} \times B_{n}^{-} \backslash \GL_{2n}(\C) / P_{(n,n)} .
}} 
\end{equation}
%%where we use Notation \ref{Convention-complexifications}.

%%\begin{proof}[Proof of (1) in Theorem \ref{theorem-Galois-RDFV}]
First, suppose that we are in Case \caseA.
    Let $\tau \in \regT_{n} / \Sgroup_{n}$ and $\mathcal{O}_{\tau}$ be the $B_{H_{\C}}$-orbit on $G_{\C} / P_{G_{\C}}$ corresponding to $\tau$ via the isomorphism of Theorem \ref{Thm-complex-orbit-decomp-AIII-and-CI}. 
    By \cite[Fact 3.1]{Nishiyama.Tauchi.2024}, we know $\mathcal{O}_{\tau}(\R) \neq \emptyset $ if and only if $\mathcal{O}_{\tau} \cap (\Sp_{2n}(\C) / P_{S}^{\C}) \neq \emptyset$.   The above commutative diagram tells that the last claim is equivalent to 
    $\tau \in \regR_{n} / \Sgroup_{n}$.  
%%    (see the diagram \eqref{eq-diagram-of-regR-regT-flags}).  
This completes the proof of \eqref{theorem-Galois-RDFV:item:1:nonempty.R.pts} in Theorem \ref{theorem-Galois-RDFV} in Case \caseA.  
    
Next, suppose that we are in Case \caseB.
    Note that all entries of $\tau \in \regR_{n} / \Sgroup_{n}$ belongs to $\Z \subset \R$, and thus, $\tau $ is a real point because the complex conjugation of $G_{\C} = \Sp_{2n}(\C)$ with respect to $G=\Sp_{2n}(\R)$ is merely the usual complex conjugation of each entries. Thus, every orbit $\mathcal{O}_{\tau}$ corresponding to $\tau \in \regR_{n} / \Sgroup_{n}$ has a real point $\tau$ itself. 
    This completes the proof of \eqref{theorem-Galois-RDFV:item:1:nonempty.R.pts} in Theorem \ref{theorem-Galois-RDFV} in Case \caseB.
%%\end{proof}

\subsection{Matrix realization of the Grassmannian}
%%Preparation from the orbit decomposition $B_{H_{\C}} \backslash G_{\C} / P_{\C}$

We only prove \eqref{theorem-Galois-RDFV:item:2:Galois.tells.orbits.R} of Theorem \ref{theorem-Galois-RDFV} 
in Case \caseA, but the statements for Case \caseB hold without changes.  
So let us assume we are in Case \caseA till the end of this section.   We need some preparations.  

First, we give a matrix realization of the Grassmannian $G_{\C} / P_{\C}$, which is an analogue of Lemma \ref{Lem-G/P_S=LM/GL}.

\begin{definition}\label{Def-regM}
For $n\in \Z_{\geq 0}$, we define $\regMat_{2n,n} \subset \Mat_{2n,n}(\C)$ by
\begin{equation*}
    \regMat_{2n,n} \coloneqq	
    \left\{
        \omega_{C,D}=
        \begin{pmatrix}
            C \\ D
        \end{pmatrix}       \in \Mat_{2n,n}(\C)
    \Bigm| 
         \begin{array}{c}
              C,D\in\Mat_{n}(\C) \\
              % C^{*}D \in  \Her_{n}(\C) \\
              \rank \omega_{C,D} = n
         \end{array}
    \right\}.
\end{equation*}
\end{definition}

\begin{lemma}\label{Lem-G/P=M/GL}
   Let $n\in\Z_{\geq 0}$. Then, we have a $G_{\C}$-equivariant isomorphism
   \begin{equation*}
   \begin{array}{ccc}
        \regMat_{2n,n} / \GL_{n}(\C) & \xrightarrow{\sim} & G_{\C}/P_{G_{\C}}  \\
        \rotatebox{90}{$\in$} && \rotatebox{90}{$\in$} \\
        \omega & \mapsto & [\omega],
   \end{array}
   \end{equation*}
   where $[\omega] \in G_{\C}/P_{G_{\C}} \simeq \Gr_{n}(\C^{2n})$ is the $n$-dimensional subspace of $\C^{2n}$ defined by the image of $\omega \colon \C^{n} \to \C^{2n}$.  
    By this isomorphism, the action of $B_{n}^{+} \times B_{n}^{-}$ on $[\omega] \in G_{\C} /P_{G_{\C}}$
    % \; (\tau \in \regR_{n}) $ 
    is given by 
    \begin{equation}\label{eq-def-action-of-BB-on-regRst}
        (b_1,b_2) \cdot 
        \begin{pmatrix}
            C \\ D
        \end{pmatrix} \GL_{n}(\C)
        =
        \begin{pmatrix}
            b_1 C \\ b_2 D
        \end{pmatrix} \GL_{n}(\C) , 
        \quad
        \text{ where } \;\;
        \omega = \omega_{C,D} = \begin{pmatrix}
        C \\ D
    \end{pmatrix} \in \regMat_{2n,n}.
    \end{equation}
\end{lemma}

\begin{proof}
We omit the proof.  See the similar arguments for Lemma \ref{Lem-G/P_S=LM/GL}.
\end{proof}

Next, we need to prepare some notations about the parameter set $\regR_{n}$ (Definition \ref{Def-PP-regT-regR}).

\begin{notation}\label{Notation-K-for-tau-in-regR}
    Let $\tau \in \regR_{n}$.
    %%(Definition \ref{Def-PP-regT-regR}). 
    Then, by definition of $\regR_{n}$, we can write 
    \begin{equation*}
        \tau = 
        \begin{pmatrix}
            \tau_{1} \\ \tau_{2}
        \end{pmatrix}
        =
        \begin{pmatrix}
            e_{i_{1}} & e_{i_{2}} & \dots & e_{i_{n}} \\
            e_{j_{1}} & e_{j_{2}} & \dots & e_{j_{n}} \\ 
        \end{pmatrix},
        \qquad
        \text{ for } i_{k},j_{k} \in \{0,1,\dots,n\},
    \end{equation*}
    where $e_{\ell} \in \Mat_{1,n}(\C)$ denotes the $\ell$-th elementary basis  vector for $1\leq \ell \leq n$, and we set $e_{0} \coloneqq 0 \in \Mat_{1,n}(\C)$. 
    In this setting, we define $K_{c}, K_{d}, K_{1}, K_{2}$ for subsets of $[n] = \{1,2,\dots,n\}$ as follows.
    \begin{equation*}
    \begin{split}
        K_{c} & \coloneqq \{j_{k} \in [n] \mid i_{k} = 0 \}, \\
        K_{d} & \coloneqq \{i_{k} \in [n] \mid j_{k} = 0 \}, \\
        K_{1} & \coloneqq \{i_{k} \in [n] \mid i_{k} = j_{k} \neq 0\}, \\
        K_{2} & \coloneqq \{i_{k} \in [n] \mid i_{k} = j_{l} \neq 0 \neq  j_{k} = i_{l} \text{ for some } l \neq k\}.
    \end{split}
    \end{equation*}
    By definition, they satisfy $K_{c} \coprod K_{d} \coprod K_{1} \coprod K_{2} = [n]$.  
    %%(see Example \ref{Ex-K-for-tau-in-regR} and Remark \ref{remark-K-for-tau-in-regR})
    Note that these sets are invariant with respect to the action of $\Sgroup_{n}$ on $\regR_{n}$ given in Remark \ref{Remark-Weyl-group-action-An-on-regT-regR}, and thus we will also consider this decomposition for an element in $\regR_{n}/\Sgroup_{n}$.
\end{notation}

\begin{example}\label{Ex-K-for-tau-in-regR}
    Consider an element $\tau = \begin{pmatrix}
        \tau_{1} \\ \tau_{2}
    \end{pmatrix} \in \regR_{5}$ given by
\begin{equation*}
    \tau_{1} =
    \begin{pmatrix}
        0 & 0 & 0 & 0 & 1 \\
        0 & 0 & 0 & 1 & 0 \\
        1 & 0 & 0 & 0 & 0 \\
        0 & 0 & 0 & 0 & 0 \\
        0 & 1 & 0 & 0 & 0 \\
    \end{pmatrix},
    \qquad
    \tau_{2} =
    \begin{pmatrix}
        0 & 0 & 0 & 0 & 1 \\
        0 & 0 & 0 & 0 & 0 \\
        0 & 1 & 0 & 0 & 0 \\
        0 & 0 & 1 & 0 & 0 \\
        1 & 0 & 0 & 0 & 0 \\        
    \end{pmatrix}.
\end{equation*}
    Then, the decomposition $[5] = K_{c} \coprod K_{d} \coprod K_{1} \coprod K_{2}$ in Notation \ref{Notation-K-for-tau-in-regR} is given by
    \begin{equation*}
        K_{c} = \{4\}, \quad 
        K_{d} = \{2\}, \quad 
        K_{1} = \{1\}, \quad 
        K_{2} = \{3,5\}.
    \end{equation*}
%
%%\begin{remark}\label{remark-K-for-tau-in-regR}
One can easily determine the above decomposition of $[n]$ 
%%given in Notation \ref{Notation-K-for-tau-in-regR} 
using the graphic representation $\mathcal{G}(\tau)$ given in \cite[Sect.~2.1]{Fresse.N.2023}.  
For $\tau\in \regR_{5}$ considered above 
%%given in Example \ref{Ex-K-for-tau-in-regR}, 
the graphic representation $\mathcal{G}(\tau)$ is
    
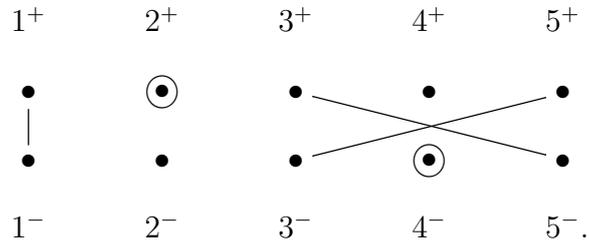
\begin{equation*}
        \xymatrix@C=30pt@R=10pt{
        1^{+}
		&
        2^{+}
		&
        3^{+} 
		&
        4^{+} 
		&
		5^{+}
		\phantom{.} 
        \\
        \bullet \ar@{-}[d]
		&
        \text{\textcircled{$\bullet$}}
		&
        \bullet \ar@{-}[drr]
		&
        \bullet 
		&
		\bullet \ar@{-}[dll]
		\phantom{.} 
		\\
		\bullet
		&
        \bullet 
		&
        \bullet 
		&
        \text{\textcircled{$\bullet$}} 
		&
		\bullet
        \phantom{.}
        \\
        1^{-}
		&
        2^{-}
		&
        3^{-} 
		&
        4^{-} 
		&
		5^{-}.
        }
    \end{equation*}
    Namely, the sets $K_{c}, K_{d}, K_{1}$ and $K_{2}$ are equal to the sets of indices corresponding to the diagram
    
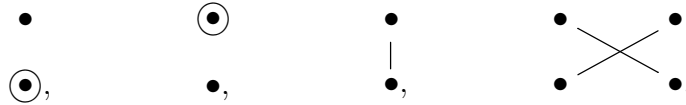
\begin{equation*}
        \xymatrix@C=30pt@R=10pt{\bullet\phantom{,} \\ \text{\textcircled{$\bullet$}},}\qquad\qquad
        \xymatrix@C=30pt@R=10pt{\text{\textcircled{$\bullet$}}\phantom{,}   \\ \bullet,}\qquad\qquad
        \xymatrix@C=30pt@R=10pt{\bullet\phantom{,}  \ar@{-}@<-2pt>[d] \\ \bullet,}\qquad\qquad
        \xymatrix@C=30pt@R=10pt{\bullet  \ar@{-}[dr] & \bullet\phantom{,}  \ar@{-}[dl]\\ \bullet & \bullet,}
    \end{equation*}
    respectively.
%%\end{remark}
\end{example}

\subsection{Basics from the theory of Galois cohomology}\label{Section-facts-Galois-cohomology-flag}
To obtain the orbit decomposition $B_{H} \backslash G /P_{G}$ over $ \R $, 
it is sufficient to know the following (i)--(iii): 
\begin{itemize}
\item[\upshape{(i)}] the orbit decomposition of its complexification $B_{H_{\C}} \backslash G_{\C} /P_{G_{\C}}$, 
\item[\upshape{(ii)}] identifying which orbit has an $\R$-rational point among orbits in the complexification, and 
\item[\upshape{(iii)}] the Galois cohomologies of stabilizer subgroups of such orbits.
\end{itemize}
The precise statements will be given in Lemma \ref{lemma-real-orbits-via-Galois-cohomology-non-homogeneous-case-flag} below.  Before that we need some basics of the theory of Galois cohomology.  
The following lemmas seem to be well known, and therefore we only give appropriate references.

\begin{lemma}\label{Lem-Galois-coh=gamma-conjugation-flag}
    Let $B$ be a real algebraic group, $B_{\C}$ its complexification, and $\gamma \colon B_{\C} \to B_{\C}$ the complex conjugation of $B_{\C}$ with respect to $B$. 
    Write $B_{\C}^{-\gamma}$ for the set of $(-\gamma)$-fixed points of $B_{\C}$:
    \begin{equation}
        B_{\C}^{-\gamma} \coloneqq \{b\in B_{\C} \mid \gamma(b)^{-1} = b\}.
    \label{def-eq--gamma-fixed-set-flag}
    \end{equation}
    Then, the first Galois cohomology $H^{1}(\R, B_{\C})$ of $B_{\C}$ is isomorphic to the set of equivalence classes on $B_{\C}^{-\gamma}$ with respect to the $\gamma$-conjugation:
    \begin{equation}
        H^{1}(\R, B_{\C}) \simeq B_{\C}^{-\gamma}/ \sim_{\gamma},
    \label{eq-isom-H1=B-gamma-over-gamma-conj}
    \end{equation}
    where we write $\sim_{\gamma}$ for the equivalence relation defined by the $\gamma$-conjugation:
    \begin{equation*}
        x \sim_{\gamma} y \iff x = \gamma(g)^{-1}yg \quad \text{ for some }g\in B_{\C}.
    \end{equation*}
\end{lemma}

\begin{proof}
    This follows directly from the definition of the Galois cohomology \cite[Chap.~II, \S 1.1 and Chap.~I, \S 5.1]{Serre-Galois-Cohomology}, in which the Galois cohomology is considered over arbitrary fields. For our case $\C/\R$, the evaluation of the nontrivial element of $\Gal(\C/\R)$ gives the desired isomorphism \eqref{eq-isom-H1=B-gamma-over-gamma-conj}. See the proof of \cite[Lem.~A.4]{Nishiyama.Tauchi.2024} for details.
\end{proof}

\begin{lemma}\label{Lem-Galois-isom-real-points-of-orbit-for-hom-sp-case-flag}
    Let $B$ be a real algebraic group, and $C$ its real algebraic subgroup.
    Write $B_{\C}, C_{\C}$ for their complexifications (and thus, the set of real points $B_{\C}(\R)$ is equal to $B$) 
    and $\gamma\colon B_{\C} \to B_{\C}$ for the complex conjugation of $B_{\C}$ with respect to $B$ 
    (which also defines the complex conjugation of $C_{\C}$ with respect to $C$).
    Suppose that the first Galois cohomology $H^{1}(\R, B_{\C})$ is trivial:
    \begin{equation*}
        H^{1}(\R, B_{\C})=1.
    \end{equation*}
    Then, the set of $B$-orbits on the real points $(B_{\C}/C_{\C})(\R)$ of $B_{\C}/C_{\C}$ is isomorphic to the first Galois cohomology $H^{1}(\R, C_{\C})$:
    \begin{equation*}
    \begin{array}{ccc}
         B \backslash (B_{\C}/C_{\C})(\R) & \xrightarrow{\;\;\sim\;\;} & H^{1}(\R, C_{\C}) \\
         \rotatebox{90}{$\in$} & & \rotatebox{90}{$\in$} \\
         B g C_{\C} & \mapsto &  [g^{-1} \gamma(g)]_{\gamma},
    \end{array}
    \end{equation*}
    where we identify $H^{1}(\R, C_{\C})$ with $C_{\C}^{-\gamma}/ \sim_{\gamma}$ via Lemma \ref{Lem-Galois-coh=gamma-conjugation-flag}, 
    and $[g^{-1}\gamma(g)]_{\gamma}$ denotes the equivalence class of $g^{-1}\gamma(g) \in C_{\C}^{-\gamma}$ under the $\gamma$-conjugation.
\end{lemma}

\begin{proof}
    This follows from the long exact sequence of the Galois cohomology \cite[Chap.~I, \S 5.4, Prop.~36]{Serre-Galois-Cohomology}:
    \begin{equation*}
        1 \to C \to B \to (B_{\C}/C_{\C})(\R) \to H^{1}(\R,C_{\C}) \to H^{1}(\R,B_{\C}).
    \end{equation*}
    Note that $C_{\C}(\R) = C$ and $B_{\C}(\R) = B$ by definition.
    See also \cite[Fact.~A.5]{Nishiyama.Tauchi.2024}.
\end{proof}

%%We note that the following lemma seems to be also well known.

\begin{lemma}\label{lemma-real-orbits-via-Galois-cohomology-non-homogeneous-case-flag}
Let $B$ be a real algebraic group, and $X$ a real algebraic $B$-variety.
Write $B_{\C}$ and $X_{\C}$ for complexifications of $B$ and $X$ respectively
(hence the set of $\R$-rational points $X_{\C}(\R)$ of $X_{\C}$ coincides with $X$),  
and write $\gamma \colon B_{\C} \to B_{\C}$ for the complex conjugation of $B_{\C}$ with respect to $B$. 
Moreover, let $\Xi\subset B_{\C}\backslash X_{\C} $ be a subset of the orbits 
$ 
	B_{\C}\backslash X_{\C}
$
with nontrivial $\R$-rational points:
\begin{equation*}
	\Xi\coloneqq
	\{
        \mathcal{O} \in B_{\C} \backslash X_{\C}
    \mid
	    \mathcal{O}(\R)\neq \emptyset
    \}.
\end{equation*}
Suppose that the first Galois cohomology $H^{1}(\R, B_{\C})$ is trivial:
\begin{equation*}
    H^{1}(\R, B_{\C}) = 1.
\end{equation*}
Then, there is a bijection
\begin{equation}
\begin{array}{ccc}
	B \backslash X  & \xrightarrow{\;\;\sim\;\;} & 
    \coprod_{\mathcal{O}\in\Xi} H^{1}(\R, (B_{\C})_{x_{\mathcal{O}}})
    \\
    \rotatebox{90}{$\in$}  & & \rotatebox{90}{$\in$}
    \\
    Bgx_{\mathcal{O}} & \mapsto & [g^{-1}\gamma(g)]_{\gamma},
\end{array}
\label{eq-BX=H1(RBx)-flag}
\end{equation}
where $(B_{\C})_{x_{\mathcal{O}}}$ is the stabilizer at $x_{\mathcal{O}} \in \mathcal{O}(\R)$ in $B_{\C}$, and 
we identify $H^{1}(\R, (B_{\C})_{x_{\mathcal{O}}})$ with $(B_{\C})_{x_{\mathcal{O}}}^{-\gamma}/\sim_{\gamma}$ via Lemma \ref{Lem-Galois-coh=gamma-conjugation-flag}, 
and $[g^{-1}\gamma(g)]_{\gamma} \in H^{1}(\R, (B_{\C})_{x_{\mathcal{O}}})$ denotes the equivalence class of $g^{-1}\gamma(g) \in (B_{\C})_{x_{\mathcal{O}}}^{-\gamma}$ under the $\gamma$-conjugation.
\end{lemma}

\begin{proof}
This follows from Lemma \ref{Lem-Galois-isom-real-points-of-orbit-for-hom-sp-case-flag} by applying it to each $B_{\C}$-orbit $\mathcal{O} \in \Xi$ on $X_{\C}$.
See also \cite[Corollary. A.6]{Nishiyama.Tauchi.2024}.
\end{proof}

In this article, we use this lemma in the case where the algebraic group $B_{\C}$ is a Borel subgroup, in particular, it is a solvable group.
In such situation, 
the determination of the first Galois cohomology reduces to that of a maximal torus.  
This seems to be well known, but since we cannot find good references, we prove it in more general setting below.

From this point to the end of the following lemma, 
we denote the base field by $ k $, whose   
characteristic we assume to be zero 
($ k $ corresponds to $ \R $ in our situation above), and 
$ B $ denotes an affine algebraic group over $ k $.  
In the following lemma, $ T $ might not be a torus, but it denotes simply a Levi subgroup.

% \begin{lemma}\label{Lem-H^1(Nc,R)=1-flag}
% Let $k$ be a field, $K$ its separable closure,
% and
% $N_{K}$ an algebraic group defined over $k$, which is isomorphic to an iterated extension of a one-dimensional additive group $K$ (as an algebraic group defined over $k$).
% Then, the Galois cohomology of $N_{K}$ is trivial:
%     \begin{equation*}
%         H^{1}(k, N_{K})=1.
%     \end{equation*}
% \end{lemma}

% \begin{proof}
% By the long exact sequence of Galois cohomology \cite[Chap.~I, \S 5.4, Prop.~36]{Serre-Galois-Cohomology}, 
% it is sufficient to prove the lemma in the case $N=K$,
% in which case, the lemma follows from 
% \cite[Chap.~X, \S~1, Prop.~1]{Serre.1962}.
% % \cite[Chap.~II, \S 1.2, Prop.~1]{Serre-Galois-Cohomology}.
% \end{proof}

\begin{lemma}\label{Lem-for-solvableB-H^1(B,R)=H^1(T,R)-flag}
Let $k$ be a perfect field, $K$ its algebraic closure.  
    Moreover, let $B_{K}$ be an algebraic group defined over $k$, and assume we have  a Levi decomposition $B_{K}=T_{K}N_{K}$ defined over $k$.
    Then, for the Galois cohomology groups, we have
    \begin{equation*}
        H^{1}(k, B_{K}) \simeq H^{1}(k, T_{K}).
    \end{equation*}
\end{lemma}

\begin{proof}
First, we will prove
   \begin{equation}\label{eq-H^1(Nc,R)=1-in-the-proof-flag}
       H^{1}(k, N_{K})=1.
   \end{equation}
By \cite[Cor.~15.5 (ii)]{Borel.1991} (because we assume $k$ is perfect), $N_{K}$ is an algebraic group defined over $k$, which is isomorphic to an iterated extension of a one-dimensional additive group $K$.
Therefore, by the long exact sequence of the Galois cohomology \cite[Chap.~I, \S 5.4, Prop.~36]{Serre-Galois-Cohomology}, 
it is sufficient to prove \eqref{eq-H^1(Nc,R)=1-in-the-proof-flag} in the case $N=K$,
in which case, \eqref{eq-H^1(Nc,R)=1-in-the-proof-flag} follows from \cite[Chap.~X, \S~1, Prop.~1]{Serre.1962}.
This completes the proof of \eqref{eq-H^1(Nc,R)=1-in-the-proof-flag}.
% \cite[Chap.~II, \S 1.2, Prop.~1]{Serre-Galois-Cohomology}.
% Thus, Lemma \ref{Lem-H^1(Nc,R)=1-flag} implies 

   Next, applying  \cite[Chap.~I, \S 5.5, Prop.~38]{Serre-Galois-Cohomology} to $(A,B,C)=(N_{K},B_{K},T_{K})$ (under the notation of [loc.~cit.]), we have an exact sequence (of pointed sets) 
   \begin{equation}
       H^{1}(k, N_{K}) \to H^{1}(k, B_{K}) \to H^{1}(k, T_{K}).
   \label{eq-exact-seq-NBT}
   \end{equation}
   Thus, we have an injection
   \begin{equation}
       H^{1}(k, B_{K}) \hookrightarrow H^{1}(k, T_{K})
    \label{H^1(B,R)-hookrightarrow-H^1(T,R)-flag}
   \end{equation}
   by \eqref{eq-H^1(Nc,R)=1-in-the-proof-flag}. (Note that, by definition, the exact sequence \eqref{eq-exact-seq-NBT} of pointed sets only guarantees that the fiber of the neutral element of the third term comes from the first one. However, in the Galois cohomology context, there exist isomorphisms between the fibers of the second map of \eqref{eq-exact-seq-NBT}, see \cite[Chap.~I, \S 5.4, 2nd paragraph]{Serre-Galois-Cohomology}. Thus \eqref{eq-H^1(Nc,R)=1-in-the-proof-flag} gives the injection \eqref{H^1(B,R)-hookrightarrow-H^1(T,R)-flag}.)
   Moreover, the inclusion $T_{K} \hookrightarrow B_{K}$ induces (by \cite[Chap.~I, \S 5.4]{Serre-Galois-Cohomology}) a map
   \begin{equation}
       H^{1}(k, T_{K}) \to H^{1}(k, B_{K}).
    \label{H^1(T,R)-to-H^1(B,R)-flag}
   \end{equation}
   The composition \eqref{H^1(T,R)-to-H^1(B,R)-flag} and \eqref{H^1(B,R)-hookrightarrow-H^1(T,R)-flag} is induced by the composition $T_{K} \hookrightarrow B_{K} \twoheadrightarrow B_{K}/N_{K} \simeq T_{K}$, which is equal to the identity on $T_{K}$. 
   Thus, the composition \eqref{H^1(T,R)-to-H^1(B,R)-flag} and \eqref{H^1(B,R)-hookrightarrow-H^1(T,R)-flag} is also the identity, 
   which implies that the map \eqref{H^1(B,R)-hookrightarrow-H^1(T,R)-flag} is surjective.
   This completes the proof.
\end{proof}

Now let us return back to our basic settings.  
By the general lemma above, the determination of the first Galois cohomology $ H^{1}(\R, (B_{\C})_{x_{\mathcal{O}}}) $ in the right-hand side of \eqref{eq-BX=H1(RBx)-flag} is reduced to that of the maximal torus of $(B_{\C})_{x_{\mathcal{O}}}$.

\begin{corollary}\label{Cor-Galois-cohomology-of-stabilizer-subgroup-of-B-flag}
Let $B$ be a real algebraic solvable group, $X$ a real algebraic $B$-variety, and $x \in X$.
Write $B_{\C} =T_{\C}N_{\C}$ for the Levi decomposition of a complexification $B_{\C}$ of $B$.
Suppose that the Levi decomposition of the stabilizer subgroup $(B_{\C})_{x}$ at $x\in X$ in $B_{\C}$ is given by $(B_{\C})_{x}=(T_{\C})_{x}(N_{\C})_{x}$.
Then, we have
\begin{equation*}
    H^{1}(\R, (B_{\C})_{x}) \simeq H^{1}(\R, (T_{\C})_{x}).
\end{equation*}
\end{corollary} 

\begin{proof}
This corollary is an immediate consequence of Lemma \ref{Lem-for-solvableB-H^1(B,R)=H^1(T,R)-flag}.
\end{proof}

\subsection{Determination of the Galois cohomology}

As we mentioned in the first paragraph of Section \ref{Section-facts-Galois-cohomology-flag}, we have to consider  three issues (i)--(iii) written there in order to determine the orbit decomposition $B_{H} \backslash G /P_{G}$. 
We already know the answers of first two issues (i) and (ii) by Theorem \ref{Thm-complex-orbit-decomp-AIII-and-CI} and Theorem \ref{theorem-Galois-RDFV} \eqref{theorem-Galois-RDFV:item:1:nonempty.R.pts} 
(which is already proved in Section \ref{Section-proof-of-(1)-of-Thm-RDFV}), respectively.
In this subsection, we deal with the last part (iii), namely, we determine the first Galois cohomology $H^{1}(\R, (B_{n}^{+} \times B_{n}^{-})_{[\tau]})$ for $\tau \in \regR_{n}$.
We reduce this to that of the maximal torus of $(B_{n}^{+} \times B_{n}^{-})_{[\tau]}$.  
To do so, let us discuss the structure of 
the stabilizer $(B_{n}^{+} \times B_{n}^{-})_{[\tau]}$.  

\begin{definition}\label{Def-PI}
For $m\in\Z_{\geq 0}$, we write $\PI_{m}$ for the set of partial involutions:
\begin{equation*}
\begin{split}
    \PI_{m} & \coloneqq \{ \tau \in \PP_{m} \mid \tau^{2} \text{ is a diagonal matrix whose entries belong to the set }\{0,1\}\}\\
    &= \PP_{m} \cap \Sym_{m}(\Z),
\end{split}
\end{equation*}
where $\PP_{m}$ is the set of partial permutations defined in Definition \ref{Def-PP-regT-regR}.
\end{definition}

\begin{lemma}\label{Lem-Levi-decomp-of-(B+-B-)-tau}
 Let $\tau \in \regR_{n}$ (Definition \ref{Def-PP-regT-regR}). Then, the Levi decomposition of the stabilizer subgroup $(B_{n}^{+}\times B_{n}^{-})_{[\tau]}$ at $[\tau] \in \regMat_{2n,n} / \GL_{n}(\C) $ is given by
 \begin{equation}
     (B_{n}^{+}\times B_{n}^{-})_{[\tau]} = (T_{n}^{+}\times T_{n}^{-})_{[\tau]} (N_{n}^{+}\times N_{n}^{-})_{[\tau]},
 \label{eq-BB=TT-NN-case-regR}
 \end{equation}
 where $N_{n}^{+}$ and $N_{n}^{-}$ are the subgroups of upper and lower matrices of $B_{n}^{+}$ and $B_{n}^{-}$, respectively. 
\end{lemma}

\begin{proof}
It is clear that the right-hand side of \eqref{eq-BB=TT-NN-case-regR} is contained in the left-hand side.
Thus,
it is sufficient to prove that 
\begin{equation}\label{eq-sufficient-for-BB=TTNN-regR}
   (b,c) \in (B_{n}^{+}\times B_{n}^{-})_{[\tau]} \quad\text{ implies }\quad (b_{T},c_{T}) \in (T_{n}^{+}\times T_{n}^{-})_{[\tau]},
\end{equation}
where we write $(b_{T},c_{T}) \in T_{n}^{+}\times T_{n}^{-}$ for the torus part of $(b,c) \in B_{n}^{+}\times B_{n}^{-}$.

    In order to prove \eqref{eq-sufficient-for-BB=TTNN-regR}, we first determine the necessary and sufficient condition for $(b,c) \in B_{n}^{+} \times B_{n}^{-}$ to belong the stabilizer $(B_{n}^{+}\times B_{n}^{-})_{[\tau]}$.
    Let $\tau = \begin{pmatrix}\tau_{1} \\ \tau_{2} \end{pmatrix}\in \regR_{n} \subset \regMat_{2n,n}$ and set $I \coloneqq \ech\tau_{2}$, where $\ech$ is defined in \eqref{eq-def-ech}. 
    Then, the similar argument to the proof of the surjectivity of Proposition \ref{Prop-Omega-I=SPI} implies that there exists $\eta \in \PI_{m}$ such that we have
    \begin{equation}
        h(\sigma_{I}) \tau \GL_{n}(\C) 
        = 
        h(\sigma_{I}) 
        \begin{pmatrix}
          \eta & 0 \\ 0 & \unitmatrix_{n-m} \\ \unitmatrix_{m} & 0 \\ 0 & 0
        \end{pmatrix} \GL_{n}(\C)
        \in \regMat_{2n,n} / \GL_{n}(\C),
    \label{eq-h-tau-GL=h-eta-GL}
    \end{equation}
    where $m\coloneqq \# I$ and $h(\sigma_{I}) = \diag(\sigma_{I},(\sigma_{I}^{*})^{-1})$. 
    Then, the action of $(b,c)\in B_{n}^{+} \times B_{n}^{-}$ (which is given by \eqref{eq-def-action-of-BB-on-regRst}) is described as
    \begin{equation}
        (b,c) \cdot h(\sigma_{I}) 
        \begin{pmatrix}
          \eta & 0 \\ 0 & \unitmatrix_{n-m} \\ \unitmatrix_{m} & 0 \\ 0 & 0
        \end{pmatrix} \GL_{n}(\C)
        =
        h(\sigma_{I}) 
        \begin{pmatrix}
            \sigma_{I}^{-1} b \sigma_{I} & 0  \\
            0 & \sigma_{I} c \sigma_{I}^{-1}
        \end{pmatrix}
        \begin{pmatrix}
          \eta & 0 \\ 0 & \unitmatrix_{n-m} \\ \unitmatrix_{m} & 0 \\ 0 & 0
        \end{pmatrix} \GL_{n}(\C).
    \label{eq-(b,c)-act-tau-eta}
    \end{equation}
    In order to compute the right-hand side of \eqref{eq-(b,c)-act-tau-eta},
    we write (by using Lemma \ref{Lem-B-p-cap-GL-times-GL}, or its proof)
    \begin{equation}
        \sigma_{I}^{-1} b \sigma_{I} = 
        \begin{pmatrix}
            b_{m} & a \\ d & b_{n-m}    
        \end{pmatrix},
        \qquad 
        \sigma_{I} c \sigma_{I}^{-1} = 
        \begin{pmatrix}
            c_{m} & \alpha \\ \delta & c_{n-m}    
        \end{pmatrix}
    \label{eq-def-bsigmaI-and-csigmaI}
    \end{equation}
    for some
    \begin{equation*}
        b_{m} \in B_{m}^{+},\quad b_{n-m} \in B_{n-m}^{+}, \quad
        c_{m} \in B_{m}^{-},\quad c_{n-m} \in B_{n-m}^{-}, \quad
    \end{equation*}
    and
    \begin{equation*}
        a, \alpha \in \Mat_{m,n-m}(\C), \quad
        d, \delta \in \Mat_{n-m,m}(\C).
    \end{equation*}
    Then, the right-hand side of \eqref{eq-(b,c)-act-tau-eta} is equal to
    \begin{equation*}
        h(\sigma_{I})
        \begin{pmatrix}
            b_{m} & a & 0 & 0 \\ d & b_{n-m} & 0 & 0 \\ 0 & 0 & c_{m} & \alpha \\ 0 & 0 & \delta & c_{n-m}
        \end{pmatrix}
        \begin{pmatrix}
          \eta & 0 \\ 0 & \unitmatrix_{n-m} \\ \unitmatrix_{m} & 0 \\ 0 & 0
        \end{pmatrix} \GL_{n}(\C)
        =
        h(\sigma_{I})
        \begin{pmatrix}
          b_{m}\eta & a \\ d\eta & b_{n-m} \\ c_{m} & 0 \\ \delta & 0
        \end{pmatrix} \GL_{n}(\C),
    \end{equation*}
    which is equal to
    \begin{equation*}
    h(\sigma_{I})
        \begin{pmatrix}
          b_{m}\eta & a \\ d\eta & b_{n-m} \\ c_{m} & 0 \\ \delta & 0
        \end{pmatrix} 
        \begin{pmatrix}
          c_{m}^{-1} & 0 \\ 0 & b_{n-m}^{-1} 
        \end{pmatrix}
        \GL_{n}(\C)
        =
    h(\sigma_{I})
        \begin{pmatrix}
          b_{m}\eta c_{m}^{-1} & ab_{n-m}^{-1} \\ d\eta c_{m}^{-1} & \unitmatrix_{n-m} \\ \unitmatrix_{m} & 0 \\ \delta c_{m}^{-1} & 0
        \end{pmatrix} \GL_{n}(\C).
    \end{equation*}
    Thus, $(b,c) \in (B_{n}^{+} \times B_{n}^{-})_{[\tau]}$ if and only if there exists 
    $\begin{pmatrix}
            x & y \\ z & w
        \end{pmatrix}
        \in  \GL_{n}(\C)$ such that
    \begin{equation}\label{eq-eta=abcd-xyzw}
        \begin{pmatrix}
          \eta & 0 \\ 0 & \unitmatrix_{n-m} \\ \unitmatrix_{m} & 0 \\ 0 & 0
        \end{pmatrix}
        =
        \begin{pmatrix}
          b_{m}\eta c_{m}^{-1} & ab_{n-m}^{-1} \\ d\eta c_{m}^{-1} & \unitmatrix_{n-m} \\ \unitmatrix_{m} & 0 \\ \delta c_{m}^{-1} & 0
        \end{pmatrix}
        \begin{pmatrix}
            x & y \\ z & w
        \end{pmatrix}.
    \end{equation}
    The right-hand side of \eqref{eq-eta=abcd-xyzw} is equal to
    \begin{equation}\label{eq-abcd-xyzw}
        \begin{pmatrix}
          b_{m}\eta c_{m}^{-1} x + ab_{n-m}^{-1}z & b_{m}\eta c_{m}^{-1} y + ab_{n-m}^{-1} w \\ 
          d\eta c_{m}^{-1}     x +              z & d\eta c_{m}^{-1}     y +               w \\ 
                               x                  &                      y \\ 
          \delta c_{m}^{-1}    x                  & \delta c_{m}^{-1}    y
        \end{pmatrix}.
    \end{equation}
    By comparing \eqref{eq-abcd-xyzw} and the left-hand side of \eqref{eq-eta=abcd-xyzw}, the condition $(b,c) \in (B_{n}^{+} \times B_{n}^{-})_{[\tau]}$ is satisfied if and only if there exists 
    $\begin{pmatrix}
        x & y \\ z & w 
    \end{pmatrix} \in \GL_{n}(\C)$ such that
    \begin{equation}\label{eq-sufficient-cond-stabilizer-LM}
        x = \unitmatrix_{m},     \quad
        y = 0,                   \quad
        z = -d\eta c_{m}^{-1},   \quad
        w = \unitmatrix_{n-m},   \quad
        a = 0,                   \quad
        \delta = 0,              \quad
        b_{m}\eta c_{m}^{-1} = \eta.
    \end{equation}
    Note that we always have $\begin{pmatrix}
        \unitmatrix_{m} & 0 \\ -d\eta c_{m}^{-1} & \unitmatrix_{n-m} 
    \end{pmatrix} \in \GL_{n}(\C)$.
    Therefore, $(b,c) \in (B_{n}^{+} \times B_{n}^{-})_{[\tau]}$ if and only if
    \begin{equation}\label{eq-bc-contained-in-stabilizer-in-case-regR}
        a = 0, \quad \delta = 0, \quad b_{m}\eta c_{m}^{-1} = \eta.
    \end{equation}
    The first two conditions of \eqref{eq-bc-contained-in-stabilizer-in-case-regR} are not related to the torus part $(b_{T},c_{T}) \in T_{n}^{+} \times T_{n}^{-}$ of $(b,c) \in B_{n}^{+} \times B_{n}^{-}$ by  \eqref{eq-def-bsigmaI-and-csigmaI} because the adjoint by $\sigma_{I}^{-1}$ preserves both of the diagonal torus and the root space decomposition (with respect to this diagonal torus).
    This implies that 
    \begin{equation}
        (b_{T},c_{T}) \in (T_{n}^{+} \times T_{n}^{-})_{[\tau]}\quad \iff \quad b_{m}\eta c_{m}^{-1} = \eta.
    \label{eq-bc-contained-in-stabilizer-in-case-regR-2}
    \end{equation}
    Thus, by \eqref{eq-def-bsigmaI-and-csigmaI} and \eqref{eq-bc-contained-in-stabilizer-in-case-regR-2}, in order to prove \eqref{eq-sufficient-for-BB=TTNN-regR}, it is sufficient to show that
    \begin{equation}
        (b,c) \in (B_{m}^{+}\times B_{m}^{-})_{\eta} \quad\text{ implies }\quad (b_{T},c_{T}) \in (T_{m}^{+}\times T_{m}^{-})_{\eta},
    \label{eq-BB=TT-NN-case-Mat}
    \end{equation}
where $(b_{T},c_{T}) \in T_{m}^{+} \times T_{m}^{-}$ is the torus part of $(b,c) \in B_{m}^{+} \times B_{m}^{-}$.
We will prove \eqref{eq-BB=TT-NN-case-Mat} by the induction on $m$.
% From now on $(b,c)$ stands for an element of $(B_{m}^{+}\times B_{m}^{-})_{[\tau]}$.

First, consider the case $m=1$. In this case, we have $B_{m}^{+} \times B_{m}^{-}= T_{m}^{+} \times T_{m}^{-}$, and thus \eqref{eq-BB=TT-NN-case-Mat} is clear.

Next, assume that \eqref{eq-BB=TT-NN-case-Mat} holds in the case $m \leq l$, and assume $m=l+1$. 
In this case, we divide the cases according to whether $\eta \in \PI_{m}$ is invertible or not.

First, we assume that $\eta \in \PI_{m}$ is invertible. Thus, for $(b,c)\in B_{m}^{+} \times B_{m}^{-}$, we have
\begin{equation}
    (b,c) \in (B_{m}^{+}\times B_{m}^{-})_{\eta} \iff
    b\eta c^{-1} = \eta \iff \eta^{-1}b\eta = c.
    \label{eq-betac-1-eta-1betac}
\end{equation}
Because the diagonal and off-diagonal entries are preserved by the adjoint of $\eta$, 
\eqref{eq-betac-1-eta-1betac} implies that $(b,c) \in (B_{m}^{+}\times B_{m}^{-})_{\eta}$ if and only if
\begin{equation*}
    \eta^{-1}b_{T}\eta = c_{T} \quad\text{ and }\quad \eta^{-1}b_{N}\eta = c_{N},
\end{equation*}
where we write $(b,c) \eqqcolon (b_{T},c_{T})(b_{N},c_{N}) \in (T_{m}^{+}\times T_{m}^{-}) (N_{m}^{+}\times N_{m}^{-}) = B_{m}^{+}\times B_{m}^{-}$.
Thus, in this case, \eqref{eq-BB=TT-NN-case-Mat} holds.

Next, we assume that $\eta \in \PI_{m}$ is not invertible. In this case, there exists $1 \leq k \leq m$ such that both of the $k$-th column and $k$-th row are zero by the definition of $\PI_{m}$ (Definition \ref{Def-PI}). Therefore, we can write $\eta$ as a block matrix corresponding to the partition $m=(k-1)+1+(m-k)$ in the following way.
\begin{equation*}
    \eta =
    \begin{pmatrix}
        \eta_{11} & 0 & \eta_{13} \\
        0 & 0 & 0 \\
        \eta_{31} & 0 & \eta_{33}
    \end{pmatrix}.
\end{equation*}
We also write
\begin{equation}
    b = 
    \begin{pmatrix}
        b_{11} & b_{12} & b_{13} \\
        0      & b_{22} & b_{23} \\
        0      & 0      & b_{33}
    \end{pmatrix} \in B_{m}^{+},
    \qquad
    c =
    \begin{pmatrix}
        c_{11} & 0      & 0\\
        c_{21} & c_{22} & 0\\
        c_{31} & c_{32} & c_{33}
    \end{pmatrix} \in B_{m}^{-}.
\label{eq-def-bc-3-times-3-BB}
\end{equation}
Then, we have
\begin{equation}
    b\eta  =
    \begin{pmatrix}
        b_{11}\eta_{11} + b_{13}\eta_{31} & 0 & b_{11}\eta_{13} + b_{13}\eta_{33} \\
        b_{23}\eta_{31}                   & 0 & b_{23}\eta_{33} \\
        b_{33}\eta_{31}                   & 0 & b_{33}\eta_{33}
    \end{pmatrix}
\label{eq-b-eta}
\end{equation}
and
\begin{equation}
    \eta c =
    \begin{pmatrix}
        \eta_{11}c_{11} + \eta_{13}c_{31} & \eta_{13}c_{32} & \eta_{13}c_{33}\\
        0 & 0 & 0\\
        \eta_{31}c_{11} + \eta_{33}c_{31} & \eta_{33}c_{32} & \eta_{33}c_{33}
    \end{pmatrix}.
\label{eq-eta-c}
\end{equation}
We note that $b\eta c^{-1} = \eta$ if and only if $b\eta = \eta c$.
Thus, \eqref{eq-b-eta} and \eqref{eq-eta-c} imply that $(b,c) \in (B_{m}^{+} \times B_{m}^{-1})_{\eta}$ if and only if
\begin{equation}
    b_{23}\eta_{31} = b_{23}\eta_{33} = \eta_{13}c_{32} = \eta_{33}c_{32} = 0, \quad \text{ and } \quad
    b'\eta'(c')^{-1} = \eta',
\label{eq-bc-contained-in-stabilizer-in-case-Mat}
\end{equation}
where we write
\begin{equation*}
    \eta' \coloneqq
    \begin{pmatrix}
        \eta_{11} & \eta_{13} \\
        \eta_{31} & \eta_{33}
    \end{pmatrix} \in \PI_{m-1}, \quad
    b' \coloneqq
    \begin{pmatrix}
        b_{11} & b_{13} \\
        0      & b_{33}
    \end{pmatrix} \in B_{m-1}^{+}, \quad
    c' =
    \begin{pmatrix}
        c_{11} & 0 \\
        c_{13} & c_{33}
    \end{pmatrix} \in B_{m-1}^{-}.
\end{equation*}
Note that the first equation of \eqref{eq-bc-contained-in-stabilizer-in-case-Mat} is not related to the torus part $(b_{T},c_{T}) \in T_{m}^{+} \times T_{m}^{-}$ of $(b,c) \in B_{m}^{+} \times B_{m}^{-}$ by \eqref{eq-def-bc-3-times-3-BB}.
This implies that $(b_{T},c_{T}) \in (T_{m}^{+} \times T_{m}^{-})_{\eta}$ if and only if
    \begin{equation*}
        b'\eta'(c')^{-1} = \eta'.
    \end{equation*}
Because $\eta'\in \PI_{m-1}$, the induction hypothesis completes the proof of \eqref{eq-BB=TT-NN-case-Mat} in this case, and thus, completes the proof of the lemma.
\end{proof}

Now we can determine the first Galois cohomology of the stabilizer.

\begin{lemma}\label{Lem-H^1(BB-tau)=H^1(TT-tau)-flag}
Let $\tau \in \regR_{n}$ (Definition \ref{Def-PP-regT-regR}). Then, we have
\begin{equation*}
    H^{1}(\R,  (B_{n}^{+}\times B_{n}^{-})_{[\tau]}) \simeq H^{1}(\R,  (T_{n}^{+}\times T_{n}^{-})_{[\tau]}),
\end{equation*}
where $T_{n}^{+}$ and $T_{n}^{-}$ are the subgroups of diagonal matrices of $B_{n}^{+}$ and $B_{n}^{-}$, respectively. 
\end{lemma}

\begin{proof}
This follows from immediately Corollary \ref{Cor-Galois-cohomology-of-stabilizer-subgroup-of-B-flag} and Lemma \ref{Lem-Levi-decomp-of-(B+-B-)-tau}.   
\end{proof}

We begin with the Galois cohomology of $(T_{n}^{+}\times T_{n}^{-})_{[\tau]}$ with $ n = 1 $ or $ 2 $.

\begin{lemma}\label{Lem-Ttau=1-flag}\label{Lem-H^1-tau-3cases-flag}
Consider the action $T_{n}^{+} \times T_{n}^{-} \leftaction \regMat_{2n,n} / \GL_{n}(\C)$ induced by the action $B_{n}^{+} \times B_{n}^{-} \leftaction \regMat_{2n,n} / \GL_{n}(\C)$ defined in \eqref{eq-def-action-of-BB-on-regRst}. 
Then, we have:
\begin{enumerate}
    \item[\upshape{($K_{c}$)}] For $ n = 1 $ and $\tau = \begin{pmatrix}
        0 \\ 1
    \end{pmatrix} \in \regR_{1}$, we have
    \begin{equation*}
        (T_{1}^{+} \times T_{1}^{-})_{[\tau]} = T_{1}^{+} \times T_{1}^{-}, \quad\text{ and }\quad
        H^{1}(\R, (T_{1}^{+} \times T_{1}^{-})_{[\tau]}) = 1.
    \end{equation*}

    \item[\upshape{($K_{d}$)}] For $ n = 1 $ and $\tau = \begin{pmatrix}
        1 \\ 0
    \end{pmatrix} \in \regR_{1}$, we have
    \begin{equation*}
        (T_{1}^{+} \times T_{1}^{-})_{[\tau]} = T_{1}^{+} \times T_{1}^{-}, \quad\text{ and }\quad
        H^{1}(\R, (T_{1}^{+} \times T_{1}^{-})_{[\tau]}) = 1.
    \end{equation*}

    \item[\upshape{($K_{1}$)}] For $ n = 1 $ and $\tau = \begin{pmatrix}
        1 \\ 1 
    \end{pmatrix} \in \regR_{1}$, we have
    \begin{equation*}
        (T_{1}^{+} \times T_{1}^{-})_{[\tau]} = \Delta(T_{1}),  \:\text{ and }\:
        H^{1}(\R, (T_{1}^{+} \times T_{1}^{-})_{[\tau]}) = \{[(1,1)]_{\gamma},[(-1,-1)]_{\gamma}\},
    \end{equation*}
    where $\Delta\colon T_{1} \hookrightarrow T_{1}^{+} \times T_{1}^{-}$ is the diagonal embedding.

    \item[\upshape{($K_{2}$)}] For $ n = 2 $ and $\tau \in \regR_{2}$ having the form
    \begin{equation*}
        \tau =
        \begin{pmatrix}
            0 & 1 \\ 
            1 & 0 \\
            1 & 0 \\ 
            0 & 1
        \end{pmatrix} \in \regR_{2},
    \end{equation*}
    we have
    \begin{equation*}
        (T_{2}^{+} \times T_{2}^{-})_{[\tau]} = \Delta'(T_{2}),\quad\text{ and }\quad
        H^{1}(\R, (T_{1}^{+} \times T_{1}^{-})_{[\tau]}) = 1,
        % \qquad T(\tau) = \{(\unitmatrix_{2},\unitmatrix_{2})\},
    \end{equation*}
    where we write $\Delta' \colon T_{2} \to T_{2}^{+} \times T_{2}^{-}$ for an embedding defined by
    \begin{equation*}
        \Delta'(\diag(t_{1},t_{2})) = 
        \left(\diag(t_{1},t_{2}), \diag(t_{2},t_{1})\right).
    \end{equation*}
\end{enumerate}
\end{lemma}

\begin{proof}
%     This easily follows from straightforward calculation.
% \end{proof}

% \begin{lemma}\label{Lem-H^1-tau-3cases-flag}
% We have
%     \begin{equation*}
%         H^{1}(\R, T_{1}^{+} \times T_{1}^{-}) = 1, \quad 
%         H^{1}(\R, \Delta(T_{1})) = \{[(1,1)]_{\gamma},[(-1,-1)]_{\gamma}\}, \quad
%         H^{1}(\R, \Delta'(T_{2})) = 1.
%     \end{equation*}
% \end{lemma}

% \begin{proof}
We only prove Case ($K_{1}$).  The other cases are similar (and easier). 

    First, we prove that $(T_{1}^{+} \times T_{1}^{-})_{[\tau]} = \Delta(T_{1}) = \{(t,t) \mid t \in T_{1}\}$. Because the action is given by \eqref{eq-def-action-of-BB-on-regRst}, this follows from
    \begin{equation*}
        (t,s) \tau \GL_{1}(\C) =  
        (t,s)
        \begin{pmatrix}
            1 \\ 1 
        \end{pmatrix}
        \GL_{1}(\C)
        =
        \begin{pmatrix}
            t \\ s 
        \end{pmatrix}
        \GL_{1}(\C)
        =\begin{pmatrix}
            ts^{-1} \\ 1 
        \end{pmatrix}\GL_{1}(\C).
    \end{equation*}
    
    Next, we claim that 
    \begin{equation*}
        \Delta(T_{1})^{-\gamma}:= \{(t,t)\in \Delta(T_{1})\mid \gamma(t,t)^{-1} = (t,t) \}   
        =
        \{(t,t)\mid t\in \GL_{1}(\R) \}.
    \end{equation*}
    By the definition of $\gamma$ (given in \eqref{eq-def-gamma-caseA}), this follows from
    \begin{equation*}
        \gamma(t,t)^{-1}=(\overline{t},\overline{t})\qquad \text{ for } t\in \GL_{1}(\C).
    \end{equation*}
    
    By Lemma \ref{Lem-Galois-coh=gamma-conjugation-flag}, we have 
    \begin{equation*}
      H^{1}(\R, \Delta(T_{1})) \simeq \Delta(T_{1})^{-\gamma} / \sim_{\gamma},
    \end{equation*}
    which is equal to $\{[(1,1)]_{\gamma},[(-1,-1)]_{\gamma}\}$, since
    \begin{equation*}
        \gamma(t,t)^{-1}(t,t)=(|t|^{2},|t|^{2})\qquad \text{ for } t\in \GL_{1}(\R).
    \end{equation*}
\end{proof}

%%Thus, we finish the proof of Lemma \ref{Lem-H^1(BB-tau)=H^1(TT-tau)-flag}. 
% Next, we want to know the right-hand side of Lemma \ref{Lem-H^1(BB-tau)=H^1(TT-tau)-flag}.

By Lemma \ref{Lem-Ttau=1-flag}, it is expected that $H^{1}(\R,  (T_{n}^{+}\times T_{n}^{-})_{[\tau]}) \simeq \{ \pm 1\}^{k_{1}}$, where $k_{1} \coloneqq \# K_{1}$ and $K_{1}$ is the subset of $[n]$ given in Notation \ref{Notation-K-for-tau-in-regR} for $\tau \in \regR_{n}$, or more precisely, it is expected that we have
\begin{equation*}
    H^{1}(\R,  (T_{n}^{+}\times T_{n}^{-})_{[\tau]}) \simeq
    \{[(t,t)]_{\gamma} \in T_{n}^{+} \times T_{n}^{-} \mid t_{i} \in \{-1, 1\} \text{ if } i \in K_{1}, \text{ and }
            t_{i} = 1           \text{ if } i \notin K_{1} \},
\end{equation*}
where we write $t= \diag(t_{1},\dots,t_{n}) \in T_{n}$.
In fact, this expectation is correct.
From now on, we make some preparation in order to prove this for the higher rank case.

\begin{definition}\label{Def-T-tau-flag}
    Let $T_{n}^{+}$ and $T_{n}^{-}$ be the subgroups of diagonal tori of $B_{n}^{+}$ and $B_{n}^{-}$, respectively.
    Then, for $\tau \in \regR_{n}$, we define a finite subset $T(\tau) $ of $T_{n}^{+} \times T_{n}^{-}$ by 
%%    (see also Example \ref{Ex-T-tau-flag})
    \begin{equation*}
        T(\tau) \coloneqq 
        \{ (t,t) \in T_{n}^{+} \times T_{n}^{-} 
        \mid 
            t_{i} \in \{-1, 1\} \text{ if } i \in K_{1}, \text{ and }
            t_{i} = 1           \text{ if } i \notin K_{1}
        \},
    \end{equation*}
    where  $K_{1}$ is defined in Notation \ref{Notation-K-for-tau-in-regR} for $\tau \in \regR_{n} $ and we write $t= \diag(t_{1},\dots,t_{n}) \in T_{n}^{+}$.
    Note that we have $T(\tau \sigma) = T(\tau)$ for any $\sigma \in \Sgroup_{n}$, hence 
    $T([\tau]) \; ([\tau] \in \regR_{n} /\Sgroup_{n}) $ is well defined.
\end{definition}

\begin{remark}\label{Rem-Ttau=dtau}
    By definition, we have $T(\tau) \simeq \{\pm 1\}^{d(\tau)}$ for $\tau\in \regR_{n}$ because $\# K_{1} = d(\tau)$, where $d(\tau)$ is defined in Definition \ref{Def-d-tau}.
\end{remark}

\begin{example}\label{Ex-T-tau-flag}
Consider the same $\tau \in \regR_{n}$ as in Example \ref{Ex-K-for-tau-in-regR}. 
Then, we have
\begin{equation*}
    T(\tau) =\left\{
    \left(\unitmatrix_{5},\unitmatrix_{5}\right),
    \left(\diag(-1,\unitmatrix_{4}),\diag(-1,\unitmatrix_{4})\right)
    \right\},
\end{equation*}
where $\unitmatrix_{5} \in \GL_{5}(\C)$ is the identity matrix.
\end{example}

\begin{lemma}\label{Lem-map-from-Ttau-to-Galois-flag}\label{Lem-Ttau=H^1(TT,R)-flag}
    Let $\tau \in \regR_{n}$ 
%%(Definition \ref{Def-PP-regT-regR}) 
    and $T(\tau)$ be the subset of $T_{n}^{+}\times T_{n}^{-}$ defined in Definition \ref{Def-T-tau-flag}.
    Then, there exists an isomorphism
    \begin{equation}\label{eq-Lem-map-T-to-H1}
    \begin{array}{ccc}
         T(\tau) & \xrightarrow{\sim} & H^{1}(\R, (T_{n}^{+} \times T_{n}^{-})_{[\tau]}) \\
         \rotatebox{90}{$\in$} & & \rotatebox{90}{$\in$} \\
         (t,t) & \mapsto & [(t,t)]_{\gamma},
    \end{array}
    \end{equation}
    where we identify $H^{1}(\R, (T_{n}^{+} \times T_{n}^{-})_{[\tau]})$ with $(T_{n}^{+} \times T_{n}^{-})_{[\tau]}^{-\gamma} / \sim_{\gamma}$, 
    % via the isomorphism of Lemma \ref{Lem-Galois-coh=gamma-conjugation-flag},
    and $[(t,t)]_{\gamma}$ denotes the equivalence class of $(t,t)$ in $(T_{n}^{+} \times T_{n}^{-})_{[\tau]}^{-\gamma} / \sim_{\gamma}$.
% The map given in Lemma \ref{Lem-map-from-Ttau-to-Galois-flag} is, in fact, the isomorphism
%     \begin{equation*}
%     \begin{array}{ccc}
%          T(\tau) & \xrightarrow{\sim} & H^{1}(\R, (T_{n}^{+} \times T_{n}^{-})_{[\tau]}) \\
%          \rotatebox{90}{$\in$} & & \rotatebox{90}{$\in$} \\
%          (t,t) & \mapsto & [(t,t)]_{\gamma}.
%     \end{array}
%     \end{equation*}
\end{lemma}

\begin{proof}
    First, we note that the map \eqref{eq-Lem-map-T-to-H1}  is well defined because there exists an inclusion
    \begin{equation}
        T(\tau) \subset (T_{n}^{+} \times T_{n}^{-})_{[\tau]}^{-\gamma},
    \label{eq-Ttau-inclusion-TT_tau^-gamma-flag}
    \end{equation}
    % where the complex conjugation $\gamma$ is defined in Section \ref{Section-Preliminary},
    % and the superscript ``$-\gamma$'' means the set of $(-\gamma)$-fixed points (defined in \eqref{def-eq--gamma-fixed-set-flag}).
    which follows from the definitions of $\gamma$ (given in \eqref{eq-def-gamma-caseA}) 
    and the action $T_{n}^{+} \times T_{n}^{-} \leftaction \regMat_{2n,n} / \GL_{n}(\C)$ 
    (induced by $B_{n}^{+} \times B_{n}^{-} \leftaction \regMat_{2n,n} / \GL_{n}(\C)$ given in \eqref{eq-def-action-of-BB-on-regRst}).

Thus, from now on, we prove that the map \eqref{eq-Lem-map-T-to-H1} is in fact an isomorphism.
    Let $\tau \in \regR_{n}$ and $K_{c} \coprod K_{d} \coprod K_{1} \coprod K_{2} = [n]$ be the decomposition given in Notation \ref{Notation-K-for-tau-in-regR}.
    Set $k_{j} \coloneqq \# K_{j}$ for $j \in \{c,d,1,2\}$.
By Lemma \ref{Lem-Ttau=1-flag},
    % Because the action (given in \eqref{eq-def-action-of-BB-on-regRst}) of $T_{n}^{+} \times T_{n}^{-}$ on $[\tau] \in \regMat_{2n,n} / \GL_{n}(\C)$ is compatible with respect to this decomposition, 
    it easily follows that 
    \begin{equation*}
        (T_{n}^{+} \times T_{n}^{-})_{[\tau]}
        \simeq 
        (T_{1}^{+} \times T_{1}^{-})^{k_{c}}   \times 
        (T_{1}^{+} \times T_{1}^{-})^{k_{d}}   \times
        \Delta(T_{1})^{k_{1}}   \times
        \Delta'(T_{2})^{k_{2}/2},
    \end{equation*}
    where we use the notation defined in Lemma \ref{Lem-Ttau=1-flag}.
    Because the Galois cohomology preserves the direct product, we get 
    \begin{equation*}
        H^{1}(\R, (T_{n}^{+} \times T_{n}^{-})_{[\tau]})
        \simeq 
        H^{1}(\R,T_{1}^{+} \times T_{1}^{-})^{k_{c}+k_{d}}   \times 
        H^{1}(\R,\Delta(T_{1}))^{k_{1}}   \times
        H^{1}(\R,\Delta'(T_{2}))^{k_{2}/2}.
    \end{equation*}
    By this isomorphism, Lemma \ref{Lem-Ttau=1-flag} tells us that
    \begin{equation*}
    \begin{array}{ccc}
        H^{1}(\R, (T_{n}^{+} \times T_{n}^{-})_{[\tau]}) & \xleftarrow{\;\sim\;} & 1 \times \{[(1,1)]_{\gamma},[(-1,-1)]_{\gamma}\}^{k_{1}} \times 1 \\
        \rotatebox{90}{$\in$} && \rotatebox{90}{$\in$} \\
        ~[(t,t)]_{\gamma}& & (1,(\varepsilon_{1}, \dots \varepsilon_{k_{1}}),1),
    \end{array}
    \end{equation*}
where $\varepsilon_{j} \in \{[(1,1)]_{\gamma},[(-1,-1)]_{\gamma}\}$ and $t =\diag(t_{1}, \dots, t_{n})$ is defined by 
    \begin{equation*}
        t_{j} = 
        \begin{cases}
            -1 & \text{ if $\varepsilon_{j} = [(-1,-1)]_{\gamma}$} \\
            1 & \text{ otherwise}.
        \end{cases}
    \end{equation*}
    From these arguments, we can easily see the inclusion \eqref{eq-Ttau-inclusion-TT_tau^-gamma-flag} gives the 
    isomorphism between $ T(\tau) $ and the Galois cohomology.  Note that $ k_1 = \# K_1 = d(\tau) $.
    %%
    %%By this, the inclusion \eqref{eq-Ttau-inclusion-TT_tau^-gamma-flag} gives the desired isomorphism.
    % Then, because the correspondence $(1,(\varepsilon_{1}, \dots \varepsilon_{k_{1}}),1) \mapsto (t,t)$ induces an isomorphism $T(\tau) \simeq 1 \times \{\pm 1\}^{k_{1}} \times 1$, this completes the proof.
    % we have
    % \begin{equation*}
    %     H^{1}(\R, (T_{n}^{+} \times T_{n}^{-})_{[\tau]})
    %     \simeq 
    %     \{\pm 1\}^{k_{1}},
    % \end{equation*}
    % which completes the proof.
    % Thus, by the definition of $T(\tau)$, it is sufficient to prove
    % % \begin{equation*}
    % %     H^{1}(\R, T_{1}^{+} \times T_{1}^{-}) = 1, \quad 
    % %     H^{1}(\R, \Delta(T_{1})) = \{[(1,1)]_{\gamma},[(-1,-1)]_{\gamma}\}, \quad
    % %     H^{1}(\R, \Delta'(T_{2})) = 1.
    % % \end{equation*}
    % This is given in Lemma \ref{Lem-H^1-tau-3cases-flag} below.
\end{proof}

Then, we are ready to prove \eqref{theorem-Galois-RDFV:item:2:Galois.tells.orbits.R} of Theorem \ref{theorem-Galois-RDFV}.

\begin{proof}[Proof of Theorem \ref{theorem-Galois-RDFV} \eqref{theorem-Galois-RDFV:item:2:Galois.tells.orbits.R}]
By Remark \ref{Rem-Ttau=dtau}, we have $T(\tau) \simeq \{\pm 1\}^{d(\tau)}$. Thus, Lemmas \ref{Lem-H^1(BB-tau)=H^1(TT-tau)-flag} and \ref{Lem-Ttau=H^1(TT,R)-flag} imply
\begin{equation*}
    H^{1}(\R,(B_{n}^{+} \times B_{n}^{-})_{[\tau]}) \simeq H^{1}(\R,(T_{n}^{+} \times T_{n}^{-})_{[\tau]}) \simeq T(\tau) \simeq  \{\pm 1\}^{d(\tau)},
\end{equation*}
which completes the proof.
\end{proof}

\subsection{Explicit representative of Galois cohomology and orbits}

In this subsection, we give an explicit correspondence of the Galois cohomology and representatives of orbits.  
%%(Proposition \ref{Prop-T(tau)=H1=H1=BO}).

\begin{lemma}\label{Lem-(t,t)=(t,1)gamma(t,1)}
    Let $\tau \in \regR_{n}$, and $T(\tau)$ be the subset of $  T_{n}^{+} \times  T_{n}^{-}$ defined in Definition \ref{Def-T-tau-flag}. Then, for any $(t,t) \in T(\tau)$, we have
    \begin{equation*}
        (t,t) = (t,\unitmatrix_{n})^{-1} \gamma(t,\unitmatrix_{n}),
    \end{equation*}
    where $\gamma$ is the involution given in \eqref{eq-def-gamma-caseA}.
\end{lemma}

\begin{proof}
Note that $t^{-1} = t = \overline{t}$ because for any $(t,t) \in T(\tau)$, $t$ is a diagonal matrix whose entries belong to $\{\pm 1\}$.
Thus, this lemma is clear by the definition of $\gamma$.
\end{proof}

\begin{proposition}\label{Prop-T(tau)=H1=H1=BO}
     Let $[\tau] \in \regR_{n} / \Sgroup_{n}$. Then, we have bijections 
     \begin{equation*}
     \begin{array}{ccccccc}
          T([\tau]) & \xrightarrow{\sim} & H^{1}(\R, (T_{n}^{+} \times T_{n}^{-})_{[\tau]}) & \xrightarrow{\sim} & H^{1}(\R, (B_{n}^{+} \times B_{n}^{-})_{[\tau]}) & \xrightarrow{\sim} & 
          B_{n} \backslash \mathcal{O}_{[\tau]}(\R) \\
          \rotatebox{90}{$\in$} && \rotatebox{90}{$\in$} & & \rotatebox{90}{$\in$} && \rotatebox{90}{$\in$} \\
          (t,t) & \mapsto & [(t,t)]_{\gamma} & \mapsto & [(t,t)]_{\gamma} & \mapsto & B_{n} \left((t,\unitmatrix_{n}) \cdot [\tau] \right),
     \end{array}
     \end{equation*}
     where $\mathcal{O}_{[\tau]} \coloneqq (B_{n}^{+} \times B_{n}^{-})/(B_{n}^{+} \times B_{n}^{-})_{[\tau]} \simeq (B_{n}^{+} \times B_{n}^{-})\cdot [\tau] \in (B_{n}^{+} \times B_{n}^{-}) \backslash \Mat_{2n,n} /\GL_{n}(\C)$ is the $(B_{n}^{+} \times B_{n}^{-})$-orbit corresponding to $[\tau]$.
\end{proposition}

\begin{proof}
    By Lemmas \ref{Lem-Ttau=H^1(TT,R)-flag}, \ref{Lem-(t,t)=(t,1)gamma(t,1)} and \ref{Lem-H^1(BB-tau)=H^1(TT-tau)-flag}, we have
     \begin{equation*}
     \begin{array}{ccccc}
          T([\tau]) & \xrightarrow{\sim} & H^{1}(\R, (T_{n}^{+} \times T_{n}^{-})_{[\tau]}) & \xrightarrow{\sim} & 
          H^{1}(\R, (B_{n}^{+} \times B_{n}^{-})_{[\tau]})  \\
          \rotatebox{90}{$\in$} && \rotatebox{90}{$\in$} & & \rotatebox{90}{$\in$} \\
          (t,\unitmatrix_{n})^{-1} \gamma(t,\unitmatrix_{n})& \mapsto & [(t,\unitmatrix_{n})^{-1} \gamma(t,\unitmatrix_{n})]_{\gamma} &\mapsto & [(t,\unitmatrix_{n})^{-1} \gamma(t,\unitmatrix_{n})]_{\gamma}.
     \end{array}
     \end{equation*}
     Moreover, by the isomorphism \eqref{eq-BX=H1(RBx)-flag}, we have
     \begin{equation*}
     \begin{array}{ccc}
          H^{1}(\R, (B_{n}^{+} \times B_{n}^{-})_{[\tau]}) & \xrightarrow{\sim} & 
          B_{n} \backslash \mathcal{O}_{[\tau]}(\R) \\
          \rotatebox{90}{$\in$} && \rotatebox{90}{$\in$} \\
          {[(t,\unitmatrix_{n})^{-1} \gamma(t,\unitmatrix_{n})]_{\gamma}} & \mapsto & B_{m} \left((t,\unitmatrix_{n}) \cdot [\tau] \right),
     \end{array}
     \end{equation*}
     which completes the proof.
\end{proof}

    Recall that we have $\Z_{2}^{n} \backslash \regOmega_{n} / (\Z_{2}^{n} \rtimes \Sgroup_{n}) \simeq B_{H} \backslash G /P_{G}$ by Theorem \ref{Thm-Gamma(n)=DFVR}. 
    Moreover, for any $B_{H_{\C}}$-orbit $\mathcal{O}$ on $G_{\C} / P_{\C}$, the set of real points $\mathcal{O}(\R)$ is nonempty if and only if $\mathcal{O}$ corresponds to $[\tau] \in \regR_{n} / \Sgroup_{n}$ (Theorem \ref{theorem-Galois-RDFV} (\ref{theorem-Galois-RDFV:item:1:nonempty.R.pts})). 
    Thus, by Proposition \ref{Prop-T(tau)=H1=H1=BO} and considering how $\mathcal{O}_{[\tau]}(\R)$ decomposes into $B_{H}$-orbits, we have a bijection
    \begin{equation*}
        \begin{array}{ccc}
            \coprod_{\tau \in \regR_{n}/ \Sgroup_{n}} T(\tau) & \xrightarrow{\sim} & \Z_{2}^{n} \backslash \regOmega_{n} / (\Z_{2}^{n} \rtimes \Sgroup_{n}) \\
            \rotatebox{90}{$\in$} & & \rotatebox{90}{$\in$}\\
            (t,t) & \mapsto & (t,\unitmatrix_{n}) \cdot [\tau].
        \end{array}
    \end{equation*}
From this, the signs in the coordinates of $ \tau' \in \regOmega_n $ (we denoted it by $ \tau $ in the former sections, but we want to avoid confusions here) do not matter when we generate the complex orbit from the real orbit corresponding to $ \tau' $.  Thus, if we define a map $ \varphi : \regOmega_n \to \regR_n $ by sending $ \tau' \in \regOmega_n $ to $ \tau \in \regR_n $ by removing all the signs, this establishes the correspondence of real orbits and complex orbits.
Therefore, we have a commutative diagram:
\begin{equation}\label{eq-comm-digagram-Omega-regR}
\xymatrix@C=50pt{
\Z_{2}^{n} \backslash \regOmega_{n} / (\Z_{2}^{n} \rtimes \Sgroup_{n}) 
\ar[d]^-{\wr \:\:\: \text{Thm.~\ref{Thm-Gamma(n)=DFVR}}}_-{[\tau'] \mapsto B_{H}[\tau']}
\ar[r]^-{\varphi}
&
\regR_{n}/ \Sgroup_{n}.
\ar@{_{(}->}[d]_{[\tau]\mapsto B_{H_{\C}}[\tau]}^-{\text{Thm.~\ref{Thm-complex-orbit-decomp-AIII-and-CI}}}
\\
B_{H} \backslash G / P 
\ar[r]_-{\mathcal{O} \mapsto B_{H_{\C}}\mathcal{O}} &
B_{H_{\C}} \backslash G_{\C} / P_{G_{\C}}
}
\end{equation}
Recall that we also have a commutative diagram:
\begin{equation}\label{eq-comm-digagram-DFVR-DFVC}
    \xymatrix@C=50pt{
B_{H} \backslash G / P 
\ar[d]_{\wr}
\ar[r]^-{\mathcal{O} \mapsto B_{H_{\C}}\mathcal{O}}
&
B_{H_{\C}} \backslash G_{\C} / P_{G_{\C}}.
\ar[d]_{\wr}
\\
H \backslash \dblFV 
\ar[r]^-{\mathcal{O} \mapsto H_{\C}\mathcal{O}} &
H_{\C} \backslash \dblFV_{\C}
\\
}
\end{equation}

\begin{corollary}\label{Cor-dim-R=dim-C}
    Let $\tau' \in \regOmega_{n} $ and $\scrorbitR_{\tau'} \in H \backslash \dblFV$ the corresponding $H$-orbit through the diagrams \eqref{eq-comm-digagram-Omega-regR} and \eqref{eq-comm-digagram-DFVR-DFVC}. 
    Set $\tau \coloneqq \varphi (\tau') \in \regR_{n}$ and $\scrorbitC_{\tau} \in H_{\C} \backslash \dblFV_{\C}$ corresponding $H_{\C}$-orbit.
    Then, we have $ \dim_{\R} \scrorbitR_{\tau'} = \dim_{\C} \scrorbitC_{\tau} $.
\end{corollary}

\begin{proof}
    By the commutative diagram \eqref{eq-comm-digagram-Omega-regR}, we have $H_{\C}\scrorbitR_{\tau'} = \scrorbitC_{\tau}$. Thus, there exists $x \in \dblFV$ such that we have
    \begin{equation*}
        \scrorbitR_{\tau'} \simeq H / H_{x},\quad
        \scrorbitC_{\tau}  \simeq H_{\C} / (H_{\C})_{x}.
    \end{equation*}
    Because $x \in \dblFV$ implies $\dim_{\C} (H_{\C})_{x} = \dim_{\R} H_{x}$, we have
    \begin{equation*}
        \dim_{\R} \scrorbitR_{\tau'} = \dim_{\R} H -\dim_{\R} H_{x} = \dim_{\C} H_{\C} -\dim_{\C} (H_{\C})_{x}=\dim_{\C} \scrorbitC_{\tau},
    \end{equation*}
    which completes the proof.
\end{proof}

\begin{remark}
    In \cite[Thm.~2.2 (2)]{Fresse.N.2023}, the dimension of the orbits in $ H_{\C} \backslash \dblFV_{\C} $ is given by a combinatorial expression using the parameter $ \tau \in \regR_n $ in Case \caseA. Thus, the dimension of the orbits in $ H \backslash \dblFV$ can also be obtained by a combinatorial expression using the parameter $ \tau \in \regOmega_{n} $ by Corollary \ref{Cor-dim-R=dim-C}. Moreover, if $\scrorbitC \in H_{\C} \backslash \dblFV_{\C}$ and 
    $\calorbit \in B_{H_{\C}} \backslash G_{\C} / P_{G_{\C}}$ correspond each other via the isomorphism in the diagram \eqref{eq-comm-digagram-DFVR-DFVC}, then we have
    \begin{equation*}
        \dim_{\C} \scrorbitC = \dim_{\C} \calorbit + \dim_{\C} H_{\C} / B_{H_{\C}}.
    \end{equation*}
\end{remark}

\section{Families of flag varieties and Matsuki duality}\label{Section-Matsuki-duality}

Recall that we have isomorphisms 
\begin{align*}
H \backslash \dblFV \simeq B_H \backslash G / P_S 
\simeq B_{H} \backslash \LregMat_{2n,n} / \GL_{n}(\C) = \coprod_{I\subset[n]} B_{H}\backslash \mathrm{LM}_{2n,n}^{I} / \GL_{n}(\C)
\end{align*}
in Case \caseA (see the results in \S \ref{Section-Orbit-decomposition}).  
Let us consider the double coset space indexed by $ I $ in the last expression.  
If we put $ m = \# I $, 
it is isomorphic to $ B_m(\C) \backslash \Her_m(\C) $  (Lemma \ref{Lem-(B-Her)=(B-LM-I/GL)}).
We will reduce the orbit classification of $B_{m}(\C) \backslash \Her_{m}(\C)$ first to 
those of smaller full flag varieties for the action of 
smaller indefinite unitary/orthogonal groups  
(we have already stated the precise statement in Theorem \ref{theorem-DFV-contain-KGB}).  
By Matsuki duality, the classification boils down 
to the KGB classification over complex flag varieties (see below).    
In the setting of Case \caseA, there exists a well known combinatorial description of $ K_{\C} $-orbits by Matsuki-Oshima in terms of ``clans'' (see \cite{Matsuki.Oshima.1990}).  
For Case \caseB, we will get a new interpretation of ``clans'' for the double coset space $ B_m(\R) \backslash \GL_m(\R) / \OO(p,q) $.  

In this way, we get yet another classification of the orbits in the double flag variety $ \dblFV $.  

We will mainly explain the classification of orbits in detail for Case \caseA, and the first part of this section also gives a proof of Theorem \ref{theorem-DFV-contain-KGB} for both cases \caseA and \caseB.
%%Thus, we divide this section into two parts.

\subsection{Reduction to smaller flag varieties}

%%Now we restore our original notations so that $ G $ denotes $ \U(n, n) $ or $ \Sp_{2n}(\R) $ (Notation \ref{Notation-caseA-and-caseB}).  
In this subsection, we give a proof of Theorem \ref{theorem-DFV-contain-KGB}. 
The results obtained in the course of the proof will be used in the latter section for giving another proof of Proposition \ref{Prop-SPI=(B-Her)} (which is already proved in Section \ref{Section-Gaussian-elimination} by using a version of Gaussian elimination).
As in the previous sections, we only treat Case \caseA in this subsection because we can apply the same proof \textit{mutatis mutandis} to Case \caseB by rereading the symbols according to the list in Notation \ref{Notation-caseA-and-caseB-2} below.

\begin{notation}\label{Notation-caseA-and-caseB-2}
The corresponding table of notations for Case \caseA and case \caseB of Setting \ref{Setting-DFVR} is as follows:
\begin{equation*}
    \begin{array}{|c|c|c|c|c|c|c|c|c|c|}\hline
         &  & B_{m} &  & a^{*}    & &
        \\ \hline
        \text{\caseA} & \C & B_{m}(\C) & \Her_{m}(\C) & a^{*} & \Her_{m}^{(p,q;r)}(\C)  & \U(p,q) 
        \\ \hline
        \text{\caseB} & \R  & B_{m}(\R) & \Sym_{m}(\R) & \transpose{a}   & \Sym_{m}^{(p,q;r)}(\R) & \OO(p,q) 
        \\ \hline
    \end{array}
\end{equation*}
where $\Her_{m}^{(p,q;r)}(\C)$ is defined in Definition \ref{Def-Her-pqr}, and $\Sym_{m}^{(p,q;r)}(\R)$ is defined similarly:
\begin{equation*}
	\Sym_{m}^{(p,q;r)}(\R) \coloneqq \{Z \in \Sym_{m}(\R)  \mid \sign Z = (p,q,r) \},
\end{equation*}
where $\sign Z$ denotes the signature of a real symmetric matrix, i.e., $ p $ denotes the number of positive eigenvalues, $ q $ for negative ones, and $ r $ denotes the multiplicity of the eigenvalue $ 0 $.
\end{notation}

In order to know the orbit decomposition $B_{m} \backslash \Her_{m}(\C) $, we first divide $\Her_{m}(\C)$ into $B_{m}$-stable subsets.

\begin{definition}\label{Def-Her-pqr}
Let $p,q \geq 0$ satisfying $p+q\leq m$ and set $r\coloneqq m-(p+q)$.
We define a subset $\Her_{m}^{(p,q;r)}(\C)$ of $\Her_{m}(\C)$ by
\begin{equation*}
	\Her_{m}^{(p,q;r)}(\C) \coloneqq \{Z\in \Her_{m}(\C)  \mid \sign Z = (p,q,r) \} .
\end{equation*}
%%where $\sign Z$ denotes the signature of $Z$.
\end{definition}

Let us denote
\begin{equation}\label{eq-def-I_p,q,r}
	I_{p,q;r} \coloneqq
	\begin{pmatrix}
		I_{p,q} & 0 
		\\
		0 & 0 
	\end{pmatrix}  \in  \Her_{m}^{(p,q;r)}(\C),
\quad \text{ where we set } \quad 
	I_{p,q}
    \coloneqq
    \begin{pmatrix}
	   \unitmatrix_{p} & 0 \\
	   0 & -\unitmatrix_{q}
    \end{pmatrix}.
\end{equation}

%%The following lemma is obvious.

\begin{lemma}\label{Lem-Her=coprod-Her-rst}\label{Lem-Hep-qrt-stable}\label{Lem-B-Her=coprod-B-Her-rst}
We have a natural bijection
\begin{equation*}
	B_{m} \backslash \Her_{m}(\C)
	=
	\coprod_{0\leq p+q \leq m}
	B_{m} \backslash \Her_{m}^{(p,q;r)} (\C).
\end{equation*}	
\end{lemma}

\begin{remark}\label{remark:Her-pqr.Bm.isom.SPIst-pqr}
In Proposition \ref{Prop-SPI=(B-Her)}, 
we have described the representatives $ \SPIst_m $ of $ B_{m} \backslash \Her_{m}(\C) $.  
If we consider also the signature, this is refined as follows.
\begin{equation*}
    B_{m} \backslash \Her_{m}^{(p,q;r)} (\C) \simeq  \{ \tau \in \SPIst_m \mid \sign \tau = (p,q, r) \} .
\end{equation*}
Even if we replace $ \Her_{m}^{(p,q;r)} (\C) $ by $ \Sym_{m}^{(p,q;r)} (\R) $ with the action of the real Borel subgroup, we still get the same parameter set in the right hand side.
\end{remark}

\begin{proof}
This follows from a $ B_m $-stable decomposition 
\begin{equation*}
	\Her_{m}(\C)
	=
	\coprod_{0\leq p+q \leq m}
	\Her_{m}^{(p,q;r)}(\C).
% \label{eq-coprod-Her}
\end{equation*}
\end{proof}

%%%%%%%% skipover:begin %%%%%%%%%%%%%%%%%%
\skipover{
\begin{proof}
Regarding $\sign $ as a function $\sign \colon \Her_{m}(\C)\to \Z_{\geq 0}^{3}$,
we have $\Her_{m}^{(p,q;r)} (\C)  = \sign ^{-1}(p,q,r)$,
which completes the proof.
\end{proof}

\begin{lemma}\label{Lem-Hep-qrt-stable}
$\Her_{m}^{(p,q;r)}(\C)$ is $B_{m}$-invariant, where the action is given by \eqref{eq-def-B-action-Her}.
\end{lemma}

\begin{proof}
The signature of a Hermitian matrix is $\GL_{m}(\C)$-invariant, where the action is given by
\begin{equation*}
    g\cdot Z\coloneqq gZg^{*}.
\end{equation*}
In particular, it is invariant under the action of $B_{m}\subset \GL_{m}(\C)$, which completes the proof.
\end{proof}

\begin{lemma}\label{Lem-B-Her=coprod-B-Her-rst}
We have
\begin{equation*}
	B_{m} \backslash \Her_{m}(\C)
	=
	\coprod_{0\leq p+q \leq m}
	B_{m} \backslash \Her_{m}^{(p,q;r)} (\C).
\end{equation*}	
\end{lemma}

\begin{proof}
This follows from Lemmas \ref{Lem-Her=coprod-Her-rst} and \ref{Lem-Hep-qrt-stable}.
\end{proof}
}
%%%%%%%% skipover:end   %%%%%%%%%%%%%%%%%%

In order to consider $B_{m} \backslash \Her_{m}^{(p,q;r)}(\C)$, 
we realize $\Her_{m}^{(p,q;r)}(\C)$ as a homogeneous space.
%%For this purpose, we need a notation.

\begin{definition}\label{Def-P-(p,q;r)}
Define a subgroup $Q_{(p,q;r)}$ of $\GL_{m}(\C)$ by
\begin{equation}
	Q_{(p,q;r)}
    \coloneqq
	\left\{
	   \begin{pmatrix}
		   u & m \\
	       0 & g
	   \end{pmatrix}
	\Bigm| 
	   \begin{array}{c}
	       u\in \U(p,q)\\
		   m\in \mathrm{M}_{p+q,r}(\C)\\
		   g\in \GL_{r}(\C)
	   \end{array}
	\right\}.
\label{eq-def-P-(p,q;r)}
\end{equation}
\end{definition}

\begin{lemma}\label{Lem-G-P=Her-rst}\label{Lem-B-GL-P=B-Her-rst}
Consider a map
\begin{equation}
\begin{array}{ccc}
    \GL_{m}(\C) & \to & \Her_{m}^{(p,q;r)}(\C) \\
	\rotatebox{90}{$\in$} && \rotatebox{90}{$\in$} \\
	g & \mapsto & gI_{p,q;r}g^{*}.
\end{array}
\label{eq-def-psi}
\end{equation}
This map 
%%\eqref{eq-def-psi} 
induces a $B_{m}$-equivariant isomorphism:
\begin{equation*}
\begin{array}{ccc}
    \GL_{m}(\C) / Q_{(p,q;r)} & \xrightarrow{\sim} & \Her_{m}^{(p,q;r)} (\C) \\
	\rotatebox{90}{$\in$} && \rotatebox{90}{$\in$} \\
	gQ_{(p,q;r)} & \mapsto & gI_{p,q;r}g^{*},
\end{array}
\end{equation*}
where $I_{p,q;r}$ is defined in \eqref{eq-def-I_p,q,r}.  
As a consequence, we get a bijection 
\begin{equation*}
	B_{m} \backslash \GL_{m}(\C) / Q_{(p,q;r)} \xrightarrow{\sim} B_{m}\backslash \Her_{m}^{(p,q;r)} (\C).
\end{equation*}
\end{lemma}

\begin{proof}
The map \eqref{eq-def-psi} defines an action of $\GL_{m}(\C)$ on $\Her_{m}^{(p,q;r)} (\C)$.
Therefore, it is sufficient to prove that the action is transitive, 
and that the isotropy subgroup of $\GL_{m}(\C)$ at $I_{p,q;r}$ is equal to $Q_{(p,q;r)}$.
The former claim follows from Sylvester's law of inertia (\cite[Chap.~IX, \S 7, n$^{\mathrm{o}}$2, Thm.~1]{Bourbaki.Algebra.9}),
which states that any element of $\Her_{m}^{(p,q,;r)}(\C)$ is transformed to $I_{p,q;r}$ by the action of $\GL_{m}(\C)$.
The latter claim follows from an easy calculation.
%%Moreover, the $B_{m}$-equivariance follows from the definition of the action \eqref{eq-def-B-action-Her}.
%%The rest of the statements in the lemma is obvious.
\end{proof}

Now we analyze the double coset space $B_{m} \backslash \GL_{m}(\C) / Q_{(p,q;r)}$.  
We will describe the space $\GL_{m}(\C) / Q_{(p,q;r)}$ as a total space of a certain principal fiber bundle,
and use this structure to reduce the classification of the orbit decomposition of $\GL_{m}(\C) / Q_{(p,q;r)}$
to that of the fibers of this principal fiber bundle.

\begin{definition}\label{Def-Q-(p+q,r)}
Define $P_{(p+q,r)}$ by the maximal parabolic subgroup of  $\GL_{m}(\C)$ corresponding to the partition $m=(p+q)+r$:
\begin{equation*}
	P_{(p+q,r)} \coloneqq
	\left\{
	   \begin{pmatrix}
		  g_{1} & u \\
		  0 & g_{2}
	   \end{pmatrix}
	\Bigm| 
	   \begin{array}{c}
		  u\in \Mat_{p+q,r}(\C) \\
		  g_{1}\in \GL_{p+q}(\C), g_{2}\in \GL_{r}(\C)
	   \end{array}
	\right\}.
\end{equation*}
\end{definition}

Since $ Q_{(p,q;r)} \subset P_{(p+q,r)} $, 
%%(see Definition \ref{Def-P-(p,q;r)} for $Q_{(p,q;r)}$).
we get a natural projection 
\begin{equation}\label{eq-proj-GL/P-to-GL/Q}
	\begin{array}{ccc}
		\pi\colon \GL_{m}(\C) / Q_{(p,q;r)} & \to &  \GL_{m}(\C) / P_{(p+q,r)} \\
		\rotatebox{90}{$\in$} & & \rotatebox{90}{$\in$} \\
		gQ_{(p,q;r)} & \mapsto & gP_{(p+q,r)}, 
	\end{array}
\end{equation}
which is a fiber bundle with the fiber 
$  P_{(p+q,r)} / Q_{(p,q;r)} \simeq \GL_{p+q}(\C) / \U(p,q) $.  
The $B_{m}$-orbit decomposition of the base space $ \GL_{m}(\C) / P_{(p+q,r)} $ of this fiber bundle \eqref{eq-proj-GL/P-to-GL/Q} is well known as the generalized Bruhat decomposition.

\begin{theorem}[{\cite[Chap.~IV, \S14, 14.16]{Borel.1991}}]\label{Thm-generalized-Bruhat-decomp}
We have
\begin{equation*}
	B_{m} \backslash \GL_{m}(\C) / P_{(p+q,r)}
	\simeq 
	\coprod_{\sigma \in \Sgroup_{m} /(\Sgroup_{p+q}\times \Sgroup_{r})}
	B_{m} \sigma P_{(p+q,r)}.
\end{equation*}
\end{theorem}

%%%%%%%% skipover:begin %%%%%%%%%%%%%%%%%%
\skipover{
\begin{proof}
This is the generalized Bruhat decomposition \cite[Chap.~IV, \S14, 14.16]{Borel.1991}.
\end{proof}
}
%%%%%%%% skipover:end   %%%%%%%%%%%%%%%%%%

Thus, the structure of $B_{m} \backslash \GL_{m}(\C) / Q_{(p,q;r)}$ is reduced to that of the orbit structure of the fibers of \eqref{eq-proj-GL/P-to-GL/Q} by 
a general lemma given below.
%%Lemma \ref{Lem-orbit-decomp-fiber-bundle} below
In the following lemma, we use a general notation 
for $ G, P, Q $ and $ B $ apart from the usage of the rest of this paper.

\begin{lemma}\label{Lem-orbit-decomp-fiber-bundle}
Let $G$ be a group, $Q\subset P$ subgroups of $G$, and
\begin{equation*}
	\pi \colon G/Q \to G/P
\end{equation*}
a natural projection.
Moreover, let $B$ be a subgroup of $G$, and $S\subset G$ a set of representatives of $B\backslash G/P$:
\begin{equation*}
	B\backslash G/P
	\simeq 
	\coprod_{s\in S} B s P.
\end{equation*}
Then, there exists a natural bijection between the double coset spaces:
%%the $B$-orbit decomposition of the total space $ B\backslash G / P $ is given by
%
\begin{equation*}
\begin{array}{ccc}
	\coprod\limits_{s\in S} (P\cap B^{s^{-1}}) \backslash P/Q & \xrightarrow{\;\sim\;} &  B\backslash G/Q \\
    \rotatebox{90}{$\in$} && \rotatebox{90}{$\in$} \\
    (P\cap B^{s^{-1}} ) \, x \, Q & \mapsto & B \, s x \, Q.
\end{array}
\end{equation*}
\end{lemma}
\begin{proof}
Although this lemma seems to be well known, we give a proof in Appendix \ref{Section-proof-of-Lem-orbit-decomp-fiber-bundle} for the sake of completeness.
\end{proof}

In order to apply Lemma \ref{Lem-orbit-decomp-fiber-bundle} to the map \eqref{eq-proj-GL/P-to-GL/Q},
we need to specify the representatives of  $\Sgroup_{m} / (\Sgroup_{p+q} \times  \Sgroup_{r})$ in Theorem \ref{Thm-generalized-Bruhat-decomp}.

\begin{definition}\label{Def-sigma-J-bar}
Let $J\coloneqq\{j_{1}<j_{2}<\dots<j_{p+q}\}\subset \{1,2,\dots,m\}$ and $\{1,2,\dots,m\}\backslash J \eqqcolon \{j'_{1}<j'_{2}<\dots<j'_{r}\}$.
Then, we define $\overline{\sigma}_{J} \in \Sgroup_{m} $ by
\begin{equation*}
    \overline{\sigma}_{J}
    \coloneqq
    \begin{pmatrix}
        1 & \dots & p+q & p+q+1 & \dots & m\\
        j_{1} & \dots & j_{p+q} & 
        j'_{1} & \dots & j'_{r}
    \end{pmatrix}, 
\end{equation*}
which is called a Grassmannian permutation.
Note that we have $\sigma_{I}=\overline{\sigma}_{J} \in \Sgroup_{n}$ if $n=m$, where $\sigma_{I}$ is given in Definition \ref{Def-sigma-I}.
\end{definition}

Let us recall the notation 
\begin{equation*}
    \dbinom{[m]}{k} 
    = \{ J \subset [m] \mid \# J = k \}, 
    \quad
    \text{where $ [m] = \{ 1, \dots, m \} $.}
\end{equation*}

\begin{lemma}\label{Lem-sigma-J-representative}
Let $\overline{\sigma}_{J}$ be 
the Grassmannian permutation defined above.
%%the element of  $\Sgroup_{m}$ defined in Definition \ref{Def-sigma-J-bar}.
Then,
\begin{equation*}
    \Bigl\{ \overline{\sigma}_{J} \in \Sgroup_{m}
    \Bigm| 
    J \in \dbinom{[m]}{p + q}
%%        J\subset\{1,2,\dots,m\},\:   \# J =p+q
    \Bigr\}
\end{equation*}
is a set of representatives of $ \Sgroup_{m}/ (\Sgroup_{p+q}\times  \Sgroup_{r})$.
\end{lemma}

\begin{remark}\label{Remark-sigma-J-correspond-coprod}
   The index $J$ in the isomorphism in Theorem \ref{theorem-DFV-contain-KGB} corresponds to $J \subset [m] $ in Lemma \ref{Lem-sigma-J-representative}. 
   Thus, we can interpret $J$ as the representatives of $ \Sgroup_{m}/ (\Sgroup_{p+q}\times  \Sgroup_{r})$, 
   which is the Weyl group corresponding to the Grassmannian $\GL_{m}(\C) / P_{(p+q,r)}$.
\end{remark}

\begin{proof}[Proof of Lemma \ref{Lem-sigma-J-representative}]
We have an isomorphism
\begin{equation*}
	\begin{array}{ccc}
		\Sgroup_{m} / (\Sgroup_{p+q} \times \Sgroup_{r}) & \xrightarrow{\sim}  & \dbinom{[m]}{p + q}  
%%        \{ J\subset \{1,2,\dots,m\}  \mid \#J=p+q \}	
        \\
		\rotatebox{90}{$\in$} && \rotatebox{90}{$\in$} \\
		\sigma & \mapsto & \{\sigma(1), \sigma(2), \dots, \sigma(p+q)\},
	\end{array}
\end{equation*}
which implies the lemma.
\end{proof}

Thus, we can apply Lemma \ref{Lem-orbit-decomp-fiber-bundle}
to the fiber bundle \eqref{eq-proj-GL/P-to-GL/Q} and get the following lemma.

\begin{lemma}\label{Lem-B-GL-P=coprod-Q-cap-B-Q-P}
We have an isomorphism:
\begin{equation*}
\begin{array}{ccc}
\displaystyle
	\coprod_{
    J \in \binom{[m]}{p + q} 
%%	   \substack{ J\subset \{1,2,\dots,m\} \\  \#J=p+q }
	}
	( P_{(p+q,r)} \cap B_{m}^{\overline{\sigma}_{J}^{-1}} ) \backslash P_{(p+q,r)} / Q_{(p,q;r)}
    &
    \xrightarrow{\sim}
    & 
	B_{m} \backslash \GL_{m}(\C) / Q_{(p,q;r)}
    \\
    \rotatebox{90}{$\in$} && \rotatebox{90}{$\in$}
    \\
	( P_{(p+q,r)} \cap  B_{m}^{\overline{\sigma}_{J}^{-1}} ) \, x \, Q_{(p,q;r)}
    & \mapsto &
    B_{m} \overline{\sigma}_{J} \, x \, Q_{(p,q;r)}.
\end{array}
\end{equation*}
\end{lemma}

%%%%%%%% skipover:begin %%%%%%%%%%%%%%%%%%
\skipover{
\begin{proof}
It is sufficient to apply Lemma \ref{Lem-orbit-decomp-fiber-bundle} to the map \eqref{eq-proj-GL/P-to-GL/Q}
by using Fact \ref{Thm-generalized-Bruhat-decomp} and Lemma \ref{Lem-sigma-J-representative}.
\end{proof}

Let us consider $( P_{(p+q,r)} \cap  B_{m}^{\overline{\sigma}_{J}^{-1}}) \backslash P_{(p+q,r)} / Q_{(p,q;r)} $.
}
%%%%%%%% skipover:end   %%%%%%%%%%%%%%%%%%
Our next goal is 
%%Lemma \ref{Lem-orbit-decomp-B-GL-U=BB-Q-P} given below, 
to identify the double coset space $(P_{(p+q,r)} \cap B_{m}^{\overline{\sigma}_{J}^{-1}}) \backslash P_{(p+q,r)} / Q_{(p,q;r)} $
with $ B_{p+q} \backslash \GL_{p+q}(\C) / \U(p,q) $.
Before doing so, we need one more lemma.

\begin{lemma}\label{Lem-representative-Q-cap-B}
Let $ \overline{\sigma}_{J} \in \Sgroup_{m}$ be a Grassmannian permutation defined in Definition \ref{Def-sigma-J-bar}.
Then, there exists a subspace $N_{J}\subset \Mat_{p+q,r}(\C)$ satisfying
\begin{equation*}
% \label{eq:BsigmaJ.contains.smallerBorels}
P_{(p+q,r)}\cap B_{m}^{\overline{\sigma}_{J}^{-1}}
=
\left\{
    \begin{pmatrix}
        b_{p+q} & n \\
        0 & b_{r}
    \end{pmatrix}
\Bigm| 
    \begin{array}{c}
        b_{p+q} \in B_{p+q} \\
        b_{r} \in B_{r}     \\
        n\in N_{J}  	
    \end{array}
\right\}.
\end{equation*}
\end{lemma}

The proof is similar to that of Lemma \ref{Lem-B-p-cap-GL-times-GL}.  So we omit it.
Now let us prove

\begin{lemma}\label{Lem-orbit-decomp-B-GL-U=BB-Q-P}
Let $J \subset [m] $ with $\# J =p+q$.
Then, we have an isomorphism
\begin{equation}
\begin{array}{ccc}
	B_{p+q} \backslash \GL_{p+q}(\C) / \U(p,q) & \xrightarrow{\;\;\sim\;\;}&
	( P_{(p+q,r)} \cap B_{m}^{\overline{\sigma}_{J}^{-1}} )  \backslash P_{(p+q,r)} / Q_{(p,q;r)}
	\\
	\rotatebox{90}{$\in$}  && \rotatebox{90}{$\in$}
	\\
	B_{p+q}\,g \,\U(p,q)  & \mapsto &
	( P_{(p+q,r)} \cap B_{m}^{\overline{\sigma}_{J}^{-1}} )
	\begin{pmatrix}
		g & 0
		\\
		0 & \unitmatrix_{r}
	\end{pmatrix}  Q_{(p,q;r)}.
\end{array}
\label{eq-def-psi''_J-bar}
\end{equation}	
\end{lemma}

\begin{proof}
The map is well-defined and its surjectivity is clear.  
So let us prove the injectivity.  
Suppose that 
(by Lemma \ref{Lem-representative-Q-cap-B}) 
there exist
\begin{equation*}
	b_{p+q}\in B_{p+q},
	\quad
	b_{r}\in B_{r},
	\quad
	n\in N_{J},
	\quad
	u\in \U(p,q),
	\quad
	m\in \Mat_{p+q,r}(\C),
	\quad
	h\in \GL_{p}(\C)
\end{equation*}
such that
\begin{equation}
	\begin{pmatrix}
		b_{p+q} & n
		\\
		0 & b_{r}
	\end{pmatrix}
	\begin{pmatrix}
		g & 0
		\\
		0 & \unitmatrix_{r}
	\end{pmatrix}
	\begin{pmatrix}
		u & m
		\\
		0 & h
	\end{pmatrix}
	=
	\begin{pmatrix}
		g' & 0
		\\
		0 & \unitmatrix_{r}
	\end{pmatrix}.
\label{eq-BBgP=BBgP}
\end{equation}
The left-hand side of this equation is equal to
\begin{align*}
	\begin{pmatrix}
		b_{p+q} & n
		\\
		0 & b_{r}
	\end{pmatrix}
	\begin{pmatrix}
		g & 0
		\\
		0 & \unitmatrix_{r}
	\end{pmatrix}
	\begin{pmatrix}
		u & m
		\\
		0 & h
	\end{pmatrix}
	&=
	\begin{pmatrix}
		b_{p+q}gu & b_{p+q}gm+nh
		\\
		0 & b_{r}h
	\end{pmatrix}.
\end{align*}
Thus, comparing the $(1,1)$-entry of \eqref{eq-BBgP=BBgP}, we have
\begin{equation*}
	B_{p+q}g \U(p,q)
	=
	B_{p+q}g' \U(p,q),
\end{equation*}
which completes the proof.
\end{proof}

Summarizing the discussion so far (Lemmas \ref{Lem-Her=coprod-Her-rst}, \ref{Lem-B-GL-P=coprod-Q-cap-B-Q-P} and \ref{Lem-orbit-decomp-B-GL-U=BB-Q-P}), we obtain the following lemma.

\begin{lemma}\label{Lem-B-Her=BGLU}
For $ p, q \geq 0 $ satisfying $ p + q \leq m $, 
we have a bijection
    \begin{equation*}
        B_{m} \backslash \Her_{m}^{(p,q;r)} (\C) 
\simeq  \coprod_{J \in \binom{[m]}{p + q}}
        B_{p+q} \backslash \GL_{p+q}(\C) / \U(p,q) .
    \end{equation*}
Here, for different $ J $'s, the summands are all the same in the right hand side, and 
it means the disjoint union of the same set with multiplicity $ \binom{m}{p + q} $. 
\end{lemma}

%%%%%%%% skipover:begin %%%%%%%%%%%%%%%%%%
\skipover{
\begin{proof}
By Lemmas \ref{Lem-B-GL-P=coprod-Q-cap-B-Q-P} and \ref{Lem-orbit-decomp-B-GL-U=BB-Q-P}, we have
\begin{equation}
     \coprod_{\substack{ J\subset \{1,2,\dots,m\} \\  \#J=p+q }}
     B_{p+q} \backslash \GL_{p+q}(\C) / \U(p,q)  
     \xrightarrow{\sim}  
     B_{m} \backslash \GL_{m}(\C) / Q_{(p,q;r)}.
\label{eq-Lem-B-Her=BGLU-map-in-proof}
\end{equation}
On the other hand, by Lemmas \ref{Lem-B-Her=coprod-B-Her-rst} and \ref{Lem-B-GL-P=B-Her-rst}, we have
\begin{equation*}
     \coprod_{0\leq p+q \leq m} B_{m} \backslash \GL_{m}(\C) / Q_{(p,q;r)}  
     \xrightarrow{\sim} 
     \coprod_{0\leq p+q \leq m} B_{m} \backslash \Her_{m}^{(p,q;r)} (\C) = B_{m}\backslash \Her_{m}(\C).
\end{equation*}
Thus, by taking the coproduct of the map \eqref{eq-Lem-B-Her=BGLU-map-in-proof}, we have the lemma.
\end{proof}
}
%%%%%%%% skipover:end   %%%%%%%%%%%%%%%%%%

Now Lemmas \ref{Lem-B-G-P_S=coprod-B-Her}, 
\ref{Lem-B-Her=coprod-B-Her-rst} and \ref{Lem-B-Her=BGLU} prove Theorem \ref{theorem-DFV-contain-KGB} 
for Case \caseA.

\begin{theorem}\label{theorem-DFV-contaion-KGB-caseA}
Orbits in the double flag variety have a decomposition 
    \begin{equation*}
H \backslash \dblFV \simeq 
        B_{H} \backslash G /P_{S} 
= \coprod_{m = 0}^n \coprod_{I \in \binom{[n]}{m}} B_{H}\backslash \mathrm{LM}_{2n,n}^{I} /\
 \GL_{n}(\C) ,
\end{equation*}
and for each $ I \in \binom{[n]}{m} $, we have an isomorphism 
\begin{equation*}
B_{H}\backslash \mathrm{LM}_{2n,n}^{I} /\
 \GL_{n}(\C) \simeq B_m(\C) \backslash \Her_m(\C)
\simeq \coprod_{0 \leq p + q \leq m} \coprod_{J \in \binom{[m]}{p + q} }
%%        \coprod_{\substack{I\subset\{1,2,\dots,n\}\\
%%	                           J\subset \{1,2,\dots,\# I\}\\
%%	                           0\leq \# J=p+q \leq \# I
%%                          }
%%	            }
	    B_{p+q}(\C)  \backslash  \GL_{p+q}(\C) / \U(p,q).
    \end{equation*}
In the expression on the right hand side, the summands only depend on $ p $ and $ q $, 
and the indices $ \{ J \} $ only represent multiplicities.  
\end{theorem}

\begin{remark}
When $ p = q = 0 $ in the above formula, we get $ J = \emptyset $.  
We regard $ B_{p+q}(\C)  \backslash  \GL_{p+q}(\C) / \U(p,q) = \{ \ast \} $ (as the one point set) if $ p = q = 0 $.  
There are $ 2^n $ of such orbits and they contain all closed orbits.
\end{remark}

%%%%%%%% skipover:begin %%%%%%%%%%%%%%%%%%
\skipover{
\begin{proof}
By Lemmas \ref{Lem-B-G-P_S=coprod-B-Her} and \ref{Lem-B-Her=BGLU}, we have 
    \begin{equation*}
        B_{H} \backslash G /P_{S}
        \simeq 
        \coprod_{I\subset [n]} 
        B_{\# I} \backslash \Her_{\# I}(\C)
        \simeq
        \coprod_{ \substack{ I \subset [n] \\ J\subset \{1,2,\dots, \# I\} \\  \#J=p+q } }
        B_{p+q} \backslash \GL_{p+q}(\C) / \U(p,q),
    \end{equation*}
    which completes the proof.
\end{proof}
}
%%%%%%%% skipover:end   %%%%%%%%%%%%%%%%%%

The results in this subsection stays valid for Case \caseB with minor changes.  
We only state the main result below.  

\begin{theorem}\label{theorem-DFV-contaion-KGB-caseB}
Let us assume we are in Case \caseB 
so that $ G = \Sp_{2n}(\R) $.   
% There exists a bijection
%     \begin{equation*}
%         B_{H} \backslash G /P_{S} 
%         \simeq 
%         \coprod_{m = 0}^n \coprod_{I \in \binom{[n]}{m}} \coprod_{p + q = m} \coprod_{J \in \binom{[m]}{p + q} }
% 	    B_{p+q}(\R)  \backslash  \GL_{p+q}(\R) / \OO(p,q).
%     \end{equation*}
% %
% In the expression on the right hand side, the summands only depend on $ p $ and $ q $, 
% and the indices $ I $ and $ J $ only represent multiplicities.  
% 
Orbits in the double flag variety have a decomposition 
    \begin{equation*}
H \backslash \dblFV \simeq 
        B_{H} \backslash G /P_{S} 
= \coprod_{m = 0}^n \coprod_{I \in \binom{[n]}{m}} B_{H}\backslash \mathrm{LM}_{2n,n}^{I}(\R) /\
 \GL_{n}(\R) ,
\end{equation*}
where we set $\mathrm{LM}_{2n,n}^{I}(\R) \coloneqq \mathrm{LM}_{2n,n}^{I} \cap \Mat_{2n,n}(\R)$,
and for each $ I \in \binom{[n]}{m} $, we have an isomorphism 
\begin{align*}
B_{H}\backslash \mathrm{LM}_{2n,n}^{I}(\R) /\
 \GL_{n}(\R) &\simeq B_m(\R) \backslash \Sym_m(\R)
\\
&\simeq \coprod_{0 \leq p + q \leq m} \coprod_{J \in \binom{[m]}{p + q} }
%%        \coprod_{\substack{I\subset\{1,2,\dots,n\}\\
%%	                           J\subset \{1,2,\dots,\# I\}\\
%%	                           0\leq \# J=p+q \leq \# I
%%                          }
%%	            }
	    B_{p+q}(\R)  \backslash  \GL_{p+q}(\R) / \OO(p,q).
\end{align*}
In the expression on the right hand side, the summands only depend on $ p $ and $ q $, 
and the indices $ \{ J \} $ only represent multiplicities. 
\end{theorem}

\subsection{KGB decomposition and the Matsuki duality}\label{subsec:KGB.Matsuki.duakity}

%%%%%%%% skipover:begin %%%%%%%%%%%%%%%%%%
\skipover{
\cite{Mirkovic.Uzawa.Vilonen.1992}
\cite{Matsuki.Mduality.1988}: all claims
\cite{Matsuki.1991}: ICM, except closure relation
}
%%%%%%%% skipover:end   %%%%%%%%%%%%%%%%%%

For a moment, 
let $ G_{\C} $ denote an arbitrary connected complex reductive group, 
$ G_{\R} $ its real form.  
Choose a maximal compact subgroup $ K_{\R} $ of $ G_{\R} $ 
and let  
$ K_{\C} $ be its complexification.  

Take a Borel subgroup $ B_{\C} $ of $ G_{\C} $ 
and consider the full flag variety $ G_{\C}/B_{\C} $.  
Then there are only finitely many $ K_{\C} $-orbits on 
%%the full flag variety 
$ G_{\C}/B_{\C} $ and there is a classification of $ K_{\C} $-orbits using Weyl group 
(see \cite{Richardson.Springer.1990,Richardson.Springer.1993}).  We call the classification as \emph{the KGB classification}.  
Then \emph{Matsuki correspondence} (also called \emph{Matsuki duality}) claims that 
there is a bijective correspondence between 
$ K_{\C} $-orbits and $ G_{\R} $-orbits in $ X = G_{\C}/B_{\C} $, 
\begin{equation*}
    K_{\C} \backslash G_{\C}/B_{\C} \xleftrightarrow[]{\;\;\sim\;\;}
    G_{\R} \backslash G_{\C}/B_{\C} , 
\end{equation*}
which reverses the closure ordering (\cite{Matsuki.Mduality.1988}; 
also see \cite[Theorem 1.2]{Mirkovic.Uzawa.Vilonen.1992}). 
The correspondence is defined as follows.  
Consider a $ G_{\R} $-orbit $ \calorbit \subset X $ and 
a $ K_{\C} $-orbit $ \bborbit \subset X $.  
They are in correspondence if  
$ \calorbit \cap \bborbit $ is compact (and, in this case, it turns out that the intersection is a single $ K_{\R} $-orbit).

The Matsuki correspondence holds for a connected real semisimple group $ G $ in general.  
Let us explain it briefly.  

Let $ P \subset G $ be a minimal parabolic subgroup, and 
$ X = G/P $ a flag variety.  
Let $ \sigma $ be an involution of $ G $ and 
take a Cartan involution $ \theta $ commuting with $ \sigma $.  
We denote $ K = G^{\theta} $, a maximal compact subgroup of $ G $.  

Put $ \tau = \sigma \theta $.  
Then $ \tau $ is another involution which commutes with $ \sigma $ and $ \theta $.  
Let us denote $ H_{\sigma} $ a symmetric subgroup of $ \sigma $, i.e., it satisfies 
$ G^{\sigma}_0 \subset H_{\sigma} \subset G^{\sigma} $.  
Similarly we take $ H_{\tau} $ for $ \tau $, and we assume that 
$ H_{\sigma} \cap K = H_{\tau} \cap K $.  Let us denote the common intersection 
by
\begin{equation*}
U = H_{\sigma} \cap K = H_{\tau} \cap K = H_{\sigma} \cap H_{\tau} \cap K ,
\end{equation*}
which is a subgroup of $ K $.

For any $ x \in X $, the stabilizer of $ x $ in $ G $ is also a minimal parabolic subgroup, and we denote it by $ P_x $.   
Put $ \mathscr{C} = \{ x \in X \mid P_x^{-\theta} \text{ is $ \sigma $-stable} \} $, 
where $ P_x^{-\theta} = \{ g \in P_x \mid \theta(g) = g^{-1} \} $ as usual.  

According to \cite[Th.1.2]{Mirkovic.Uzawa.Vilonen.1992}, 
there exists a $ U $-invariant flow $ \varphi: \R \times X \to X $ on $ X $ satisfying 
the conditions given in Theorem 1.2 of \cite{Mirkovic.Uzawa.Vilonen.1992}.  
The set $ \mathscr{C} $ is the set of critical points of this flow 
(see also \cite[Th.3(2)]{Matsuki.1991}), or 
\begin{equation*}
    \mathscr{C} = \{ x \in X \mid \varphi(t, x) = x \;\; (\forall t \in \R) \} .
\end{equation*}
Note that $ \mathscr{C} $ is $ U $-invariant.  

\begin{theorem}[Matsuki \cite{Matsuki.Mduality.1988}]\label{Thm-Matsuki-Mduality-1998}
There is a unique bijective correspondence between $ H_{\sigma} $-orbits in $ X $ and 
$ H_{\tau} $-orbits in $ X $, which satisfies the following.  

\begin{itemize}
\item[\MIP] %%Matsuki duality intersection property 
$ \calorbit \in H_{\sigma} \backslash X $ and $ \bborbit \in H_{\tau} \backslash X $ are in correspondence if and only if 
$ \calorbit \cap \bborbit $ is nonempty and closed (hence compact).
\end{itemize}

The bijection is called as the Matsuki correspondence.  
The Matsuki correspondence further satisfies the following properties.
\begin{penumerate}
    \item The bijection of the orbits reverses the closure relations.
    \item If $ \calorbit $ and $ \bborbit $ are in correspondence as above, then $ \calorbit \cap \bborbit $ is a single $ U $-orbit contained in $ \mathscr{C} $.
    \item Three subsets $ H_{\sigma} \backslash X $, $ H_{\tau} \backslash X $, and $ \mathscr{C}/U $ are in bijection.
\end{penumerate}
\end{theorem} 

The theorem above is proved in \cite{Matsuki.Mduality.1988} and \cite{Mirkovic.Uzawa.Vilonen.1992}, 
but the claim of the present theorem is slightly different from the original ones.  
If we prove the implication ``$ \calorbit \cap \bborbit $ is nonempty and closed $\Rightarrow$ $ \calorbit \cap \bborbit \subset \mathscr{C}$'', then the other statements are clear from 
Theorem in \cite{Matsuki.Mduality.1988} and Theorem 1.2 in \cite{Mirkovic.Uzawa.Vilonen.1992}.  
Thus let us give a proof of the implication above in the following lemma for the readers' convenience.

\begin{lemma}
    Let $ \calorbit \in H_{\sigma} \backslash X $ and $ \bborbit \in H_{\tau} \backslash X $. (We do NOT assume that they are in correspondence.) If the intersection $\calorbit \cap \bborbit$ is nonempty and closed, then we have $\calorbit \cap \bborbit \subset \mathscr{C}$. Moreover, this subset $\calorbit \cap \bborbit \subset \mathscr{C}$ is a single $U$-orbit.
\end{lemma}

\begin{proof}
%%    By \cite[Thm.~1.2]{Mirkovic.Uzawa.Vilonen.1992}, there exists 
    Consider the $U$-invariant flow $\varphi\colon \R\times X \to X$ and the critical set 
    \begin{equation*}
        \mathscr{C} = \{ x \in X \mid \forall t\in \R, \: \varphi(t,x)=x\}
    \end{equation*}
    above. 
    Take $x \in \calorbit \cap \bborbit$.
    By \cite[Thm.~1.2 (1)]{Mirkovic.Uzawa.Vilonen.1992}, $\varphi\colon \R \times X \to X$ preserves the orbits of $H_{\sigma}$ and $H_{\tau}$.
    Thus, $x \in \calorbit \cap \bborbit$ implies
    \begin{equation*}
        \forall t\in \R, \quad \varphi(t,x) \in  \calorbit \cap \bborbit.
    \end{equation*}
    We note that the limit $\lim_{t \to \infty} \varphi(t,x)$ exists by \cite[Thm.~1.2 (2)]{Mirkovic.Uzawa.Vilonen.1992},  
and from the assumption that $\calorbit \cap \bborbit$ is closed we conclude  
    \begin{equation*}
        \lim_{t \to \infty} \varphi(t,x) \in \calorbit \cap \bborbit.
    \end{equation*}
    The limit $ x_{\infty} := \lim_{t \to \infty}\varphi(t,x)$ belongs to 
    $\mathscr{C} $, since 
    \begin{equation*}
        \varphi(s, x_{\infty}) = \lim_{t \to \infty} \varphi(s,\varphi(t,x)) 
        = \lim_{t \to \infty} \varphi(s + t,x) =  x_{\infty} .
        %%\lim_{t \to \infty} \varphi(t,x).
    \end{equation*}
    Thus, we have $x_{\infty} \in \mathscr{C} \cap (\calorbit \cap \bborbit)$,
    which implies that $U \cdot x_{\infty} \subset \mathscr{C}$ because $\mathscr{C}$ is $U$-stable.
    Therefore, it is sufficient to prove that $U \cdot x_{\infty} = \calorbit \cap \bborbit $, namely $\calorbit \cap \bborbit $ 
is a single $U$-orbit.
For this, we note that $\calorbit $ corresponds to $\bborbit $ under the correspondence of \cite[Thm.~1.3 (2)]{Mirkovic.Uzawa.Vilonen.1992}, 
since $H_{\sigma} \cdot x_{\infty} = \calorbit$ and $H_{\tau} \cdot x_{\infty} = \bborbit$.   
    Thus, by \cite[Thm.~1.3 (3)]{Mirkovic.Uzawa.Vilonen.1992}, we conclude the intersection $\calorbit \cap \bborbit $ is a single $U$-orbit.
    % the map
    % \begin{equation}\label{eq-bijection-UC=HX}
    %     \begin{array}{ccc}
    %          U \backslash \mathscr{C}& \to & H_{\sigma}  \backslash X  \\
    %          \rotatebox{90}{$\in$} && \rotatebox{90}{$\in$} \\
    %          U\cdot c & \mapsto & H_{\sigma} \cdot c
    %     \end{array}
    % \end{equation}
    % is bijective. Note that $H_{\sigma}\cdot x_{\infty} \subset \calorbit $ implies $H_{\sigma}\cdot x_{\infty} = \calorbit $ because $\calorbit $ is an $H_{\sigma}$-orbit.
    % By \cite[Thm.~1.3 (3)]{Mirkovic.Uzawa.Vilonen.1992}
    % Thus, the bijection \eqref{eq-bijection-UC=HX} makes 
    % $\mathscr{O}$ to be a single $ U $-orbit........?
\end{proof}

\subsection{Orbit decomposition in the double flag varieties and the Matsuki-Oshima's clans}

In this subsection, we discuss the relations between the orbit decomposition in the double flag varieties and Matsuki duality, KGB decompositions and Matsuki-Oshima's notion of clans.

We begin with Case \caseA.

\begin{theorem}[{\cite{Matsuki.1979} and \cite[\S4]{Matsuki.Oshima.1990}}]\label{thm-Matsuki-duality}\label{Thm-Matsuki-Oshima-classification}
There exist bijections:
\begin{equation*}
    B_{p+q} \backslash \GL_{p+q}(\C)/ \U(p,q) 
    \simeq
    B_{p+q} \backslash \GL_{p+q}(\C)  /  (\GL_{p}(\C) \times \GL_{q}(\C))
    \simeq \Gamma(p,q),
\end{equation*}
where $\Gamma(p,q)$ is the set of $(p,q)$-clans (Definition \ref{Def-Gamma(p,q)}).
\end{theorem}

\begin{proof}
    The first isomorphism is a consequence of the Matsuki correspondence.  
If we put $ G = G_{\R} = \U(p,q) $, a maximal compact subgroup of $ G_{\R} $ is $ K_{\R} = \U(p) \times \U(q) $ and its complexification is $ K_{\C} = \GL_{p}(\C) \times \GL_{q}(\C) $.  
So the situation is the same as that explained in \S\ref{subsec:KGB.Matsuki.duakity}  (KGB decomposition).

The second isomorphism is a consequence of the Matsuki-Oshima's classification of the orbits by clans.  
In their notation, the present case is denoted as AI, and there are symbols $ \pm, \circ $, called boys/girls and adults.  The orbits are classified by the sequences of symbols, where adults will appear in pair.  
If we draw an arc for such a pair of adults ``$ \circ $'', we get our $ (p,q) $-clans.
\end{proof}

Let us consider Case \caseB.  
Matsuki-Oshima's classification is available only for KGB decomposition, so there is a slight difference from Case \caseA.

\begin{theorem}\label{Thm-Matsuki-Oshima-classification:CaseB}
There exist bijections:
\begin{equation*}
    B_{p+q}(\R) \backslash \GL_{p+q}(\R)/ \OO(p,q) 
    \simeq
    B_{p+q}(\R) \backslash \GL_{p+q}(\R)  /  (\GL_{p}(\R) \times \GL_{q}(\R))
    \simeq \Gamma(p,q),
\end{equation*}
where $\Gamma(p,q)$ is the set of $(p,q)$-clans (Definition \ref{Def-Gamma(p,q)}).
\end{theorem}

\begin{proof}
The first isomorphism is a consequence of the Matsuki duality as in Theorem \ref{Thm-Matsuki-Oshima-classification}.

Let us consider the second isomorphism.  
%%\label{Rem-MD-isom-explicit-formula}
In \cite[Thm.~2.2.14]{Yamamoto.Atsuko.1997}, 
Yamamoto explicitly gives representatives of the left-hand side of Theorem \ref{Thm-Matsuki-Oshima-classification} in matrix form, which are constructed by using the diagrams of clans in $\Gamma(p,q)$.  
If we denote one of her representatives by $ g $, 
then $ g I_{p,q} g^* $ matches exactly with a signed partial permutation $ \tau \in \SPIst_m \; (m = p + q) $.  
   Moreover, one can verify that such representatives are ``special points'' in $ \mathscr{C} \subset B_{p+q} \backslash \GL_{p+q}(\C)$ in the sense of \cite[Def.~in 1p.]{Matsuki.1991} (see also \S \ref{subsec:KGB.Matsuki.duakity}).
   For such special points, the bijection of Theorem \ref{thm-Matsuki-duality} is given by the ``identity morphism'', namely, the bijection of Theorem \ref{thm-Matsuki-duality} is given by
\begin{equation*}
\begin{array}{ccc}
    B_{p+q} \backslash \GL_{p+q}(\C)/ \U(p,q) &\xrightarrow{\sim}& B_{p+q} \backslash \GL_{p+q}(\C)  /  (\GL_{p}(\C) \times \GL_{q}(\C)) \\
    \rotatebox{90}{$\in$} && \rotatebox{90}{$\in$} \\
    B_{p+q} g \U(p,q) & \mapsto & B_{p+q} g (\GL_{p}(\C) \times \GL_{q}(\C)),
\end{array}	
\end{equation*}
    if $B_{p+q} g \in B_{p+q} \backslash \GL_{p+q}(\C)$ is special in the sense of \cite[Def.~in 1p.]{Matsuki.1991}.

The exactly same representatives of Yamamoto are also special points in the real flag variety 
$ B_{p+q}(\R) \backslash \GL_{p+q}(\R) $ in Case \caseB and they are corresponding to $ \Gamma(p,q) $.  
If we take Remark \ref{remark:Her-pqr.Bm.isom.SPIst-pqr} into account, 
we can conclude the same representatives from $ \Gamma(p,q) $ classify the orbits 
$ B_{p+q}(\R) \backslash \GL_{p+q}(\R)  /  (\GL_{p}(\R) \times \GL_{q}(\R)) $.
\end{proof}

% The classification of the right-hand side $B_{p+q} \backslash \GL_{p+q}(\C) / ( \GL_{p}(\C) \times \GL_{q}(\C) )$ of this Theorem \ref{thm-Matsuki-duality} is given by Matsuki-Oshima's clan.

% \begin{theorem}[{\cite[\S4]{Matsuki.Oshima.1990}}]\label{Thm-Matsuki-Oshima-classification}
% We have
% \begin{equation*}
%     B_{p+q} \backslash \GL_{p+q}(\C) / ( \GL_{p}(\C) \times \GL_{q}(\C) )
%     \simeq 
%     \Gamma(p,q),
% \end{equation*}
% where $\Gamma(p,q)$ is the set of $(p,q)$-clans (Definition \ref{Def-Gamma(p,q)}).
% \end{theorem}

By Theorem \ref{thm-Matsuki-duality}, we can give a combinatorial description of the orbit decomposition $B_{m}\backslash \Her_{m}(\C)$.

\begin{corollary}\label{Cor-DVR=gamma(p,q)}
We have an isomorphism
\begin{equation*}
	B_{m}\backslash \Her_{m}(\C)
	\simeq 
    B_{m}(\R) \backslash \Sym_{m}(\R)
	\simeq 
    \coprod_{0 \leq p + q \leq m}
	\coprod_{
    J \in \binom{[m]}{p + q}
%%	   \substack{ J\subset \{1,2,\dots,m\}\\ 0\leq \# J=p+q \leq m}
	       }
	\Gamma(p,q).
\end{equation*}
\end{corollary}

\begin{proof}
The statement follows from Lemma \ref{Lem-B-Her=BGLU} and 
Theorem \ref{Thm-Matsuki-Oshima-classification}.
\end{proof}

\subsection{Combinatorial descriptions of orbits by graphs}

The right-hand side of Corollary \ref{Cor-DVR=gamma(p,q)} has another interpretation.

% \begin{definition}\label{Def-Gamma'(m)}
% Let $m\in\Z_{\geq 0}$.
% Then, we define a set $\Gamma'(m)$ consisting of $m$ vertices, which are decorated by $+,-,c$, or arcs, according to the following rules.
% \begin{enumerate}
% 	\item[$(*)$] The two end points of each arc are different.
% \end{enumerate}
% More rigorously, we set
% \begin{equation*}
% 	\Gamma'(m)
% 	\coloneqq
% 	\left\{
%         \gamma' \colon \{1,2,\dots,m\} \to \{1,2,\dots,m\} \cup \{+,-,c\}
% 	\Bigm| 
% 		  \text{$\gamma'$ satisfies $(*')$}
% 	\right\},
% \end{equation*}
% where
% \begin{enumerate}
% 	\item[$(*')$]
% 	if $\gamma'(i)=j$,
% 	then we have $i\neq j$ and $\gamma'(j)=i$,
% \end{enumerate}
% % We call an element of $\Gamma(p,q)$ a $(p,q)$-\textit{clan}.
% \end{definition}

% \begin{example}\label{Ex-Gamma'(m)}
% Let $\gamma'\colon \{1,2,\dots,6\} \to \{1,2,\dots,6\} \cup \{+,-,c\}$ be a function defined by
% \begin{equation*}
%         \gamma'(1) = -,\quad
%         \gamma'(2) = c,\quad
%         \gamma'(3) = 5,\quad
%         \gamma'(4) = +,\quad
%         \gamma'(5) = 3,\quad
%         \gamma'(6) = -.
% \end{equation*}
% Then we have $\gamma' \in \Gamma'(6)$.
% Its graphical representation is given by
% \begin{equation*}
%     \gamma'
%     =
%     \xymatrix{
%         1_{-} &
%         2_{c} &
%         3 \ar@{-}@/^10pt/[rr] &
%         4_{+} &
%         5     &
%         6_{-}.
%     }
% \end{equation*}
% \end{example}

\begin{lemma}\label{Lem-coprod-Gamma(p,q)=Gamma'(m)}
    For $m\in \Z_{\geq 0}$, we have
    \begin{equation}\label{eq-isom-coprod-Gamma(p,q)-to-Gamma'(m)}
    \coprod_{0 \leq p + q \leq m}
	\coprod_{
    J \in \binom{[m]}{p + q}
%%	   \substack{ J\subset \{1,2,\dots,m\}\\ 0\leq \# J=p+q \leq m}
	       }
	   \Gamma(p,q) 
       \simeq
       \Gamma'([m]),
    \end{equation}
where $\Gamma'([m])$ is defined in \eqref{eq-Def-Gamma'(I)}.
\end{lemma}

\begin{proof}
Let $\gamma \in \Gamma(p,q)$ be an element of the left-hand side of \eqref{eq-isom-coprod-Gamma(p,q)-to-Gamma'(m)} corresponding to the summand with parameter $J\subset \{1,2,\dots,m\}$.
We write $\phi \colon \{1,2,\dots,\#J\} \xrightarrow{\sim} J$ for the unique order-preserving isomorphism.
Then, we define $\gamma' \in \Gamma'([m])$ by (see Example \ref{Ex-Gamma(p,q)-to-Gamma'(m)})
\begin{equation}
    \gamma'(i)
    \coloneqq
    \begin{cases}
        \gamma(j) & \text{if $\phi(j)=i$} \\
        c         & \text{if $i \not\in J$}.
    \end{cases}
\label{eq-correspondence-from-Gamma(p,q)-to-Gamma'(m)}
\end{equation}
The map $\gamma \mapsto \gamma'$ gives the desired isomorphism.
\end{proof}

\begin{example}\label{Ex-Gamma(p,q)-to-Gamma'(m)}
    Let $\gamma \in \Gamma(3,2)$ be the same as in Example \ref{Ex-Gamma(p,q)} and 
    \begin{equation*}
        J\coloneqq \{1,3,4,5,6\} = [6] \backslash \{2\} \subset [6] \coloneqq \{1,2,3,4,5,6\}.  
    \end{equation*}
    Then, in this case, the unique order-preserving isomorphism $\phi \colon \{1,2,\dots,\# J\} \xrightarrow{\sim} J$ in the proof of Lemma \ref{Lem-coprod-Gamma(p,q)=Gamma'(m)} is given by
     \begin{equation*}
         \phi(1) = 1,\quad 
         \phi(2) = 3,\quad 
         \phi(3) = 4,\quad 
         \phi(4) = 5,\quad 
         \phi(5) = 6.
     \end{equation*}
     Let $\gamma' \in \Gamma([6])$ correspond to this $\gamma$ 
     (regarded as an element of the left-hand side of \eqref{eq-isom-coprod-Gamma(p,q)-to-Gamma'(m)} corresponding to the summand with parameter $J$) in via \eqref{eq-correspondence-from-Gamma(p,q)-to-Gamma'(m)}.
     Then, the graphical representations of $\gamma$ and $\gamma'$ are given by
     \begin{equation*}
     \begin{array}{c}
        \xymatrix{
            \gamma = &
            1_{-} &
            \: &
            2 \ar@{-}@/^10pt/[rr] &
            3_{+} &
            4     &
            5_{-}, \\
            \gamma' = &
            1_{-} &
            2_{c} &
            3 \ar@{-}@/^10pt/[rr] &
            4_{+} &
            5     &
            6_{-}.
        }
     \end{array}
     \end{equation*}    
\end{example}

Moreover, this $\Gamma'([m])$ gives a combinatorial interpretation of $\SPIst_{m}$.
 
\begin{lemma}\label{Lem-Gamma'(m)=SPI^st}
    For $m\in\Z_{\geq 0}$, we have
    \begin{equation*}
        \Gamma'([m]) \simeq \SPIst_{m},
    \end{equation*}
    where $\SPIst_{m}$ is defined in Definition \ref{Def-SPI-and-SPIst}.
\end{lemma}

\begin{proof}
    Let $\gamma'\in\Gamma'([m])$.
    Then, we define $\tau \in \SPIst_{m}$ by (See Example \ref{Ex-Gamma'(M)=SPIst})
    \begin{equation*}
        \tau_{ij}
        =
        \begin{cases}
            1 & \text{if $\gamma'(i)=j$} \\
            1 & \text{if $\gamma'(i)=+$ and $i=j$}\\
            -1& \text{if $\gamma'(i)=-$ and $i=j$}\\
            0 & \text{otherwise},
        \end{cases}
    \end{equation*}
    where $\tau_{ij}$ is the $(i,j)$-th entry of $\tau \in \SPIst_{m}$.
    The map $\gamma' \mapsto \tau$ is the desired isomorphism.
\end{proof}

\begin{example}\label{Ex-Gamma'(M)=SPIst}
    Let $\gamma'\in\Gamma'([6])$ be the same as Example \ref{Ex-Gamma(p,q)-to-Gamma'(m)}.
    Then, the $\tau \in \SPIst_{m}$ constructed in the proof of Lemma \ref{Lem-Gamma'(m)=SPI^st} is given by
    \begin{equation*}
        \tau =
        \begin{pmatrix}
            -1 & 0 & 0 & 0 & 0 & 0 \\
            0  & 0 & 0 & 0 & 0 & 0 \\
            0  & 0 & 0 & 0 & 1 & 0 \\
            0  & 0 & 0 & 1 & 0 & 0 \\
            0  & 0 & 1 & 0 & 0 & 0 \\
            0  & 0 & 0 & 0 & 0 & -1 \\
        \end{pmatrix}.
    \end{equation*}
\end{example}

\section{Closure relations}\label{section:proof.closure.relations}

In this section, we prove Theorem \ref{Thm-closure-relation}.  
As we already claimed in Introduction, the closure relation of $ H $-orbits in $ \dblFV $ and that of $ B_H $-orbits in $ G/P_G $ are the same (cf.~\cite[Lemma 3.1]{Fresse.N.2023}).  So we will prove the theorem for $B_{H} \backslash G / P_{G}$.

\begin{proof}[Proof of Theorem \ref{Thm-closure-relation}]
Let $\gamma, \gamma ' \in \Gamma (n)$ (Definition \ref{Def-Gamma(n)}) and $\mathcal{O}_{\gamma}, \mathcal{O}_{\gamma'} \in H \backslash \dblFV$ corresponding orbits under the isomorphism $\Gamma(n) \simeq H \backslash \dblFV $ of Theorem \ref{theorem-H-G-P_G=Gamma(n)}, and we write $\overline{\mathcal{O}}_{\gamma}$ for the closure of $\mathcal{O}_{\gamma}$ in $\dblFV$.
We have to show
$\overline{\mathcal{O}}_{\gamma} \supset \mathcal{O}_{\gamma'}$ if one of the conditions in Theorem \ref{Thm-closure-relation} holds.
The proof is almost the same for all the cases. Namely, we construct a one-parameter family $b(t)\colon \R^{\times} \to B_{H}$, and show that the limit $\lim_{t\to 0} b(t) \cdot z_{\gamma}$ is contained in $\mathcal{O}_{\gamma '}$, where $z_{\gamma}$ is an element of $\mathcal{O}_{\gamma}$. Therefore, we will only prove the case (v)$_{c}$.
Moreover, we only treat the case $n=3$ because the general case can be reduced to this case by considering appropriate submatrices.
Thus, we have to prove that if 
\begin{equation*}
    \gamma = 
    \xymatrix@R=5pt@C=5pt{
    1_{c}&2\ar@{-}@/^10pt/[r] & 3 \\
    },  \qquad
    \gamma' =
    \xymatrix@R=5pt@C=5pt{
    1_{\varepsilon}&2_{-\varepsilon} & 3_{c}, \\
    }   \qquad (\varepsilon \in \{-,+\}), 
\end{equation*}
then $\overline{\mathcal{O}}_{\gamma} \supset \mathcal{O}_{\gamma '}$ holds.
Recall that $\gamma, \gamma' \in \Gamma(3)$ produce representatives of the $B_{H}$-orbits in $ G / P_{G}$ 
in the following way.  
\begin{equation*}
\gamma \leftrightarrow    
\begin{pmatrix}
        0 & 0 & 0 \\ 0 & 0 & 1 \\ 0 & 1 & 0 \\ \hline 1 & 0 & 0 \\ 0 & 1 & 0 \\ 0& 0 & 1
    \end{pmatrix},
    \quad
\gamma' \leftrightarrow    
    \begin{pmatrix}
        \varepsilon & 0 & 0 \\ 0 & -\varepsilon & 0 \\ 0 & 0 & 0 \\ \hline 1 & 0 & 0 \\ 0 & 1 & 0 \\ 0& 0 & 1
    \end{pmatrix}  \quad\in \LregMat_{6,3},
\end{equation*}
where $\LregMat_{6,3}$ is defined in Definition \ref{Def-LM-circ}.
Therefore it is sufficient to construct a one-parameter family $b(t)\colon \R^{\times} \to B_{H}$ satisfying
\begin{equation*}
    b(t) 
    \begin{pmatrix}
        0 & 0 & 0 \\ 0 & 0 & 1 \\ 0 & 1 & 0 
    \end{pmatrix}
    b(t)^{*}
    \xrightarrow{t \to 0 }
    \begin{pmatrix}
        \varepsilon & 0 & 0 \\ 0 & -\varepsilon & 0 \\ 0 & 0 & 0 
    \end{pmatrix}.
\end{equation*}
For this, we take 
\begin{equation*}
b(t) =         
\begin{pmatrix}
                1 & 1 & \varepsilon/2 \\
                0 & 1 & -\varepsilon/2 \\
                0 & 0 & t
\end{pmatrix}
\end{equation*}
and get
\begin{equation*}
b(t) 
	\begin{pmatrix}
		0 & 0 & 0 \\
		0 & 0 & 1 \\
		0 & 1 & 0
	\end{pmatrix}
b(t)^*
	=
	\begin{pmatrix}
		\varepsilon & 0 & t \\
		0 & -\varepsilon & t \\
		t & t & 0
	\end{pmatrix}
	\xrightarrow{t\to 0}
	\begin{pmatrix}
		\varepsilon & 0 & 0 \\
		0 & -\varepsilon & 0 \\
		0 & 0 & 0
	\end{pmatrix},
\end{equation*}
which completes the proof.
\end{proof}

\appendix

\section{Orbit decomposition of a fiber bundle}\label{Section-proof-of-Lem-orbit-decomp-fiber-bundle}

In this appendix, we give a proof of Lemma \ref{Lem-orbit-decomp-fiber-bundle} for the sake of completeness. More precisely, we prove Lemma \ref{lemma:2023-11-15} below, which is completely general.
We do \emph{not} require any additional structure on groups 
$G,P,Q, B$, etc.
appearing in Appendix \ref{Section-proof-of-Lem-orbit-decomp-fiber-bundle}.
%%In other words, all $G,P,Q$ and $B$ in two lemmas below are just groups (which are \emph{not} assumed to be neither algebraic nor Lie groups).

\begin{lemma}\label{Lem-B-eq-X-Y}
    Let $B$ be a group, $f \colon X \to Y$ a $B$-equivariant map between two $B$-sets $X,Y$, and $S \subset Y$ be a representative of the orbit decomposition $B \backslash Y$:
    \begin{equation}\label{eq-BY=Bs}
        B \backslash Y \simeq \coprod_{s \in S} B \cdot s.
    \end{equation}
    Then, we have a bijection
\begin{equation}\label{eq-BsX=BX}
\begin{array}{ccc}
\displaystyle
        \coprod_{s\in S} B_{s}\backslash f^{-1}(s)
        & \xrightarrow{\sim} &
        B\backslash X
        \\
        \rotatebox{90}{$\in$} && \rotatebox{90}{$\in$}
        \\
        \calorbit
        & \mapsto &
        B \calorbit.
\end{array}
\end{equation}
\end{lemma}

\begin{proof}
    First, we prove the injectivity of the map \eqref{eq-BsX=BX}.
    Let $x, x' \in f^{-1}(s)$ with $B \cdot x = B \cdot x'$. 
    Then, there exists $b \in B$ such that $bx = x'$. Applying the $B$-equivariant map $f \colon X \to Y$ to this equality, we have $bs=s$ (because $x, x' \in f^{-1}(s)$), which implies $b \in B_{s}$. Therefore, $B_{s} \cdot x = B_{s} \cdot x'$, which 
    implies the injectivity of the map \eqref{eq-BsX=BX}.

    Next, let us prove the surjectivity.   
    %%of the map \eqref{eq-BsX=BX}.
    Pick any $x\in X$. Then, 
    %%$f(x) \in Y$. Thus, 
    there exist $s\in S$ and $b\in B$ such that $ s = b f(x) = f(b x) $ by 
    %%\eqref{eq-BY=Bs} and 
    the $B$-equivariance of $f\colon X \to Y$. Thus, $ b x \in f^{-1}(s)$, hence $B_{s} (b x) \in B_s \backslash f^{-1}(s) $ maps to $ B B_s (b x) = B \cdot x $ 
    by the map of \eqref{eq-BsX=BX}, and we have done.
\end{proof}

\begin{lemma}\label{lemma:2023-11-15}
Let $G$ be a group, $Q\subset P$ subgroups of $G$,
\begin{equation*}
    \pi \colon G/Q \to G/P
\end{equation*}
a natural projection.
Moreover, let $B$ be a subgroup of $G$, and $S\subset G$ a set of representatives of the orbit decomposition $B\backslash G/P$:
\begin{equation}
    B\backslash G/P
    \simeq 
    \coprod_{s\in S} B s P.
\label{eq-Lem-orbit-decomp-base}
\end{equation}
Then, we have isomorphisms
\begin{equation}
\begin{array}{ccccc}
\displaystyle
        \coprod_{s \in S} (P\cap B^{s^{-1}} )\backslash P/Q        
        & \xrightarrow{\sim} &
\displaystyle
        \coprod_{s\in S} B_{s P}\backslash sP/Q
        & \xrightarrow{\sim} &
        B\backslash G/ Q
        \\
        \rotatebox{90}{$\in$} && \rotatebox{90}{$\in$} && \rotatebox{90}{$\in$}
        \\
        (P\cap B^{s^{-1}} ) p Q
        & \mapsto &
        B_{s P} s p Q
        & \mapsto &
        B s p Q
\end{array}
\label{eq-BGU=B_gopi^-1(goP)-11-15}
\end{equation}
where $B_{s P}$ is the isotropy subgroup of $B$ at $s P \in G/P$ and $B^{s^{-1}} \coloneqq s^{-1}Bs$.
\end{lemma}

\begin{proof}
The second isomorphism of \eqref{eq-BGU=B_gopi^-1(goP)-11-15} is a consequence of Lemma \ref{Lem-B-eq-X-Y} (by taking $\pi\colon G/Q \to G/P$ as $f\colon X\to Y$).

For the first isomorphism of \eqref{eq-BGU=B_gopi^-1(goP)-11-15}, we have a commutative diagram:
% \begin{equation*}
%     \xymatrix{
%     B_{sP} 
%     \ar[d]_-{b \mapsto s^{-1}bs}^-{\wr}
%     & \times &  
%     sP/Q 
%     \ar[d]^-{spQ \mapsto p Q}_-{\wr}
%     \ar[rr]^-{\text{action}} 
%     &
%     & 
%     sP/Q 
%     \ar[d]^-{spQ \mapsto p Q}_-{\wr}
%     \\
%     P \cap B^{s^{-1}} 
%     & \times & 
%     P / Q 
%     \ar[rr]^-{\text{action}} 
%     &
%     & 
%     P / Q .
%     }
% \end{equation*}
\begin{equation*}
    \xymatrix{
    (P \cap B^{s^{-1}} )
    \times  
    P / Q  
    % \ar[d]_-{b \mapsto s^{-1}bs}^-{\wr}
    \ar[rrr]^-{(b,pQ )\mapsto (sbs^{-1}, s p Q)}_-{\sim}
    \ar[d]^-{\text{action}} 
    & & &
    B_{sP} 
    \times  
    sP/Q
    \ar[d]^-{\text{action}}
    \\
    P / Q
    \ar[rrr]^-{pQ \mapsto s p Q}_-{\sim}
    &
    & 
    &
    s P / Q,
    }
\end{equation*}
which gives the desired isomorphism.
% consider two isomorphism 
% \begin{equation*}
% \begin{array}{ccc}
%     \chi\colon
%     B_{s P} = s P s^{-1} \cap B & \xrightarrow{\sim} & P  \cap  s^{-1} B s = P \cap B^{s^{-1}}
%     \\
%     \rotatebox{90}{$\in$} && \rotatebox{90}{$\in$}
%     \\
%     b & \mapsto & s^{-1} b s,
% \end{array}
% \end{equation*}
% and
% \begin{equation*}
% \begin{array}{ccc}
%     \eta \colon
%     s P/Q = \pi^{-1}(s P) & \xrightarrow{\sim} & P/Q
%     \\
%     \rotatebox{90}{$\in$} && \rotatebox{90}{$\in$}
%     \\
%     s p Q & \mapsto & p Q,
% \end{array}   
% \end{equation*}
% which is a $\chi$-morphism, in other words, we have
% \begin{equation*}
%     \eta(b s p Q)
%     =
%     \eta(s (s^{-1} b s) p Q)
%     =
%     (s^{-1} b s) p Q
%     =
%     \chi(b) \eta(s p Q),
% \end{equation*}
% where $b\in B_{s P}$ and $s p Q \in s P/Q$.
% These two morphism give the desired isomoprihsm.
\end{proof}

\section{Real points of double coset decomposition}\label{Section-real-points-of-double-coset-decomposition}

In this section, we give a proof of Proposition \ref{fact-real-orbits-via-Galois-cohomology}. More precisely, we prove 
Proposition \ref{Prop-real-orbits-via-Galois-cohomology-in-Appx} below in more general setting, from which Proposition \ref{fact-real-orbits-via-Galois-cohomology} immediately follows.  
%%(see also the footnote to the proof of Proposition \ref{Prop-real-orbits-via-Galois-cohomology-in-Appx}).
Moreover, we also discuss the properties of real points of double coset decomposition.

\begin{proposition}\label{Prop-real-orbits-via-Galois-cohomology-in-Appx}
Let $G$ be a real algebraic group, $H$ and $L$ its real algebraic subgroups.
Write $G_{\C}, H_{\C}$, and $L_{\C}$ for complexifications $G,H$ and $L$, respectively, and write $X_{\C} \coloneqq G_{\C} / L_{\C}$.
Let $\Xi\subset H_{\C}\backslash X_{\C} $ be a subset of $H_{\C}$-orbits on $X_{\C}$ having nonempty $\R$-rational points:
\begin{equation*}
	\Xi\coloneqq
	\{
        \mathcal{O} \in H_{\C} \backslash X_{\C}
    \mid
	   \mathcal{O}(\R)\neq \emptyset
    \}.
\end{equation*}
Suppose that the first Galois cohomologies $H^{1}(\R, H_{\C})$ and $H^{1}(\R, L_{\C})$ are trivial, i.e., 
\begin{equation*}
    H^{1}(\R, H_{\C}) = 1, \quad H^{1}(\R, L_{\C}) = 1.
\end{equation*}
Then, we have a bijection
\begin{equation*}
	H \backslash G / L 
	\simeq 
	\coprod_{\mathcal{O}\in\Xi}
	H^{1}(\R, (H_{\C})_{x_{\mathcal{O}}}),
% \label{eq-fact-real-orbits-via-Galois-cohomology-Appx}
\end{equation*}
where $(H_{\C})_{x_{\mathcal{O}}}$ denotes the stabilizer in $H_{\C}$ at $x_{\mathcal{O}} \in \mathcal{O}(\R)$.
\end{proposition}

\begin{proof}
    Applying \cite[Chap.~I, \S 5.4, Prop.~36]{Serre-Galois-Cohomology} to $(H_{\C})_{x_{\mathcal{O}}} \subset H_{\C}$, we have a long exact sequence (of pointed sets):
    \begin{equation*}
        1 \to H_{x_{\mathcal{O}}} \to H \to (H_{\C} / (H_{\C})_{x_{\mathcal{O}}})(\R)
        \to H^{1}(\R, (H_{\C})_{x_{\mathcal{O}}}) \to H^{1}(\R, H_{\C}).
    \end{equation*}
    Thus, the first assumption $H^{1}(\R, H_{\C}) = 1$ implies 
    $H \backslash \mathcal{O}(\R) \simeq 
    H \backslash (H_{\C} / (H_{\C})_{x_{\mathcal{O}}})(\R) \simeq
    H^{1}(\R, (H_{\C})_{x_{\mathcal{O}}})$ for any $\mathcal{O} \in \Xi$.
    Taking the coproduct over $ \mathcal{O} \in \Xi$, we have
    \begin{equation*}
        H \backslash (G_{\C} / L_{\C})(\R) 
	    \simeq 
	    \coprod_{\mathcal{O}\in\Xi}
	    H^{1}(\R, (H_{\C})_{x_{\mathcal{O}}}).
    % \label{eq-HGL(R)=H^1}
    \end{equation*}
    In the same way, 
    applying \cite[Chap.~I, \S 5.4, Prop.~36]{Serre-Galois-Cohomology} to $L_{\C} \subset G_{\C}$, we have a long exact sequence (of pointed sets):
    \begin{equation*}
        1 \to L \to G \to (G_{\C} / L_{\C})(\R)
        \to H^{1}(\R, L_{\C}) \to H^{1}(\R, G_{\C}).
    \end{equation*}
    Thus, the second assumption $H^{1}(\R, L_{\C}) = 1$ implies
$(G_{\C}/L_{C})(\R) = G/L$.  
Combining these two, we get 
the proposition.  
%%by \eqref{eq-HGL(R)=H^1}.
\end{proof}

\begin{remark}
To apply this proposition to get Proposition \ref{fact-real-orbits-via-Galois-cohomology}
we take $ L $ as the parabolic subgroup $ P $ of $ G $.  
Although the assumption $ H^1(\R, P_{\C}) = 1 $ is \emph{not} assumed in Proposition \ref{fact-real-orbits-via-Galois-cohomology}, we can elude the condition.  
%%Here we use the second condition $H^{1}(\R, L_{\C}) = 1$ to conclude $(G_{\C}/L_{\C})(\R) = G/L$.  However, 
Namely, 
if we assume $ P $ is a \emph{real parabolic subgroup} of a real reductive group $ G $, 
the conclusion $(G_{\C}/P_{\C})(\R) = G/P$ always holds 
(see \cite[Thm.~4.13]{Borel.Tits.1965} or \cite[Prop.~20.5]{Borel.1991}).
Thus we do not need the assumption $ H^{1}(\R, P_{\C}) = 1$ in Proposition \ref{fact-real-orbits-via-Galois-cohomology} (and in fact, this last condition may fail in general).  
This fact was pointed out to us by Takuma Hayashi.  We thank him for his comment.
\end{remark}

    Suppose that we are in the setting of Proposition \ref{Prop-real-orbits-via-Galois-cohomology-in-Appx}.  
We write $X_{\C}' \coloneqq  H_{\C} \backslash G_{\C}  $ and let 
    $\Xi' \subset  X_{\C}' /L_{\C} $ be a subset of $L_{\C}$-orbits on $X_{\C}'$ having nonempty $\R$-rational points:
\begin{equation*}
	\Xi' \coloneqq
	\{
        \mathcal{O}' \in  X_{\C}' / L_{\C}
    \mid
	   \mathcal{O}' (\R)\neq \emptyset
    \}.
\end{equation*}
We may and do regard $\Xi$ and $\Xi'$ as subsets of the double coset space $H_{\C} \backslash G_{\C} /L_{\C}$.  
It seems interesting to know whether $ \Xi = \Xi' $ holds or not.  

In fact, the equality $ \Xi = \Xi' $ is not always true.  
For example, let us consider the case $(G,H,L) = (\GL_{2}(\R), B_{2}(\R), \OO_{2}(\R))$, where $B_{2}(\R)$ is the subgroup of upper triangular matrices of $\GL_{2}(\R)$ and $\OO_{2}(\R)$ is the orthogonal group.  
Then, we have $(G_{\C},H_{\C},L_{\C}) = (\GL_{2}(\C), B_{2}(\C), \OO_{2}(\C))$, 
where $ H^{1}(\R, L_{\C}) =H^{1}(\R, \OO_{2}(\C)) \neq 1 $ (so that it does not satisfy our assumption).  In this case, we get 
\begin{equation*}
    \begin{split}
        X_{\C}' =B_{2}(\C) \backslash \GL_{2}(\C)\simeq \mathbb{P}^{1}(\C), \quad
        X_{\C} = \GL_{2}(\C) / \OO_{2}(\C) \simeq \Sym_{2}(\C) \cap \GL_{2}(\C),
    \end{split}
\end{equation*}
where $\mathbb{P}^{1}(\C)$ is the projective space and $\Sym_{2}(\C)$ is the space of symmetric matrices.
Thus, an easy computation shows
\begin{equation*}
\begin{split}
    X'_{\C} / \OO_{2}(\C) 
    &\simeq 
    \left( \C \backslash\{-\sqrt{-1}, \sqrt{-1}\}\right)
    \coprod \{-\sqrt{-1}, \sqrt{-1}\},
    \\
    B_{2}(\C) \backslash X_{\C} 
    &\simeq 
    B_{2}(\C) \cdot \begin{pmatrix}
        1 & 0 \\ 0 & 1
    \end{pmatrix} 
    \coprod 
    B_{2}(\C) \cdot \begin{pmatrix}
        0 & 1 \\ 1 & 0
    \end{pmatrix},
\end{split}
\end{equation*}
where we regard $-\sqrt{-1},\sqrt{-1} \in \C \subset \mathbb{P}^{1}(\C) $.
It is clear that 
the second element of $X'_{\C} / \OO_{2}(\C)$ has no $\R$-rational points 
and that the both of two elements of $B_{2}(\C) \backslash X_{\C}$ have an $\R$-rational point,
which shows that
\(
    \Xi \neq \Xi'
\).

However, under the assumption of Proposition \ref{Prop-real-orbits-via-Galois-cohomology-in-Appx}, $\Xi$ coincides with $\Xi'$.
More generally, the following holds.

\begin{lemma}
	Let $G$ be a real algebraic group, 
	$H$ and $L$ its real algebraic subgroups.
    Write $G_{\C}, H_{\C}$ and $L_{\C}$ for complexifications of $G,H$ and $L$, respectively.  Put 
	\begin{equation*}
		X_{\mathbb{C}}' \coloneqq H_{\mathbb{C}} \backslash G_{\mathbb{C}},
		\quad
		X_{\mathbb{C}} \coloneqq G_{\mathbb{C}} / L_{\mathbb{C}},
	\end{equation*} 
	and 
	\begin{equation*}
		\Xi' \coloneqq \{ \mathcal{O}' \in X_{\mathbb{C}}' / L_{\mathbb{C}}\mid \mathcal{O}'(\mathbb{R}) \neq \emptyset\},
		\qquad
		\Xi \coloneqq \{ \mathcal{O} \in H_{\mathbb{C}} \backslash X_{\mathbb{C}} \mid \mathcal{O}(\mathbb{R}) \neq \emptyset\}.
	\end{equation*}
	Suppose that 
	\begin{math}
		H^{1}(\mathbb{R}, H_{\mathbb{C}})=1.
    \end{math}
	Then, we have
	\begin{math}
		\Xi' \subset \Xi
	\end{math}.  
Thus, if moreover $ H^1(\R, L_{\C}) = 1 $ holds, 
we get the equality $ \Xi' = \Xi $.  
\end{lemma}

\begin{proof}
First note that the assumption $H^{1}(\mathbb{R}, H_{\mathbb{C}})=1$ implies that we have the isomorphism 
\begin{equation*}
	\begin{array}{cccc}
		H \backslash G 
		& \xrightarrow{\sim} & X_{\mathbb{C}}'(\mathbb{R}) &\subset X_{\mathbb{C}}' = H_{\mathbb{C}} \backslash G_{\mathbb{C}}
		\\
		\rotatebox{90}{$\in$} && \rotatebox{90}{$\in$} &
		\\
		H g 
		& \mapsto & H_{\mathbb{C}} g &
	\end{array}
\end{equation*}
by the long exact sequence of the Galois cohomology:
    \begin{equation*}
        1 \to H \to G \to ( H_{\C} \backslash G_{\C} )(\R)
        \to H^{1}(\R, H_{\C}) \to H^{1}(\R, G_{\C}),
    \end{equation*}
which is obtained by applying 
\cite[Chap.~I, \S 5.4, Prop.~36]{Serre-Galois-Cohomology} to $H_{\C} \subset G_{\C}$.

Now let us prove $\Xi' \subset \Xi$.  
Pick $\mathcal{O}' \in \Xi'$ and choose 
$x \in \mathcal{O}'(\mathbb{R}) \subset X_{\mathbb{C}}'(\mathbb{R})$.
Then, the above isomorphism implies that
\begin{equation*}
	\text{there exists }\: g \in G
	\quad
	\text{ such that }
	\quad
	\mathcal{O}'(\mathbb{R}) \ni x = H_{\mathbb{C}} g,
\end{equation*}
and thus $\mathcal{O}' = H_{\mathbb{C}} g L_{\mathbb{C}}$.
By definition, the corresponding $H_{\mathbb{C}}$-orbit $\mathcal{O}$ on $X_{\mathbb{C}}$ satisfies $H_{\mathbb{C}} g L_{\mathbb{C}}=\mathcal{O}$.
Therefore, $\mathcal{O}$ contains the point $g L_{\mathbb{C}} \in X_{\mathbb{C}}$, which is an $\mathbb{R}$-rational point because if we write $\gamma\colon G_{\C} \to G_{\C}$ for the action of the nontrivial element of the Galois group $\Gal(\C/\R)$, we have
\begin{equation*}
	\gamma(g L_{\mathbb{C}}) = \gamma(g)L_{\mathbb{C}} = gL_{\mathbb{C}}
\end{equation*}
by $g \in G = G_{\mathbb{C}}^{\gamma}$.
This implies $\mathcal{O} \in \Xi$, hence the lemma.
%%by the definition of $\Xi$, which completes the proof.
\end{proof}

\bibliographystyle{alpha}  %
\bibliography{bib_RDFV.bib}

% \printbibliography

\end{document}